\documentclass[12pt,reqno]{amsart}
\usepackage{xifthen,setspace,multirow}
\usepackage{pkgfile}

\allowdisplaybreaks

\begin{document}

\title[Distribution-Free Graph-Based Two-Sample Tests]{A General Asymptotic Framework for Distribution-Free Graph-Based Two-Sample Tests}

\author[B. B. Bhattacharya]{Bhaswar B. Bhattacharya}
\address{Department of Statistics\\ University of Pennsylvania\\ Philadelphia, PA 19104}
\email{bhaswar@wharton.upenn.edu}

\begin{abstract}Testing equality of two multivariate distributions is a classical problem for which many non-parametric tests have been proposed over the years. Most of the popular two-sample tests, which are asymptotically distribution-free, are based either on geometric  graphs constructed using inter-point distances between the observations (multivariate generalizations of the Wald-Wolfowitz's runs test) or on multivariate data-depth (generalizations of the Mann-Whitney rank test). 

This paper introduces a general notion of distribution-free graph-based two-sample tests, and provides a unified framework for analyzing and comparing their asymptotic properties. The asymptotic (Pitman) efficiency of a general graph-based test is derived, which includes tests based on geometric graphs, such as the Friedman-Rafsky test~\cite{fr}, the test based on the $K$-nearest neighbor  graph, the cross-match test \cite{rosenbaum}, the generalized edge-count test \cite{chenfriedman}, as well as tests based on multivariate depth functions (the Liu-Singh rank sum statistic \cite{liu_singh}). The results show how the combinatorial properties of the underlying graph effect the performance of the associated two-sample test, and can be used to validate and decide which tests to use in practice. Applications of the results are illustrated both on synthetic and real datasets. 
\end{abstract}


\maketitle

\section{Introduction}
\label{graph2tests}

Let $F$ and $G$ be two continuous distribution functions in $\R^d$. Given independent and identically distributed samples 
\begin{align}\label{2}
\sX_{N_1}=\{X_1, X_2, \ldots, X_{N_1}\} \text{ and } \sY_{N_2}=\{Y_1, Y_2, \ldots, Y_{N_2}\}
\end{align} 
from two unknown distributions $F$ and $G$, respectively, the two-sample problem is to  distinguish the hypotheses 
\begin{equation}\label{test}
H_0: F=G, \quad \text{versus} \quad H_1: F\ne G.
\end{equation}
More precisely, $H_0$ is the collection of all distributions of mutually independent i.i.d. observations with sample size $N_1+N_2$ from a distribution in $\R^d$; and $H_1$ is the collection of all distributions of mutually independent i.i.d. observations with sample size $N_1$ from some distribution $F$  in $\R^d$,  and i.i.d. observations with sample size $N_2$ from some other distribution $G \ne F$  in $\R^d$.

There are many multivariate two-sample testing procedures, ranging from tests for parametric hypotheses such as the Hotelling's $T^2$-test, and the likelihood ratio test, to more general non-parametric procedures \cite{aslan_zech,bfranz,bickel,tsp,twosamplelinear,fr,gretton,hall_tajvidi,henzenn,rousson,rosenbaum,schilling,weiss}. In this paper, we consider multivariate two-sample tests, which are {\it asymptotically distribution-free}, that is, tests for which the asymptotic null distribution do not depend on the underlying (unknown) distribution of the data. As a result, these tests can be directly implemented as an asymptotically level $\alpha$ test, making them practically convenient. 

For univariate data, there are several celebrated nonparametric distribution-free tests such as the the Wald-Wolfowitz runs test \cite{ww} and the Mann-Whitney rank test \cite{mann_whitney} (see the textbook \cite{elnp} for more on these tests). Many multivariate generalizations of these tests, which are asymptotically distribution-free, have been proposed. Most of these tests can be broadly classified into two categories:\vspace{-0.05in}
\begin{itemize}
\item[(1)] {\it Tests based on geometric graphs:} These tests are constructed using the inter-point distances between the observations. This includes the test based on the Euclidean minimal spanning tree by Friedman and Rafsky \cite{fr} (generalization the Wald-Wolfowitz runs test \cite{ww} to higher dimensions) and tests based on nearest neighbor graphs \cite{henzenn,schilling}. This also include Rosenbaum's \cite{rosenbaum} test based minimum non-bipartite matching and the test based on the Hamiltonian path by Biswas et al. \cite{tsp} (both of which are distribution-free in finite samples), and the recent tests of Chen and Friedman \cite{chenfriedman}. Refer to Maa et al. \cite{interpoint} for theoretical motivations for using tests based on inter-point distances.

\item[(2)] {\it Tests based on depth functions}: The Liu-Singh rank sum test \cite{liu_singh} are a class of multivariate two-sample tests that generalize the Mann-Whitney rank test using the notion of data-depth. This include tests based on halfspace depth \cite{tukey} and simplicial depth \cite{liu,liu_datadepth}, among others. For other generalizations of the Mann-Whitney test, refer to the survey by Oja \cite{ojaR} and the references therein. 
 \end{itemize}

In this paper, we provide a general framework of graph-based two-sample tests, which includes all the tests discussed above. We begin with a few definitions: A subset $S\subset \R^d$ is {\it locally finite} if $S\cap C$ is finite, for all compact subsets $C\subset \R^d$.  A locally finite set $S\subset \R^d$ is {\it nice} (with respect to a metric $\rho$ in $\R^d$) if all inter-point distances among the elements of $S$ are distinct. Note that if, for example, $S$ is a set of $N$ i.i.d. points $W_1, W_2, \ldots, W_N$ from some continuous distribution $F$, then the distribution of $W_1-W_2$, and hence $||W_1-W_2||$, where $||\cdot||$ denotes the Euclidean norm, does not have any point mass, and $S$ is nice. For the same reason,  in a set of $N$ i.i.d. points from a continuous distribution ties occur with zero probability. 

A  {\it graph functional} $\sG$ in $\R^d$ defines a graph for all finite subsets of $\R^d$, that is, given  $S\subset \R^d$ finite, $\sG(S)$ is a graph with vertex set $S$. A graph functional is said to be {\it undirected/directed} if the graph $\sG(S)$ is an undirected/directed graph with vertex set $S$. We assume that $\sG(S)$ has no self loops and multiple edges, that is, no edge is repeated more than once in the undirected case, and no edge in the same direction is repeated more than once in the directed case.  The set of edges in the graph $\sG(S)$ will be denoted by $E(\sG(S))$. The cardinality of a finite set  $A$, is denoted by $|A|$.

\begin{defn}\label{graph2} Let $\sX_{N_1}$ and $\sY_{N_2}$ be i.i.d. samples of size $N_1$ and $N_2$ from densities $f$ and $g$, respectively, as in (\ref{2}). The {\it 2-sample test statistic based on the graph functional $\sG$} is defined as 
\begin{align}\label{graph2test}
T(\sG(\sX_{N_1}\cup\sY_{N_2}))=\frac{\sum_{i=1}^{N_1} \sum_{j=1}^{N_2} \pmb 1\{(X_i, Y_j)\in E(\sG(\sX_{N_1}\cup \sY_{N_2}))\}}{|E(\sG(\sX_{N_1}\cup \sY_{N_2}))|}.
\end{align}
\end{defn}

Denote by $N=N_1+N_2$ and $\cZ_N=\sX_{N_1}\cup \sY_{N_2}=\{Z_1, Z_2, \ldots, Z_N\}$ the elements of the pooled sample, with $Z_i$ labelled $c_i=1$ if  $Z_i\in \sX_{N_1}$ and  $c_i=2$ if $z_i\in \sY_{N_2}$. Then~\eqref{graph2test} can be re-written as 
\begin{align}\label{graph2z}
T(\sG(\cZ_N)):=\frac{\sum_{1\leq i\ne j\leq N} \psi(c_i, c_j) \pmb 1\{(Z_i, Z_j)\in E(\sG(\cZ_N))\}}{|E(\sG(\cZ_N))|}, 
\end{align}
where $\psi(c_i, c_j)=\pmb1\{c_i=1, c_j=2\}$.  If $\sG$ is an undirected graph functional, then the statistic \eqref{graph2z} counts the proportion of edges in the graph $\sG(\cZ_N)$ with one end point in $\sX_{N_1}$ and the other end point in $\sY_{N_2}$. If $\sG$ is a directed graph functional, then \eqref{graph2z} is the proportion of directed edges with the outward end in $\sX_{N_1}$ and the inward end in $\sY_{N_2}$. By conditioning on the graph $\sG(\cZ_N)$, it is easy to see that under the null $H_0$, $\E(T(\cZ_N)))$ is $\frac{2N_1N_2}{N(N-1)}$ or $\frac{N_1N_2}{N(N-1)}$, depending on whether the graph functional is undirected or directed. 

In this paper, the graph functionals are computed based on the Euclidean distance in $\R^d$, and the rejection region of the statistic~\eqref{graph2test} will be based on its asymptotic null distribution in the usual limiting regime $N\rightarrow \infty$, with 
\begin{align}\label{pq}
\frac{N_1}{N_1+N_2}\rightarrow p\in (0, 1), \quad \frac{N_2}{N_1+N_2}\rightarrow q:=1-p.
\end{align} 
The test statistics considered in this paper will have $N^{\frac{1}{2}}$-fluctuations under $H_0$. Thus, depending on the type of alternative, the test based on~\eqref{graph2test} will reject $H_0$ for large and/or small values of the standardized statistic 
\begin{align}
\cR(\sG(\cZ_N))=&\sqrt N\left\{T(\sG(\cZ_N))-\E(T(\sG(\cZ_N))) \right\}. 
\label{2sampleG1}
\end{align}

\subsection{Two-Sample Tests Based on Geometric Graphs}
\label{sg}

Many popular multivariate two-sample test statistics are of the form~\eqref{graph2test} where the graph functional $\sG$ is constructed using the inter-point distances of the pooled sample. 

\subsubsection{Wald-Wolfowitz (WW) Runs Test} This is one of the earliest known non-parametric tests for the equality of two univariate distributions \cite{ww}:  Let $\sX_{N_1}$ and $\sY_{N_2}$ be i.i.d. samples of size $N_1$ and $N_2$ as in~\eqref{2}. A {\it run} in the pooled sample $\cZ_N=\sX_{N_1}\cup \sY_{N_2}$ is a maximal non-empty segment of adjacent elements with the same label when the elements in $\cZ_N$ are arranged in increasing order. If the two distributions are different, the elements with labels 1 and 2 would be clumped together, and the total number of runs $C(\cZ_N)$ in $\cZ_N$ will be small. On the other hand, for distributions which are equal/close, the different labels are jumbled up and $C(\cZ_N)$ will be large. Thus, the WW-test rejects $H_0$ for small values of $C(\cZ_N)$. 

Note that the number of runs in $\cZ_N$ minus 1 equals  the number of times the sample label changes as one moves along $\cZ_N$ in increasing order. This implies that the WW-test is a graph-based test~\eqref{graph2test}: $\frac{C(\sX_{N_1}\cup \sY_{N_2})-1}{N-1}=T(\cP(\sX_{N_1}\cup \sY_{N_2}))$, where $\cP(S)$ is the path with $|S|-1$ edges through the elements of $S$ arranged in increasing order, for a finite set $S\subset \R$. 

Let $\cR(\cP(\cZ_N))$  be the standardized version of $T(\cP(Z_N))$ as in~\eqref{2sampleG1}. Wald and Wolfowitz \cite{ww} proved that $\cR(\cP(\cZ_N))$ is distribution-free in finite samples, is asymptotically normal under $H_0$,  and consistent under all fixed alternatives. The WW-test often has low power in practice and has zero asymptotic efficiency, that is, it is powerless against $O(N^{-\frac{1}{2}})$ alternatives \cite{moodefficiency}.

\subsubsection{Friedman-Rafsky (FR) Test}

Friedman and Rafsky \cite{fr} generalized the Wald and Wolfowitz runs test to higher dimensions by using the Euclidean minimal spanning tree of the pooled sample. 

\begin{defn}\label{mst}
Given a nice finite set $S \subset \R^d$, a {\it spanning tree} of $S$ is a connected graph $\cT$ with vertex-set $S$ and no cycles. The {\it length} $w(\cT)$ of $\cT$ is the sum of the Euclidean lengths of the edges of $\cT$. A {\it minimum spanning tree} (MST) of $S$, denoted by $\cT(S)$, is a spanning tree with the smallest length, that is, $w(\cT(S))\leq w(\cT)$ for all spanning trees $\cT$ of $S$.
\end{defn}

Thus, $\cT$ defines a graph functional in $\R^d$, and given $\sX_{N_1}$ and $\sY_{N_2}$ as in~\eqref{2}, the FR-test rejects $H_0$ for small values of 
\begin{align}
T(\cT(\sX_{N_1}\cup \sY_{N_2}))=\frac{\sum_{i=1}^{N_1} \sum_{j=1}^{N_2} \pmb 1\{(X_i, Y_j)\in E(\cT(\sX_{N_1}\cup \sY_{N_2}))\}}{N-1}.
\label{Tfr}
\end{align} This is precisely the WW-test in $d=1$, and is motivated by the same intuition that when the two distributions are different, the number of edges across labels 1 and 2 is small. 

Friedman and Rafsky \cite{fr} calibrated~\eqref{Tfr} as a permutation test, and showed that it has good power in practice for multivariate data. Later, Henze and Penrose \cite{henzepenrose} proved that $\cR(\cT(\cZ_N))$ is asymptotically normal under $H_0$  and is consistent under all fixed alternatives. Recently, Chen and Zhang \cite{czchangepoint} used the FR and related graph-based tests in change-point detection problems, and suggested new modifications of the FR-test for high-dimensional and object data \cite{chenfriedman,hcII}.

\subsubsection{Test Based on $K$-Nearest Neighbor ($K$-NN) Graphs} As in~\eqref{Tfr}, a multivariate two-sample test can be constructed using the $K$-nearest neighbor  graph of $\cZ_N$. This was originally suggested by Friedman and Rafsky \cite{fr} and later studied by Schilling \cite{schilling} and Henze \cite{henzenn}. 

\begin{defn}\label{knn}
Given a nice finite set $S \subset \R^d$, the (undirected) $K$-nearest neighbor  graph ($K$-NN)  is a graph with vertex set $S$ with an edge $(a, b)$, for $a, b \in S$, if the Euclidean distance between $a$ and $b$ is among the $K$-th smallest distances from $a$ to any other point in $S$ and/or among the $K$-th smallest distances from $b$ to any other point in $S$. Denote the undirected $K$-NN of $S$ by $\cN_K(S)$.  
\end{defn}

Given $\sX_{N_1}$ and $\sY_{N_2}$ as in~\eqref{2}, the $K$-NN statistic is 
\begin{align}\label{Tknn}
T(\cN_K(\sX_{N_1}\cup \sY_{N_2}))=\frac{\sum_{i=1}^{N_1} \sum_{j=1}^{N_2} \pmb 1\{(X_i, Y_j)\in E(\cN_K(\sX_{N_1}\cup \sY_{N_2}))\}}{|E(\cN_K(\sX_{N_1}\cup \sY_{N_2}))|}. 
\end{align} 
As before, when the two distributions are different, the number of edges across the two samples is small (see Figure~\ref{3nn}), so the $K$-NN test rejects $H_0$ for small values of~\eqref{Tknn}. Schilling \cite{schilling} considered the case where $K$ remains fixed with $N$, and showed that the test based on $K$ nearest neighbors is asymptotically normal under $H_0$  and consistent against fixed alternatives.\footnote{The statistic~(\ref{Tknn})  is slightly different from the test used by Schilling \cite[Section 2]{schilling}, which can also be re-written as graph-based test~(\ref{graph2test}) by allowing multiple edges in the $K$-NN graph.}  However, the test statistic \eqref{Tknn} makes sense even when $K=K_N \rightarrow \infty$, which we consider in Section \ref{sec:knn}.  

\begin{figure*}[h]
\centering
\begin{minipage}[l]{0.495\textwidth}
\centering
\includegraphics[width=2.5in]
    {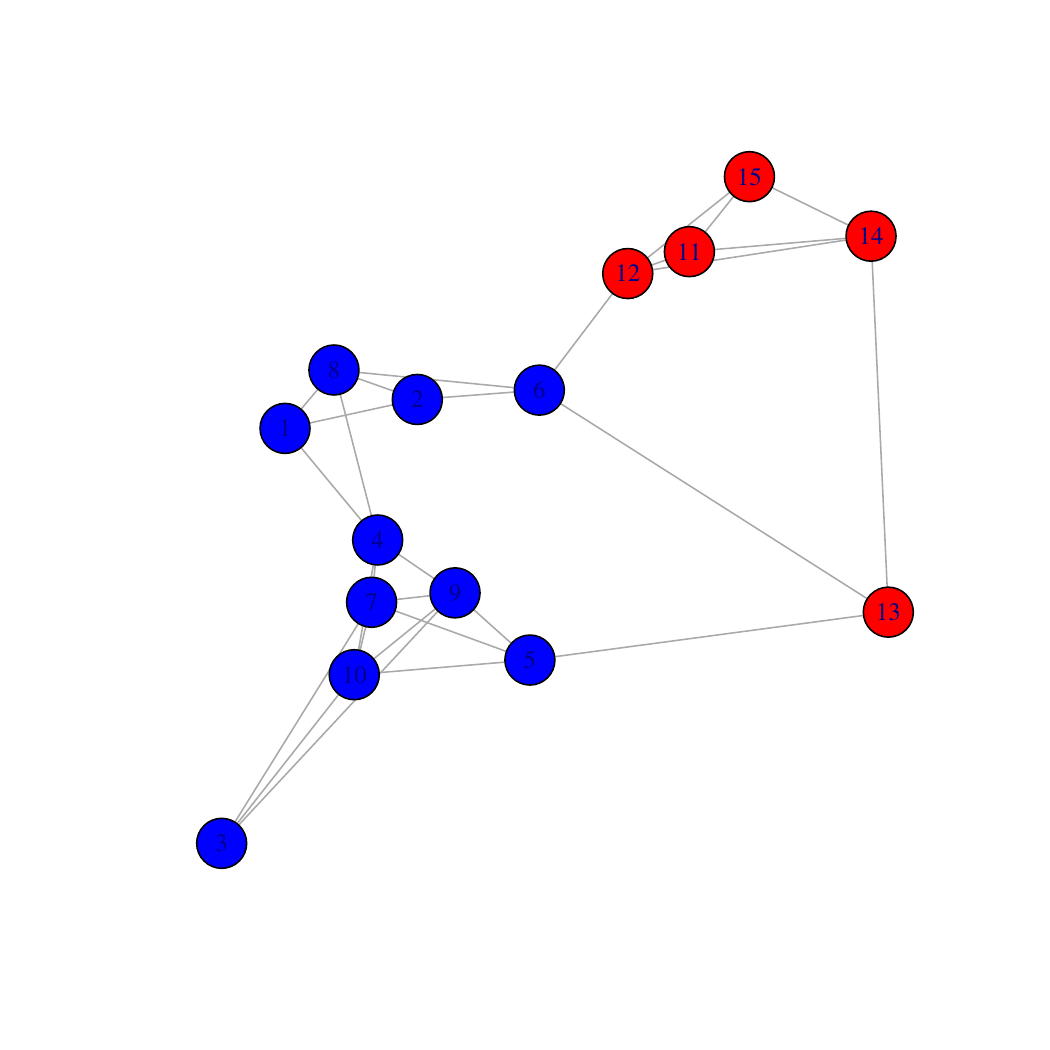}\\
\scriptsize{(a)}
\end{minipage}
\begin{minipage}[c]{0.495\textwidth}
\centering
\includegraphics[width=2.5in]
    {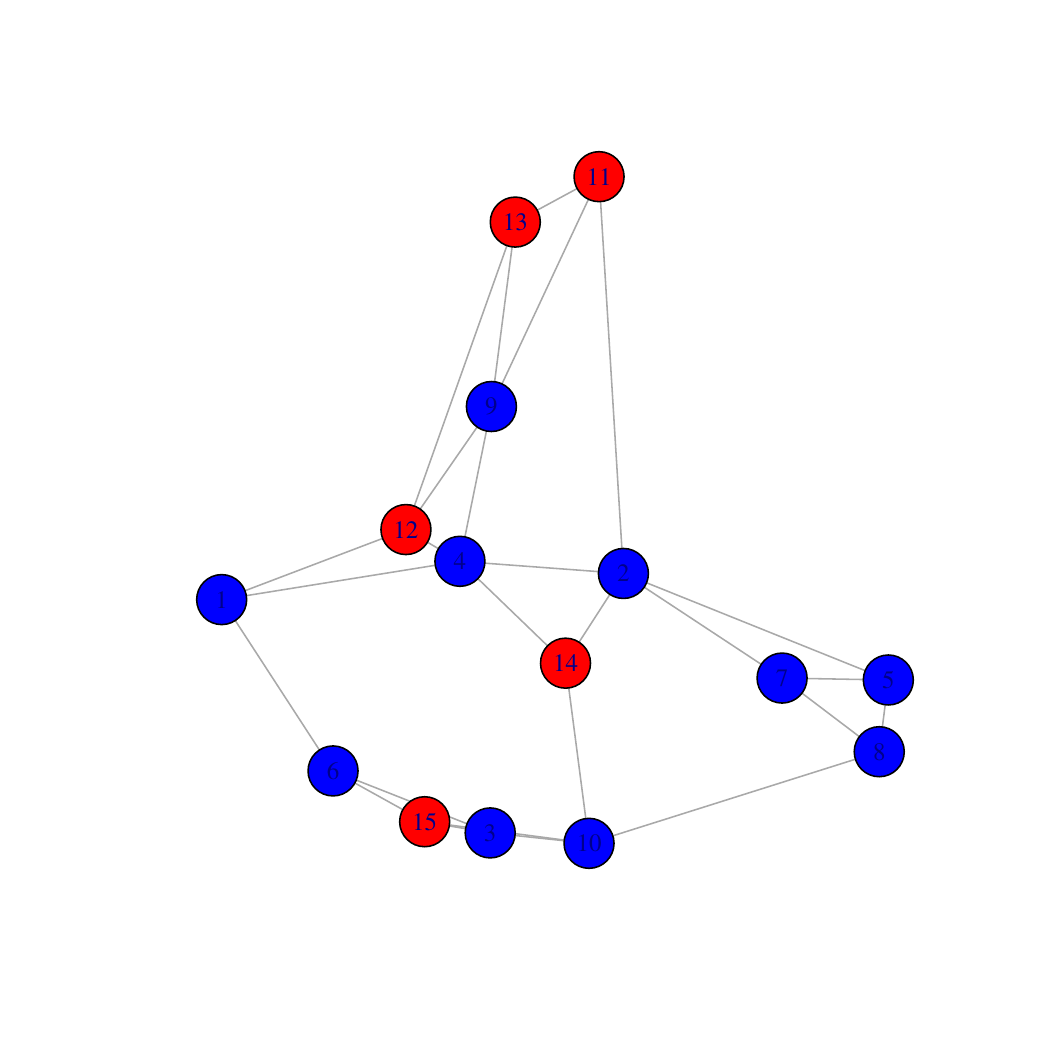}\\
\scriptsize{(b)}
\end{minipage}
\caption{\small{The 3-NN graph on a pooled sample of size 15 in $\R^2$ with 10 i.i.d. points from $N(0, \mathrm I)$ (colored blue) and 5 i.i.d. points from $N(\Delta\cdot \pmb 1, \mathrm I)$ (colored red). For (a) $\Delta=2$, there are 2 edges with endpoints in different samples, and  for (b) $\Delta=0.05$, there are 10 edges with endpoints in different samples. }}
\label{3nn}
\end{figure*}

\subsubsection{Cross-Match (CM) Test}

Rosenbaum \cite{rosenbaum} proposed a distribution-free multivariate two-sample test based on minimum non-bipartite matching. For simplicity, assume that the total number of samples $N$ is even; otherwise, add or delete a sample point to make it even.

\begin{defn}\label{mdm}
Given a finite $S\subset \R^d$ and a symmetric distance matrix $D:=((d(a, b)))_{a\ne b \in S}$, a {\it non-bipartite matching} of $S$ is a pairing of the elements $S$ into $N/2$ non-overlapping pairs, that is, a partition of $S=\bigcup_{i=1}^{N/2}S_i$, where $|S_i|=2$ and $S_i\cap S_j=\emptyset$. The weight of a matching is the sum of the distances between the $N/2$ matched pairs. A {\it minimum non-bipartite matching} (NBM) of $S$ is a matching which has the minimum weight over all matchings of $S$.
\end{defn}

The NBM defines a graph functional $\cW$ as follows: for every finite $S\subset \R^d$, $\cW(S)$ is the graph with vertex set $S$ and an edge $(a, b)$ whenever there exists $i \in [N/2]$ such that $S_i=\{a, b\}$. Note that $\cW(S)$ is a graph with $N/2$ pairwise disjoint edges. Given $\sX_{N_1}$ and $\sY_{N_2}$ as in~\eqref{2}, the CM-test rejects $H_0$ for small values of 
\begin{align}
T(\cW(\sX_{N_1}\cup \sY_{N_2}))=\frac{\sum_{i=1}^{N_1} \sum_{j=1}^{N_2} \pmb 1\{(X_i, Y_j)\in E(\cW(\sX_{N_1}\cup \sY_{N_2}))\}}{N/2}.
\label{Tmdm}
\end{align} 
Like the WW-test, but unlike the FR and the $K$-NN tests, the CM-test is distribution-free in finite samples under $H_0$. Rosenbaum \cite{rosenbaum} implemented this as a permutation test and derived the asymptotic normal distribution under the null $H_0$.

\subsection{Two-Sample Tests Based on Data-Depth}
\label{s2depth}

Many non-parametric two-sample tests are based on depth functions, which are multivariate generalizations of ranks \cite{liu_datadepth,ojaR}. Given a distribution function $F$ in $\R^d$, a {\it depth function} $D(\cdot, F): \R^d\rightarrow [0, 1]$ is a function that provides a ranking of points in $\R^d$. High depth corresponds to {\it centrality}, while low depth corresponds to {\it outlyingness}. The {\it center} consists of the points that globally maximize the depth, and is often considered as a multivariate median of $F$. For $X\sim F$ and $Y \sim G$, two independent random variables in $\R^d$, 
\begin{align}\label{outlyingness}
R_D(y, F) = \P(X: D(X, F) \leq  D(y, F)) 
\end{align}
is a measure of the {\it relative outlyingness} of a point $y\in \R^d$ with respect to $F$. In other words, $R_D(y, F)$ is the fraction of the $F$ population with depth at most as that of the point $y$. The dependence on $D$ in the notation $R_D(\cdot, F)$ will be dropped when it is clear from the context. The {\it quality index}
\begin{align}\label{eq:QFG}
Q(F, G):=\int R(y, F)\mathrm dG(y)=&\P(D(X, F)\leq D(Y, F)| X\sim F, Y\sim G)
\end{align}
is the average fraction of $F$ with depth at most as the point $y$, averaged over $y$ distributed as $G$. When $F=G$ and 
$D(X, F)$ has a continuous distribution, then $R(Y, F)\sim \dU[0, 1]$ and $Q(F, G)=1/2$ \cite[Proposition 3.1]{liu_singh}. 

\begin{defn}Let $\sX_{N_1}$ and $\sY_{N_2}$ be as in~\eqref{2} and $F_{N_1}$ and $G_{N_2}$ be the empirical distribution functions, respectively. The {\it Liu-Singh rank sum statistic} \cite{liu_singh}, is
\begin{equation}\label{QFGsample}
Q(F_{N_1}, G_{N_2}):=\int R(y, F_{N_1})\mathrm dG_{N_2}(y)=\frac{1}{N_2}\sum_{j=1}^{N_2} R(Y_j, F_{N_1}),
\end{equation}
the sample estimator of $Q(F, G)$. 
\end{defn}

The test rejects $H_0$ for small/large values of $\sqrt N\left(Q(F, G)-\frac{1}{2}\right)$ and can be re-written as a graph-based test~\eqref{graph2test}. To see this, note that
\begin{align}\label{Qdepth}
N_1N_2 Q(F_{N_1}, G_{N_2})
=&\sum_{j=1}^{N_2} \sum_{i=1}^{N_1}\pmb 1\{D(X_i, F_{N_1})\leq D(Y_j, F_{N_1})\}.
\end{align}
Let $\cZ_N$ be the pooled sample with the labeling as in~\eqref{graph2z}. Construct a graph $\sG_D(\cZ_N)$ with vertices $\cZ_N$ with a directed edge from $(Z_i, Z_j)$ whenever $D(Z_i, F_{N_1})\leq D(Z_j, F_{N_1})$. Note that $\sG_D(\cZ_N)$ is a complete graph with directions on the edges depending on the relative order of the depth of the two endpoints. Then from~\eqref{Qdepth}, 
\begin{eqnarray*}
N_1N_2 Q(F_{N_1}, G_{N_2})&=&\sum_{1\leq i \ne j \leq N} \psi(c_i, c_j) \pmb 1\{D(Z_i, F_{N_1})\leq D(Z_j, F_{N_1})\}\nonumber\\
&=&\sum_{1\leq i \ne j \leq N}  \psi(c_i, c_j) \pmb 1\{ (Z_i, Z_j)\in E(\sG_D(\cZ_N))\},
\end{eqnarray*}
with $\psi(\cdot, \cdot)$ as in~\eqref{graph2z}. Since $|E(\sG_D(\cZ_N))|=\frac{N(N-1)}{2}$, this implies, $Q(F_{N_1}, G_{N_2})=\frac{N(N-1)}{2 N_1N_2}T(\sG_D(\cZ_N))$,
where $T$ is defined in~\eqref{graph2z}. 

Thus, the Liu-Singh rank sum statistic based on a depth function $D$ is a graph-based test~\eqref{graph2test}. Common depth functions include the Mahalanobis depth, the halfspace depth, the simplicial depth, and the projection depth, among others. In the following, we recall the definitions for a few of these (see \cite{liu_datadepth} for more on depth functions).

\begin{enumerate}[1.] 

\item {\it Mann-Whitney Test}: This is one of the oldest non-parametric two-sample tests for univariate data \cite{elnp,vandervaart}. Depending on the alternative, the Mann-Whitney rank test rejects $H_0$ for small or large values of $\sum_{i=1}^{N_1}\sum_{j=1}^{N_2}\pmb 1\{X_i < Y_j\}$. This is a distribution-free test, which corresponds to taking $D(x, F)=F(x)$ in~(\ref{Qdepth}).

\item {\it Halfspace Depth}: Tukey \cite{tukey} suggested the following depth-function 
\begin{align}\label{halfspace}
HD(x, F) = \inf \int_{H_x} \mathrm d F,
\end{align}
where the infimum is taken over all closed halfspace $H_x$ with $x$ on its boundary. The two-sample test based on the halfspace depth is obtained by using~(\ref{halfspace}) in the statistic~(\ref{Qdepth}). 

\item {\it Mahalanobis Depth:} Given a distribution function $F$, the Mahalanobis depth of a point $x\in \R^d$ is 
\begin{equation}\label{mahalanobis}
MD(x, F) = \frac{1}{1+ (x -\mu (F))'\Sigma^{-1}(F)(x -\mu (F))}
\end{equation}
where $\mu(F)=\int x \mathrm dF(x)$ and $\Sigma(F)=\int (x-\mu(F))(x-\mu(F))' \mathrm dF(x)$ are the mean and 
covariance of the distribution $F$. 
\end{enumerate}

The Liu-Singh rank sum statistic is consistent against alternatives for which the quality index $Q(F, G)\ne \frac{1}{2}$ (recall \eqref{outlyingness}), under mild conditions \cite{liu_singh}. A special version of Liu-Singh rank sum statistic, with a reference sample, inherits the distribution-free property of the Mann-Whitney test in finite samples \cite[Section 4]{liu_singh}. However, the Liu-Singh rank sum statistic defined above \eqref{QFGsample} is only asymptotically distribution-free under the null hypothesis \cite{liu_singh,zuolimitdist}.  Zuo and He \cite[Theorem 1]{zuolimitdist} proved the asymptotic normality of the Liu-Singh rank sum statistic under general alternatives for depth functions satisfying certain regularity conditions.  

\subsection{Properties of Graph-Based Tests}
\label{sec:properties}

A test function $\phi_N$ for the testing problem \eqref{2} is said to be {\it asymptotically exact level} $\alpha$, if $\lim_{N\rightarrow \infty}\E_{H_0}(\phi_N)=\alpha$. An exact level $\alpha$ test function $\phi_N$ is said to be consistent against the alternative $H_1$, if $\lim_{N\rightarrow \infty}\E_{H_1}(\phi_N)=1$, that is, the power of the test $\phi_N$ converges to 1 in the usual asymptotic regime~\eqref{pq}.

To describe the asymptotic properties of the tests above, we assume that the distributions
$F$ and $G$ have densities $f$ and $g$ with respect to Lebesgue measure, respectively; and under the alternative $H_1$, $f$ and $g$ differ on a set of positive Lebesgue measure. Under this assumption, all the tests considered in Sections~\ref{sg} and~\ref{s2depth} have the following common properties: 

\begin{itemize}

\item {\it Asymptotically Distribution-Free}: The standardized test statistic \eqref{2sampleG1} converges to $N(0, \sigma^2_1)$ under $H_0$, where the limiting variance $\sigma^2_1$  depends on the graph functional $\sG$, but not on the unknown distributions. Therefore, under the null, the asymptotic distribution of the test statistic  \eqref{2sampleG1} is {\it distribution-free}. This was proved for the FR-test by Henze and Penrose \cite{henzepenrose}, for the test based on the $K$-NN graph by Henze \cite{henzenn}, and by Liu and Singh \cite{liu_singh} for depth-based tests. Finally, recall that the CM  test is exactly distribution-free in finite samples \cite{rosenbaum}. 

\item {\it Consistent Against Fixed Alternatives}: For a geometric graph functional $\sG$ in $\R^d$, the test function with rejection region
\begin{align}\label{rejregion}
\{\cR(\sG(\cZ_N))<\sigma_1z_{\alpha}\},
\end{align}
where $z_\alpha$ is the $\alpha$-th quantile of the standard normal, is asymptotically exact level (size) $\alpha$ (see \cite{lr} for definitions of level and size of a test) and consistent against all fixed alternatives, for the testing problem \eqref{2}.  This was proved for the MST by Henze and Penrose \cite{py} and for the $K$-NN graph (when $K=O(1)$) by Schilling \cite{schilling} and later by Henze \cite{henzenn}. Recently, Arias-Castro and Pelletier~\cite{cm_consistency} showed that the CM test is consistent against general alternatives. Similarly, the Liu-Singh rank sum statistic\footnote{For depth-based tests, the rejection region can be of the form $\{|\cR(\sG(\cZ_N))|\geq \sigma_1 z_{1-\alpha/2}\}$, depending on the alternative.} is asymptotically size $\alpha$ and consistent against alternatives for which the quality index $Q(F, G)\ne \frac{1}{2}$.

\end{itemize}

Even though the basic asymptotic properties of the tests based on geometric graphs and those based on data-depth are quite similar, their algorithmic complexities are very different. 

\begin{itemize}

\item Tests based on geometric graphs such as the MST, the $K$-NN graph, and the CM-test can be computed in polynomial time with respect to both the number of data points $N$ and the dimension $d$. For instance, the MST and the $K$-NN graph of a set of $N$ points in $\R^d$ can be computed easily in $O(dN^2)$ time, and the NBM matching can be computed in $O(dN^3)$ time \cite{comboptbook}.

\item On the other hand, depth-based tests are generally difficult to compute when the dimension is large because the computation time is polynomial in $N$ but exponential in $d$. In fact, computing the Tukey depth of a set of $N$ points in $\R^d$ is  NP-hard  if  both $N$ and $d$ are  parts  of  the  input  \cite{jp},  and it  is  even  hard  to  approximate  \cite{akhd}. For more details on algorithms to compute various depth functions, refer to the recent survey of Rousseeuw and Hubert \cite{depthsurvey} and the references therein. 

\end{itemize}

\subsection{Summary of Results}

The general notion of graph-based two-sample tests~\eqref{graph2test} introduced above provides a unified framework for analyzing and  {\it statistically} comparing their asymptotic properties. We derive the limiting power of these tests against {\it local alternatives}, that is, alternatives which shrink towards the null as the sample size grows to infinity. If the statistic \eqref{2sampleG1} is asymptotically normal under the null and the test based on \eqref{rejregion} is consistent, it can be expected to have non-trivial  power (greater than the level of the test) against $O(N^{-\frac{1}{2}})$ local alternatives. This is formalized using the notion of {\it Pitman efficiency} of a test (see \eqref{effdefn} below), and can be used to compare the asymptotic performances of different tests. The following is a summary of the results obtained: 

\begin{enumerate}[1.]

\item In Section~\ref{ae} the asymptotic (Pitman) efficiency of a general graph-based test is derived. 
This result can be used as a black-box to derive the efficiency of any such two-sample test (Theorem~\ref{ASYME} and Corollary~\ref{ASYMEundirected}). The results obtained show how combinatorial properties of the underlying graph effect the performance of the associated test, which can be effectively used to construct and analyze new tests. The results are illustrated through simulations and compared with other parametric tests in Section \ref{examples}. 

\item As a consequence of the general result the asymptotic efficiency of the tests described in  Sections~\ref{sg} and~\ref{s2depth} can be derived: 

\begin{itemize}

\item[$\bullet$]  It is shown that the Friedman-Rafsy test has zero asymptotic efficiency, that is, it is powerless against any $O(N^{-\frac{1}{2}})$ local alternatives (Theorem~\ref{MSTAE}). In fact, this phenomenon extends to a large class of random geometric graphs that exhibit local spatial dependence. This can be formalized using the notion of {\it stabilization} of geometric graphs \cite{py}, which includes the MST, the $K$-NN graph (where $K=O(1)$ is fixed) among others (Theorem~\ref{POWERSTABILIZE}). 

\item[$\bullet$] Our general theorem can be used to compute the asymptotic efficiency of the $K$-NN based test, when $K=K_N \rightarrow \infty$, as well (Proposition \ref{ppn:knn}). Here, the Pitman efficiency can be non-zero, when $K$ grows with $N$ sufficiently fast. This is reinforced in simulations, which combined with the computational efficiency of the $K$-NN test (running time is polynomial in $K$, the sample size $N$, and the dimension $d$), makes this test particularly attractive, both theoretically and in applications.

\item [$\bullet$] The Pitman efficiency of tests based on data-depth~\eqref{QFGsample} is computed in Section \ref{depthtest}. These tests have non-trivial local power, and hence, non-zero asymptotic efficiencies, for many $O(N^{-\frac{1}{2}})$ alternatives (Theorem~\ref{POWERSTABILIZE}). However, as mentioned earlier, these tests become computationally expensive as dimension increases.

\item [$\bullet$]  Recently, Chen and Friedman \cite{chenfriedman} proposed a modification of the test statistic \eqref{2sampleG1}, which is especially powerful when sample size is small and the dimension is large.  Our general framework can be modified to include these tests as well, and derive their limiting power against local alternatives (Theorem \ref{thm:frnew}). 
 
\end{itemize}

\item Finally, the performance of the Friedman-Rafsy test and the test based on the halfspace depth are compared on the sensorless drive diagnosis data set \cite{sensordriving} (Section \ref{drivediag}).\footnote{All codes used in the paper can be downloaded from \url{http://www-stat.wharton.upenn.edu/~bhaswar/Graph_Based_Two_Sample_Codes.zip}.} 
\end{enumerate}

In Section \ref{sec:hd} we study the performance of these tests when sample sizes are small and the dimension is comparable to the sample size. In this case, classical parametric tests, which involves computation of the sample covariance matrix, often breaks down; and the non-parametric tests start to dominate \cite{tsp,chenfriedman}. This is validated by the simulations in Section \ref{sec:hd}, where the tests based on geometric graphs, such the FR and the CF tests, outperform the other tests, in the high-dimensional regime. We conclude with a discussion about the performances of the different tests, and which tests to use in practice (Section \ref{sec:extgen}).

\section{Preliminaries}
\label{preliminaries}

Begin by recalling some definitions and notation: For $x \in \R$, $x_+=\max\{x, 0\}$, and for a vector $z\in \R^s$, $||z||=(\sum_{i=1}^s z_i^2)^{\frac{1}{2}}$, is the Euclidean norm of $z$.

To quantify the notion of local alternatives, let $\Theta\subseteq \R^p$ and $\{\P_\theta\}_{\theta\in \Theta}$ be a parametric family of distributions in $\R^d$ with density $f(\cdot|\theta)$, with respect to Lebesgue measure, indexed by a $p$-dimensional parameter $\theta \in \Theta$. Throughout, we will assume that the distributions in $\{\P_\theta\}_{\theta\in \Theta}$ have a common support $\cK \subseteq \R^d$, which does not depend on $\theta$. To compute asymptotic efficiency of tests, certain smoothness conditions are required on $f(\cdot|\theta)$. The standard technical condition is to assume that the family  $\{\P_\theta\}_{\theta\in \Theta}$ is {\it quadratic mean differentiable} (QMD)  (see \cite[Definition 12.2.1]{lr} for details). The QMD assumption implies differentiability in norm of the square root of the density, which holds for most standard families of distributions, including exponential families in natural form.

\begin{assumption}
\label{dist4}
A parametric family of distributions $\{\P_{\theta}\}_{\theta\in \Theta}$, with $\Theta\subseteq \R^p$, is said to satisfy Assumption~\ref{dist4} at $\theta=\theta_1$, if  $\P_{\theta}$ is QMD at $\theta=\theta_1$, and for $V_1\sim \P_{\theta_1}$ and all $h \in \R^p$,  $\E(|\langle h, \eta(V_1, \theta_1)\rangle |^4)<\infty$, where $\eta(\cdot, \theta):=\frac{\grad_\theta f(\cdot|\theta)}{f(\cdot|\theta)}$  is the {\it score function}. 
\end{assumption}

Let $\sX_{N_1}$ and $\sY_{N_2}$ be samples from $\P_{\theta_1}$ and $\P_{\theta_2}$ as in~\eqref{2}, respectively. For $h \in \R^p$, consider the testing problem,
\begin{equation}\label{localpower2}
H_0: \theta_2-\theta_1=0, \quad \text{versus} \quad H_1: \theta_2-\theta_1=\frac{h}{\sqrt N}.
\end{equation}
Note that the tests are still carried out in the non-parametric setup assuming no knowledge of the distributions of the two samples. However, the efficiency is computed assuming a parametric form for the unknown distributions.\footnote{Another choice of alternatives for computing the efficiency of  a non-parametric test is to consider: $H_0: g=f, \quad \text{versus} \quad H_1:g=(1-\frac{\delta}{\sqrt N})f+\frac{\delta}{N} g'$, for some density $g'$ in $\R^d$ and $\delta>0$. Then, under mild integrability assumptions, the densities associated with the alternative are contiguous, and the efficiency results derived in this paper easily extend to this case, as well. We chose the formulation in \eqref{localpower2} because it yields slightly cleaner asymptotic expansions and formulas.}

The {\it local asymptotic power} of the test with rejection region~\eqref{rejregion} for the testing problem \eqref{localpower2} is  $\lim_{N \rightarrow \infty}\P_{\theta_1+\frac{h}{\sqrt N}}(\cR(\sG(\cZ_N))<z_{\alpha})$. If $\cR(\sG(\cZ_N))$ is asymptotically normal under $H_0$, the local asymptotic power can be often written as
\begin{align}\label{effdefn}
\lim_{N \rightarrow \infty}\P_{\theta_1+\frac{h}{\sqrt N}}(\cR(\sG(\cZ_N))<z_{\alpha})=\Phi\left(z_{\alpha}-\gamma \right). 
\end{align} 
The quantity $\gamma$ (which depends on $\theta_1$, $h$, and the graph functional $\sG$) in the limit above is then referred to as {\it the asymptotic (Pitman) efficiency} of the test statistic $\cR(\sG(\cZ_N))$ for the testing problem~(\ref{localpower2}), and will be denoted by $\mathrm{AE}(\sG)$. The ratio of the asymptotic efficiencies of two test statistics is the {\it relative (Pitman) efficiency}, which is the ratio of the number of samples required to achieve the same limiting power between two size $\alpha$-tests. Refer to the textbooks \cite{elnp,lr,vandervaart} for more on contiguity, local power and asymptotic efficiencies of tests. \vspace{-0.05in}

\section{Asymptotic Efficiency of Graph-Based Two-Sample Tests}
\label{ae}

This section describes the main results about the asymptotic efficiencies of general graph-based two-sample tests. There are two cases, depending on whether the graph functional is directed or undirected. To begin, let $\sG$ be a directed graph functional in $\R^d$. Denote by $E(\sG(S))$ the set of edges in $\sG(S)$, and by $E^+(\sG(S))$ the set of pairs of vertices with edges in both directions (that is, the set of ordered pairs of vertices $(x, y)$ such that both $(x, y)\in E(\sG(S))$ and $(y, x)\in E(\sG(S))$), respectively. 

For $x\in \R^d$, let $d^\uparrow(x, \sG(S))$ be the out-degree of the vertex $x$ in the graph $\sG(S\cup \{x\})$, that is, the number of outgoing edges $(x, y)$, where  $ y \in S\cup \{x\}$, in the graph $\sG(S\cup \{x\})$. Similarly, let $d^\downarrow(x, \sG(S))$ be the in-degree of the vertex $x$ in the graph $\sG(S\cup \{x\})$, that is, the number of incoming edges $(y, x)$, where  $y \in S\cup \{x\}$, in the graph $\sG(S\cup \{x\})$. Moreover, let $d(x, \sG(S))=d^\downarrow(x, \sG(S))+d^\uparrow(x, \sG(S))$ be the total degree of the vertex $x$ in the graph $\sG(S\cup \{x\})$. Define the {\it scaled in-degree} and the {\it scaled out-degree} of a vertex as follows: 
\begin{equation}
\lambda^\downarrow(x, \sG(S))=\frac{Nd^\downarrow(x, \sG(S^x))}{|E(\sG(S^x))|}, \quad \lambda^\uparrow(x, \sG(S))=\frac{Nd^\uparrow(x, \sG(S^x))}{|E(\sG(S^x))|},
\label{lambdadefn}
\end{equation}
where $S^x=S\cup\{x\}$. Also, let 
\begin{equation}
T_2^{\uparrow}(\sG(S))=\sum_{x\in S} {d^\uparrow (x, \sG(S))\choose 2}, \quad T_2^{\downarrow}(\sG(S))=\sum_{x\in S} {d^\downarrow (x, \sG(S))\choose 2}
\label{2star}
\end{equation}
be the number of outward 2-stars and inward 2-stars in $\sG(S)$, respectively. Finally, let $T_2^{+}(\sG(S))$ be the number of 2-stars in $\sG(S)$ with different directions on the two edges.

For an undirected graph functional $\sG$, denote by $d(x, \sG(S))$ the degree of the vertex $x$ in the graph $\sG(S\cup \{x\})$. As in~\eqref{lambdadefn} and~\eqref{2star}, let
\begin{equation}
\lambda(x, \sG(S))=\frac{Nd(x, \sG(S^x))}{|E(\sG(S^x))|}, \quad T_2(\sG(S))=\sum_{x\in S} {d(x, \sG(S))\choose 2}.
\label{lambdaundirected}
\end{equation}

Intuitively, the functions $\lambda^\uparrow(x, \sG(S))$, $\lambda^\uparrow(x, \sG(S))$, and $\lambda(x, \sG(S))$, measure the relative position of the point $x$ in set $S\cup\{x\}$. For example, in the FR-test or the test based on the $K$-NN graph, large values of $\lambda$ correspond to points near the center of the data-cloud, whereas small values correspond to outliers. Similarly, for depth-based tests, small/large values of $\lambda^{\downarrow}$  generally correspond to points which are outliers. In fact, for such tests, this is directly related to the relative outlyingness of  the point $x$, as defined in~\eqref{outlyingness} (see Observation~\ref{eq:depthlimit}).

\subsection{Asymptotic Efficiency for Directed Graph Functionals}
\label{directedae}

Let $\sG$ be a directed graph functional in $\R^d$ and $\{\P_\theta\}_{\theta\in \Theta}$ a parametric family of distributions satisfying~Assumption \ref{dist4} at $\theta=\theta_1\in \Theta\subseteq \R^p$. To derive the asymptotic efficiency of a general graph-based test various assumptions are required on the graph functional $\sG$.

\begin{assumption}\label{varcondition} ({\it Variance Condition}) The pair $(\sG, \P_{\theta_1})$ is said to satisfy the {\it variance condition with parameters} $(\beta_0, \beta_0^+, \beta_1^\uparrow, \beta_1^\downarrow, \beta^+_1)$, if for $\cV_N:=\{V_1, V_2, \ldots, V_N\}$ i.i.d. from $\P_{\theta_1}$, there exist finite non-negative constants $\beta_0$, $\beta_0^+$, $\beta_1^\uparrow$, $\beta_1^\downarrow$, and $\beta^+_1$ (which do not depend on $\P_{\theta_1}$) such that 
\begin{enumerate}[(a)]
\item $\frac{N}{|E(\sG(\cV_N))|}\pto \beta_0$, and  $\frac{N|E^+(\sG(\cV_N))|}{|E(\sG(\cV_N))|^2}\pto \beta_0^+$,

\item $\frac{NT_2^{\uparrow}(\sG(\cV_N))}{|E(\sG(\cV_N))|^2}\pto \beta_1^{\uparrow},~\frac{NT_2^{\downarrow}(\sG(\cV_N))}{|E(\sG(\cV_N))|^2}\pto \beta_1^{\downarrow}$, and 

\item $\frac{NT_2^{+}(\sG(\cV_N))}{|E(\sG(\cV_N))|^2}\pto \beta_1^{+}$.
\end{enumerate}
\end{assumption}

The asymptotic efficiency will be derived using Le Cam's Third Lemma \cite[Corollary 12.3.2]{lr}, for which the joint normality of $\cR(\sG(\cZ_N))$ and the log-likelihood ratio $L_N$ is required. For this the following two conditions are required:

\begin{assumption}\label{covcondition}({\it Covariance Condition}) The pair $(\sG, \P_{\theta_1})$ is said to satisfy the covariance condition if for $\cV_N:=\{V_1, V_2, \ldots, V_N\}$ i.i.d. from $\P_{\theta_1}$ the following holds: 

\begin{enumerate}[(a)]

\item There exists functions $\lambda^\uparrow, \lambda^\downarrow: \cK \rightarrow \R$, such that,  for almost all $z\in \cK$, 
\begin{align}\label{lambda}
\lambda^\uparrow (z)=\lim_{N\rightarrow \infty}\E(\lambda^\uparrow(z, \sG(\cV_N))), \quad \text{and} \quad \lambda^\downarrow (z)=\lim_{N\rightarrow \infty}\E(\lambda^\downarrow(z, \sG(\cV_N))), 
\end{align} 
and zero otherwise. 
\item For $h \in \R^p$ as in~\eqref{localpower2},
\begin{equation}
\frac{1}{N}\sum_{i=1}^N\langle h, \eta(V_i, \theta_1)\rangle\lambda^{\uparrow}(V_i, \sG(\cV_N))\pto \int \langle h, \grad  f(z|\theta_1)\rangle \lambda^{\uparrow}(z)\mathrm d z, 
\label{covassumption1}
\end{equation}
and the same holds for $\lambda^{\downarrow}$. 
\end{enumerate}
\end{assumption}

Next, define $\Delta (\sG(S)):=\max_{i \in [N]}d(V_i, \sG(S))$ the total maximum degree of the graph $\sG(S)$, where $[N]:=\{1, 2, \ldots, N\}$. 

\begin{assumption}\label{normalcondition}({\it Normality Condition})
The pair $(\sG, \P_{\theta_1})$ is said to satisfy the {\it normality condition} if for $\cV_N:=\{V_1, V_2, \ldots, V_N\}$ i.i.d. from $\P_{\theta_1}$ {\it either} one of the following holds:
\begin{enumerate}[(a)] 

\item ({\it Condition N1}) The maximum degree $\Delta (\sG(\cV_N))=O_P(1)$.

\item  ({\it Condition N2}) The maximum degree $\Delta (\sG(\cV_N)) \pto \infty$ and $\frac{N \Delta(\sG(\cV_N))}{|E(\sG(\cV_N))|} =O_P(1)$.
\end{enumerate}
\end{assumption}

\begin{remark}Assumption~\ref{normalcondition} above implies that either (a) the graph has bounded degree (Condition N1), or (b) the maximum degree $\Delta(\sG(\cV_N))$ is of the same order as the average degree $|E(\sG(\cV_N))|/N$ (Condition N2), that is, the graph is `approximately' regular. This ensures that the conditional standard deviation $\sqrt{\Var(T(\sG(\cZ_N))|\cZ_N)}=\Theta(1/\sqrt N)$, justifying the scaling in \eqref{2sampleG1}. It is possible to consider a slightly more general class of statistics which 
re-normalizes $T(\sG(\cZ_N))-\E(T(\sG(\cZ_N)))$ by $\sqrt{\Var(T(\sG(\cZ_N))|\cZ_N)}$ itself (than $1/\sqrt N$). The asymptotic efficiency of such a  statistic can be similarly derived, after the covariance condition has been rescaled and the normality condition has been modified appropriately. We have decided to scale by $1/\sqrt N$ instead, because this leads to more interpretable conditions (Assumption~\ref{normalcondition}), which are easier to apply in our examples, and includes all known graph-based two-sample tests in the literature. 
\end{remark}

The following result gives the asymptotic efficiency of a graph-based test, if the graph functional satisfies the above conditions. To this end, define $r:=2pq$, where $p$ and $q$ are defined in \eqref{pq}.

\begin{thm}\label{ASYME} Let $\sG$ be a directed graph functional and $\{\P_\theta\}_{\theta\in \Theta}$ a parametric family of distributions in $\R^d$ satisfying~Assumption \ref{dist4} at $\theta=\theta_1\in \Theta\subseteq \R^p$. Suppose the pair $(\sG, \P_{\theta_1})$ satisfies 
the variance condition~\ref{varcondition} with parameters $(\beta_0, \beta_0^+, \beta_1^\uparrow, \beta_1^\downarrow, \beta^+_1)$, the covariance condition~\ref{covcondition}, and the normality condition~\ref{normalcondition}. Then the asymptotic efficiency of the test statistic~\eqref{2sampleG1} for the testing problem~\eqref{localpower2} is
\begin{align*}
\mathrm{AE(\sG)}:=\frac{\frac{r}{2}\left(p\int \langle h, \grad  f(z|\theta_1)\rangle \lambda^{\downarrow}(z)\mathrm d z-q \int \langle h, \grad f(z|\theta_1)\rangle \lambda^{\uparrow}(z)\mathrm d z\right)}{\sqrt{r\left\{\frac{\beta_0-1}{2}+ q \beta^{\uparrow}_1+ p\beta^{\downarrow}_1-\frac{r}{2}\left(\frac{\beta_0}{2}+\beta^{+}_0+\beta^{\downarrow}_1+ \beta^{\uparrow}_1+\beta^{+}_1-2\right)\right\}}},
\end{align*} 
whenever the denominator above is strictly positive.
\end{thm}

The proof of the theorem in given in Appendix \ref{varcovlimits}. The efficiency formula in Theorem \ref{ASYME} shows how combinatorial properties of the underlying graph effect the performance of the associated test, through the functions $\lambda^\uparrow$, $\lambda^\downarrow$. Moreover, the formula holds (for tests based on geometric graphs)  for any distance function $\rho$ in $\R^d$ as long as the pooled sample $\cZ_N$ is nice with respect to $\rho$ (that is, all pairwise distances are unique) and the pair $(\sG, \P_{\theta_1})$ satisfies the assumptions of the theorem. In our applications in Section \ref{applications}, $\rho$ will be the Euclidean metric, but the result continues to hold for other natural distance functions like $L_p$ and the Mahalanobis distance.

\begin{remark} (Limiting Null Distribution) The proof of Theorem~\ref{ASYME} also gives the limiting null distribution of the test statistic~\eqref{2sampleG1}, unifying several known results in the literature
Note that the variance condition~\ref{varcondition} and the normality condition~\ref{normalcondition} naturally extend to the pair $(\sG, \P_f)$, where $\P_f$ is the probability measure induced by $f$. The proof of Theorem~\ref{ASYME} shows that if the pair $(\sG, \P_f)$ satisfies the variance condition~\ref{varcondition} and the normality condition~\ref{normalcondition}, then under the null hypothesis ($f=g$) $\cR(\sG(\cZ_N))\dto N(0, \sigma_1^2)$, where 
$\sigma_1$ is the denominator of the formula in Theorem~\ref{ASYME}. In particular, this gives a unified proof for the asymptotic null distribution of the FR-test \cite[Theorem 1]{henzepenrose}, the $K$-NN test \cite[Theorem 3.1]{schilling}, the cross match test \cite[Proposition 2]{rosenbaum}, and the Liu-Singh rank sum statistic~\cite[Theorem 6.2]{liu_singh}. 
\end{remark}

\subsection{Asymptotic Efficiency for Undirected Graph Functionals}
\label{undirectedae}

Every undirected graph functional $\sG$ can be modified to a directed graph functional $\sG_+$ in a natural way: For $S\subset \R^d$ finite, $\sG_+(S)$ is obtained by replacing every edge in $\sG(S)$ with two edges one in each direction. The asymptotic efficiency of the test based on $\sG$ can then be derived by applying  Theorem~\ref{ASYME} to the directed graph functional $\sG_+$. The following are the analogue of Assumptions~\ref{varcondition}-\ref{normalcondition} for undirected graph functionals.

\begin{assumption}\label{varundirectedcond} Let $\sG$ be an undirected graph functional and assume $\cV_N:=\{V_1, V_2, \ldots, V_N\}$ be i.i.d. with density $\P_{\theta_1}$.  The pair $(\sG, \P_{\theta_1})$ is said to satisfy Assumption~\ref{varundirectedcond} if the following hold:

\begin{itemize}

\item ($(\gamma_0, \gamma_1)$-{\it Undirected Variance Condition}) There exists finite non-negative constants $\gamma_0$,  $\gamma_1$ such that 
\begin{align}
\frac{N}{|E(\sG(\cV_N))|}\pto \gamma_0 \quad \text{ and } \quad \frac{N|T_2(\sG(\cV_N))|}{|E(\sG(\cV_N))|^2}\pto \gamma_1.
\label{gamma0gamma1}
\end{align}

\item ({\it Undirected Covariance Condition}) There exists a function $\lambda: \cK \rightarrow \R$, such that for almost all $z\in \cK$, $\lambda(z):=\lim_{N\rightarrow \infty}\E(\lambda(z, \sG(\cV_N)))$, and zero otherwise. Moreover, for $h \in \R^p$ as in~\eqref{localpower2}, 
\begin{equation}
\frac{1}{N}\sum_{i=1}^N\langle h, \eta(V_i, \theta_1)\rangle\lambda(V_i, \sG(\cV_N))\pto \int \langle h, \grad  f(z|\theta_1)\rangle \lambda(z)\mathrm d z.
\label{covundirected}
\end{equation}

\item ({\it Normality Condition}) The pair $(\sG_+, \P_{\theta_1})$ satisfies the normality condition as in Assumption~\ref{normalcondition}.

\end{itemize}
\end{assumption}

If $\sG$ satisfies the above condition, then $\sG_+$ satisfies Assumptions~\ref{varcondition}-\ref{normalcondition}, and  applying Theorem~\ref{ASYME} for $\sG_+$ we get the following:

\begin{cor} \label{ASYMEundirected}Let $\sG$ be an undirected graph functional and $\{\P_\theta\}_{\theta\in \Theta}$ a parametric family of distributions in $\R^p$ satisfying~Assumption \ref{dist4} at $\theta=\theta_1\in \Theta$.  If the pair $(\sG, \P_{\theta_1})$ satisfies Assumption \ref{varundirectedcond}, then the asymptotic efficiency of the test statistic \eqref{2sampleG1} is
\begin{align}\label{effundirected}
\mathrm{AE(\sG)}=\frac{\left(\frac{r}{2}(p-q)\int \langle h, \grad  f(z|\theta_1)\rangle \lambda(z)\mathrm d z \right)}{ \sqrt{r\left\{\gamma_0(1-r)+ (\gamma_1-2)(1-2r) \right\}}},
\end{align}
whenever the denominator above is strictly positive.
\end{cor}

\begin{remark}\label{rm:efficiency_formula} Note that the numerator in~\eqref{effundirected} is zero when $p=q=1/2$, that is, when the two sample sizes $N_1$ and $N_2$ are asymptotically equal. This is because, conditional on the graph, the variables $\{\psi(c_i, c_j): (Z_i, Z_j) \in E(\sG(\cZ_N))\}$ are pairwise independent when $p=q$, and, as a result, the test statistic \eqref{2sampleG1} does not correlate with the likelihood ratio. Therefore, depending on the value of the denominator in \eqref{effundirected} the following two cases arise:
\begin{enumerate}[$\bullet$]
\item If the graph functional is non-sparse, that is, $|E(\sG(\cV_N))|/N\rightarrow \infty$, then $\gamma_0=0$ in~\eqref{gamma0gamma1} and the denominator in~\eqref{effundirected} is zero, since $r=2pq=1/2$. Therefore, non-sparse undirected graph functionals do not have a non-degenerate distribution at the $N^{\frac{1}{2}}$ scale when $p=q$, and Corollary~\ref{ASYMEundirected} does not apply.  This degeneracy is well-known in the graph-coloring literature: For example, the limiting distribution of the test statistic under the null hypothesis in this case follows from \cite[Theorem 1.3]{BDM}.

\item If the graph is sparse, that is, $\gamma_0>0$, then the denominator of~\eqref{effundirected} is non-zero when $p=q$. {\it This shows that tests based on sparse graph functionals cannot have non-zero efficiencies when the two-sample sizes are asymptotically equal.}

\end{enumerate}

\end{remark}

\section{Applications}
\label{applications}

In this section we compute the asymptotic efficiencies of the tests described in Section~\ref{graph2tests}, under the Euclidean distance,  using Theorem~\ref{ASYME} and Corollary~\ref{ASYMEundirected}. Extensions and generalizations which can be used to construct locally efficient tests are also discussed.

\subsection{The Friedman-Rafsky (FR) Test} Let $\cT$ be the MST functional as in Definition~\ref{mst} and  consider the two-sample test based on $\cT$~\eqref{Tfr}. The following theorem shows that this test has zero asymptotic efficiency, under the Euclidean distance.

\begin{thm}\label{MSTAE}
Let $\{\P_\theta\}_{\theta\in \Theta}$ be a parametric family of distributions in $\R^d$ satisfying~Assumption \ref{dist4} at $\theta=\theta_1\in \Theta$. Then the asymptotic efficiency of the Friedman-Rafsky test~\eqref{Tfr}, under the Euclidean distance, for the testing problem~\eqref{localpower2}, is zero, that is, $\mathrm{AE}(\cT)=0$.
\end{thm}

The FR-test has zero asymptotic efficiency because the function $\lambda(z)$, in this case, does not depend on $z$, and hence, the numerator in~\eqref{effundirected} is zero. This is a consequence of the following well-known result: For $\cV_N=\{V_1, V_2, \ldots, V_N\}$ i.i.d. $\P_{\theta_1}$ and $z \in \cK$, $\lim_{N\rightarrow \infty}\E(\lambda(z, \cT(\cV_N)))=2$ (refer to \cite[Proposition 1]{henzepenrose}). The Efron-Stein inequality \cite{esineq} can then be used to show that $(\cT, \P_{\theta_1})$ satisfies the covariance condition with the constant function $\lambda(z)=2$. The details of the proof are given in the supplementary materials.

\subsection{Tests Based on Stabilizing  Graphs}
\label{stabilizingtests}

The proof of Theorem~\ref{MSTAE} uses the fact that the MST graph functional has local dependence, that is, addition/deletion of a point only effects the edges incident on the neighborhood of that point. This phenomenon holds for many other random geometric graphs and was formalized by Penrose and Yukich \cite{py} using the notion of stabilization.

Let $\sG$ be a graph functional defined for all locally finite subsets of $\R^d$. (The $K$-NN graph can be naturally extended to locally finite infinite points sets. Aldous and Steele \cite{aldous_steele} extended the MST graph functional to locally finite infinite point sets using the Prim's algorithm.)  For $S\subset \R^d$ locally finite and $x\in \R^d$, let $E(x, \sG(S))$ be the set edges incident on $x$ in $\sG(S\cup \{x\})$. Note that $|E(x, \sG(S))|=d(x, \sG(S))$, the (total) degree of the vertex $x$ in $\sG(S\cup \{x\})$.

\begin{defn}\label{ts}Given $S\subset \R^d$ and $y\in \R^d$ and $a\in \R$, denote by $y+S=\{y+z: z\in S\}$ and $aS:=\{az: z\in S\}$. A graph functional $\sG$ is said to be {\it translation invariant} if the graphs $\sG(x+S)$ and $\sG(S)$ are isomorphic for all points $x\in \R^d$ and all locally finite $S\subset \R^d$. A graph functional $\sG$ is {\it scale invariant} if $\sG(aS)$ and $\sG(S)$ are isomorphic for all points $a\in \R$ and  and all locally finite $S\subset \R^d$. 
\end{defn}

Let $\cP_\lambda$ be the Poisson process of intensity $\lambda\geq 0$ in $\R^d$, and $\cP_{\lambda}^x := \cP_{\lambda} \cup \{x\}$, for $x\in \R^d$. Penrose and Yukich \cite{py} defined stabilization of graph functionals over homogeneous Poisson processes as follows:

\begin{defn}[Penrose and Yukich \cite{py}] \label{stabilization} A translation and scale invariant graph functional $\sG$ {\it stabilizes} $\cP_\lambda$ if there exists a random but almost surely finite variable $R$ such that
\begin{equation}
E(0, \sG(\cP_{\lambda}^0))=E(0,  \sG(\cP_{\lambda}^0\cap B(0, R)\cup \sA)),
\end{equation}
for all finite $\sA\subset \R^d \setminus B(0, R)$, where $B(0, R)$ is the (Euclidean) ball of radius $R$ with center at the point $0\in \R^d$.
\end{defn}

\begin{remark}Informally, stabilization ensures the insertion of a point (or finitely many points) `far' away from the origin $0$ does not effect the degree of $0$ in the graph $\sG(\cP_{\lambda}^0)$, that is, it only has a `local effect'. Many graph functionals such as the MST, the $K$-NN, the Delaunay graph, and the Gabriel graph, are stabilizing \cite{py}. 
\end{remark}

The following theorem shows that tests based on stabilizing graph functionals have zero asymptotic efficiency, under the Euclidean distance. The proof is given in the supplementary materials.

\begin{thm} Let $\{\P_\theta\}_{\theta\in \Theta}$ be a parametric family of distributions in $\R^d$ satisfying~Assumption \ref{dist4} at $\theta=\theta_1\in \Theta$, and $\sG$ be a translation and scale invariant graph functional which stabilizes $\cP_1$. If the pair $(\sG, \P_{\theta_1})$ satisfies Assumption~\ref{normalcondition} and
\begin{equation}
\quad \sup_{N\in \N}\sup_{z\in \R^d}\E\left(d(z, \sG(\cZ_N))^s\right)< \infty,
\label{degcond}
\end{equation}
for some $s >4$, then the asymptotic efficiency of the two-sample test based on $\sG$~\eqref{2sampleG1}, under the Euclidean distance, for the testing problem~\eqref{localpower2}, is zero, that is, $\mathrm{AE}(\sG)=0$. 
\label{POWERSTABILIZE}
\end{thm}

Theorem~\ref{POWERSTABILIZE} can be used to re-derive Theorem \ref{MSTAE} and compute the asymptotic efficiency of the $K$-NN test.

\begin{enumerate}

\item[(1)] {\it Minimum Spanning Tree (MST)}: By \cite[Lemma 2.1]{py} the MST graph functional $\cT$ stabilizes $\cP_1$. Moreover, the degree of a vertex in the MST of a set of points in $\R^d$ is bounded by a constant $B_d$, depending only on the dimension $d$ \cite[Lemma 4]{aldous_steele}. Therefore, the normality condition N1 in Assumption~\ref{normalcondition} and the moment condition~\eqref{degcond} are trivially satisfied. Theorem~\ref{POWERSTABILIZE} then implies $\mathrm{AE(\cT)}=0$, thus re-deriving Theorem~\ref{MSTAE}.

\item[(2)] {\it $K$-Nearest Neighbor ($K$-NN)}: By \cite[Lemma 6.1]{pyclt}, the $K$-NN graph functional $\cN_K$, where $K=O(1)$ is fixed with $N$, stabilizes $\cP_1$. Condition N1 in Assumption~\ref{normalcondition} and the moment condition~\eqref{degcond} are trivially satisfied, since $\cN_K$ is a bounded degree graph functional. This implies, $\mathrm{AE}(\cN_K)=0$, showing that the test based on the $K$-NN graph has no asymptotic local power, when $K=O(1)$. 

\end{enumerate}

\subsection{The Test Based on the $K$-NN Graph}
\label{sec:knn}

The result in the previous section shows that the test based on the $K$-NN graph has zero Pitman efficiency, when $K=O(1)$ is fixed with the $N$. But what about when $K=K_N \rightarrow \infty$ with $N$? In this case, the nearest-neighbor graph is no longer stabilizing, and Theorem \ref{POWERSTABILIZE} does not apply.  However, we can directly invoke Corollary \ref{ASYMEundirected} to compute the asymptotic efficiency. To this end, let $S \subset \R^d$ be a finite set and $z \in \R^d$ be a fixed point and $K=K_N \rightarrow \infty$. It follows from \cite[Lemma 1]{randomized_nn} that  $d(z, \cN_{K_N}(S)) \leq C_d K_N$, where $C_d$ is a constant depending only on the dimension $d$. This implies, for $\cV_N=\{V_1, V_2, \ldots, V_N\}$ i.i.d. $f(\cdot|\theta_1)$, the maximum degree $\Delta(\cN_{K_N}(\cV_N))=\max_{i \in [N]} d(V_i, \cN_{K_N}(\cV_N)) \leq C_d K_N$. Moreover, each vertex in the graph $\cN_{K_N}(\cV_N)$ has degree at least $K_N$, which means the total number of edges $|E(\cN_{K_N}(\cV_N))|=\frac{1}{2}\sum_{i=1}^N d(V_i, \cN_{K_N}(\cV_N)) \geq  \frac{K_N N}{2}$. Hence, 
\begin{align}\label{eq:KNN_N2condition}
\frac{N \Delta(\cN_{K_N}(\cV_N))}{|E(\cN_{K_N}(\cV_N))|} \lesssim C_d=O(1),
\end{align} 
that is, Condition N2 in Assumption \ref{normalcondition}  is satisfied.  Therefore, assuming there exists functions $\eta_0, \eta_1: \cK \rightarrow \R$ such that 
\begin{align}\label{eq:z01}
\eta_0(z):=\lim_{N \rightarrow \infty}\frac{1}{K_N} \E(d(z, \cN_{K_N}(\cV_N))) \quad \text{and} \quad \eta_1(z):=\lim_{N \rightarrow \infty}\frac{1}{K_N^2} \E{d(z, \cN_{K_N}(\cV_N))\choose 2},
\end{align}
for almost all $z \in \cK$ (and zero otherwise), the asymptotic efficiency of the $K$-NN can be derived using Corollary \eqref{ASYMEundirected}. (Note that both the limits in \eqref{eq:z01} are finite, since $\max_{z \in \R^d} d(z, \cN_{K_N}(\cV_N)) \leq C_d K_N$, almost surely.)

\begin{ppn}\label{ppn:knn} Let $\{\P_\theta\}_{\theta\in \Theta}$ be a parametric family of distributions in $\R^d$ satisfying~Assumption \ref{dist4} at $\theta=\theta_1\in \Theta$. Then the asymptotic efficiency of the $K$-NN test (where $K=K_N \rightarrow \infty$), for the testing problem~\eqref{localpower2}, is 
\begin{align}\label{eq:efficiencyKNN}
\mathrm{AE}(\cN_{K_N})=\frac{\left(r(p-q) \frac{\int \eta_0(z) \grad  f(z|\theta_1) \mathrm dz }{\int \eta_0(z) f(z|\theta_1) \mathrm dz} \right)}{ \sqrt{r \left[\frac{4 \int \eta_1(z) f(z|\theta_1) \mathrm dz}{\left(\int \eta_0(z) f(z|\theta_1) \mathrm dz \right)^2}-2\right](1-2r) }},
\end{align}
whenever the denominator above is strictly positive, where $\eta_0(\cdot)$ and $\eta_1(\cdot)$ are as defined in \eqref{eq:z01}. 
\end{ppn}

The proof of this result, which entails verifying the conditions in Corollary \ref{ASYMEundirected}, is given in the supplementary materials. Note that the formula in \eqref{eq:efficiencyKNN} has a couple of degeneracies: 

\begin{enumerate}

\item[(1)] When $p=q$, both the numerator and the denominator in \eqref{eq:efficiencyKNN} is zero, and the result does not apply (recall discussion in Remark \ref{rm:efficiency_formula}). 

\item[(2)] When $\frac{\int \eta_1(z) f(z|\theta_1) \mathrm dz}{(\int \eta_0(z) f(z|\theta_1) \mathrm dz)^2}=\frac{1}{2}$, the denominator in \eqref{eq:efficiencyKNN} is zero.  This happens, when $K_N=N-o(N)$, that is, the graph $\cN_{K_N}(\cV_N)$ is `nearly complete' (has ${N \choose 2}-o(N^2)$ edges) and $\eta_0(z)=1$, $\eta_1(z)=\frac{1}{2}$, for all $z \in \R^d$. This is expected because in the extreme case where $K_N=N-1$, $\cN_{K_N}(\cV_N)$ is the complete graph and the statistic \eqref{graph2test} is non-random and, hence,  powerless. 
\end{enumerate}

Proposition \ref{ppn:knn} has several interesting consequences. To begin with, note that when $K=O(1)$ is fixed, then $\eta_0(z)$ does not depend on $z$ and the RHS of \eqref{eq:efficiencyKNN} is zero, as shown earlier in Theorem \ref{POWERSTABILIZE}. 
The situation, however, is different when $K=K_N \rightarrow \infty$ grows with $N$. 
Even though the exact dependence of $d(z, \cN_{K_N}(V_N))$ on $z$ and $K$ appears to be quite delicate, the RHS in \eqref{eq:efficiencyKNN} is expected to be non-zero, when $K=K_N \rightarrow \infty$ sufficiently fast, for instance,  when $K_N \asymp N^{\alpha}$, for some $\alpha \in (0, 1]$. This is validated by the simulation results in Section \ref{examples}, where we observe that the local power of the $K$-NN test increases with $K$, and eventually dominates other parametric and non-parametric tests. This makes the $K$-NN test (when $K$  grows with $N$) desirable, both theoretically (non-zero Pitman efficiency) and in applications (easy computation and good finite sample power).

\subsection{Cross-Match (CM) Test}

Let $\cW$ be the minimum non-bipartite matching (NBM) graph functional as in Definition~\ref{mdm}. It is unknown  whether $\cW$ is stabilizing \cite{py}, and so, Theorem~\ref{POWERSTABILIZE} cannot be applied to compute the asymptotic efficiency of the cross-match test. However, in this case, Corollary~\ref{ASYMEundirected} can be used directly to derive the following:

\begin{cor}\label{mieq:BNatchae}
Let $\{\P_\theta\}_{\theta\in \Theta}$ be a parametric family of distributions in $\R^d$ satisfying~Assumption \ref{dist4} at $\theta=\theta_1\in \Theta$. Then the asymptotic efficiency of the cross-match (CM) test~\eqref{Tmdm}, under the Euclidean distance, for the testing problem~\eqref{localpower2}, is zero, that is, $\mathrm{AE}(\cW)=0$.
\end{cor}

\begin{proof}Let $N$ be even and $\cV_N=\{V_1, V_2, \ldots, V_N\}$ i.i.d. $\P_{\theta_1}$. In this case, $|E(\cW(\cV_N))|=N/2$ and $|T_2(\cW(\cV_N))|=0$. Therefore, $(\cW, \P_{\theta_1})$ satisfies the $(2, 0)$-undirected variance condition~\eqref{gamma0gamma1}. The normality condition N1 holds, since $d(V_i, \cW(\cV_N))=1$ for all $V_i\in \cV_N$. This also implies that the undirected covariance condition holds with the constant function $\lambda(z)=2$. The result then follows by~Corollary~\ref{ASYMEundirected}.
\end{proof}

\subsection{Depth-Based Tests}
\label{depthtest}

Let $\cZ_N=\sX_{N_1}\cup\sY_{N_2}$ be the pooled sample, and $F_{N_1}$ the empirical distribution of $\sX_{N_1}$. The two sample test based on a depth function $D$~\eqref{QFGsample} rejects for large values of $|T(\sG_D(\cZ_N))|$, where the graph $\sG_D(\cZ_N)$ has vertex set $\cZ_N$ with a directed edge $(Z_i, Z_j)$ whenever $D(Z_i, F_{N_1})\leq D(Z_j, F_{N_1})$. If the depth function  $D(X, F)$, where $X\sim F$, has a continuous distribution then $\sG_D(\cZ_N)$ is a complete graph 
($|E(\sG_D(\cZ_N))|=N(N-1)/2$) with directions on the edges depending on the relative ordering of the depth of the two end-points.

\begin{defn}\label{depthconditions} Let $F$ be a distribution function in $\R^d$ with empirical distribution function $F_N$.  A depth function $D$ is said to be {\it good with respect} to $F$ if
\begin{itemize}
\item[(A1)] For $X\sim F$, the distribution of $D(X, F)$ is continuous.

\item[(A2)] $\P(y_1 \leq D(Y, F)\leq y_2)\leq C|y_1-y_2|$, for some constant $C$ and any $y_1, y_2 \in [0, 1]$.

\item[(A3)] $\sup_{x\in \R^d} |D(x, F_N)-D(x, F)|=o(1)$ almost surely and in expectation.
\end{itemize}
\end{defn}

The standard depth functions, like the one discussed in Section \ref{s2depth}, satisfy the above conditions, for any continuous distribution function $F$. The following result gives the asymptotic efficiencies of tests based on such depth functions.  To this end, let $\{\P_\theta\}_{\theta\in \Theta}$ be a parametric family of distributions in $\R^d$ satisfying~Assumption \ref{dist4} at $\theta=\theta_1\in \Theta\subseteq \R^p$. Moreover, let $F_{\theta_1}$ be the distribution function of $\P_{\theta_1}$. 

\begin{thm}\label{POWERDEPTH}
The asymptotic efficiency of the two-sample test~\eqref{QFGsample} based on a good depth function $D$ (with respect to $F_{\theta_1}$), for the testing problem~\eqref{localpower2}, is 
\begin{equation}\label{eq:pdepth}
\mathrm{AE}(\sG_D)=-\sqrt{6 r} \int \langle h, \grad f(x|\theta_1)\rangle R(x, F_{\theta_1}) \mathrm dx,
\end{equation}
where $R(x, F_{\theta_1})$ is as defined in~\eqref{outlyingness}.
\end{thm}

\subsection{The Chen-Friedman Test} Recently, Chen and Friedman \cite{chenfriedman} proposed  a modification of the test statistic \eqref{2sampleG1}, which improves upon the finite sample power of the FR and the $K$-NN tests, especially when sample size is small and dimension is large. To this end, given an undirected graph functional $\sG$ and $b\in \{1, 2\}$, define
\begin{align}
\cR_b(\sG(\cZ_N))=&\sqrt N\left\{T_b(\sG(\cZ_N))-\E(T_b(\sG(\cZ_N))) \right\},
\label{R12}
\end{align}
where 
\begin{align}\label{T12}
T_b(\sG(\cZ_N)):=\frac{\sum_{1\leq i < j\leq N} \psi_b(c_i, c_j) \pmb 1\{(Z_i, Z_j)\in E(\sG(\cZ_N))\}}{|E(\sG(\cZ_N))|},
\end{align}
with $\psi_b(c_i, c_j)=\pmb1\{c_i=c_j=b\}$. Note that $T_1(\sG(\cZ_N))$ (respectively $T_2(\sG(\cZ_N))$)  is the number of edges in $\sG(\cZ_N)$ within sample 1 (respectively sample 2). Moreover, denote by $\Lambda_N$ the variance-covariance matrix of $(\cR_1(\sG(\cZ_N)), \cR_2(\sG(\cZ_N)))^t$ given the pooled data $\cZ_N$. The {\it Chen-Friedman} (CF) test statistic is defined as  
\begin{align*}
S(\sG(\cZ_N))=
\begin{pmatrix}
\cR_1(\sG(\cZ_N)) & \cR_2(\sG(\cZ_N))
\end{pmatrix}\Lambda_N^{-1} \begin{pmatrix}
\cR_1(\sG(\cZ_N)) \\
\cR_2(\sG(\cZ_N))
\end{pmatrix},
\end{align*}
whenever the matrix $\Lambda_N$ is invertible. Chen and Friedman \cite[Theorem 5.1.1]{chenfriedman} showed that under the null $H_0$ $S(\sG(\cZ_N)) \dto \chi_2^2$, a chi-squared distribution with 2 degrees of freedom, therefore, the test function with rejection region $\{S(\sG(\cZ_N)) > \chi_{2, 1-\alpha}^2\}$ is asymptotically size $\alpha$ for \eqref{test}, where $\chi_{2, 1-\alpha}^2$ is the $(1-\alpha)$-th quantile of the chi-squared distribution with 2 degrees of freedom. 

We can derive the efficiency of the CF-test for a general graph functional $\sG$, using techniques similar to the proof of Theorem \ref{ASYMEundirected}. To this end, recall that for $\theta \in \R^d$ and $Z \sim N(\theta, \mathrm I)$, $Z^tZ\sim \chi^2_d(\theta^t\theta)$, the non-central chi-squared distribution with $d$ degrees of freedom and non-centrality parameter $\theta^t\theta$. 
The theorem is proved in Appendix \ref{sec:pffrnew}. 

\begin{thm}\label{thm:frnew}Let $\sG$ be an undirected graph functional and $\{\P_\theta\}_{\theta\in \Theta}$ a parametric family of distributions in $\R^p$ satisfying~Assumption \ref{dist4} at $\theta=\theta_1\in \Theta$.  If the pair $(\sG, \P_{\theta_1})$ satisfies Assumption \ref{varundirectedcond}, then  limiting power of the CF test, for the testing problem \eqref{localpower2}, is given by 
$$\lim_{N \rightarrow \infty} \P_{\theta_1+\frac{h}{\sqrt N}}(S(\sG(\cZ_N)) > \chi_{2, 1-\alpha}^2)=\P(\chi_2(\mu^t\Lambda^{-1}\mu)> \chi_{2, 1-\alpha}^2),$$
where 
$$\mu= \begin{pmatrix}
-p^2q \int \langle h, \grad  f(z|\theta_1)\rangle \lambda(z)\mathrm d z \\ 
pq^2 \int \langle h, \grad  f(z|\theta_1)\rangle \lambda(z)\mathrm d z
\end{pmatrix}, \quad
\Lambda=\begin{pmatrix}
\lambda_{11} & \lambda_{12} \\
\lambda_{12} & \lambda_{22} \\
\end{pmatrix},$$
with $\lambda_{11}:=p^2((1-p^2) \gamma_0  + r(\gamma_1 -2) )$, $\lambda_{12}:= -p^2 q^2 ( \gamma_0  + 2 \gamma_1 -8)$, and $\lambda_{22}:=q^2 ((1-q^2) \gamma_0  + r(\gamma_1 -2) )$, whenever $\Lambda$ is invertible.
\end{thm}

\section{Finite Sample Local Power}\label{examples} 

In this section, we compare the power of the different tests against local alternatives in  simulations. The CF test is computed using the {\tt R} package {\tt gTests} and the depth-functions are computed using the package {\tt fda.usc}. Throughout our simulations the level of significance is set at $\alpha=0.05$.

 
\begin{figure*}[h]
\begin{minipage}[l]{0.495\textwidth}
\centering
\small
\begin{tabular}{c||ccccc}
\hline
Dimension & FR & HD & MD & CF & $T^2$   \\
\hline
4 & 0.11 & 0.06 & 0.03 & 0.04 & 0.35  \\
10 &  0.06 & 0.04 & 0.04 & 0.03 & 0.49  \\
20 & 0.1 & 0.05 & 0.04 & 0.04 & 0.68  \\
30 & 0.08 & 0.11 & 0.06 & 0.07 & 0.88  \\
50 & 0.09 & 0.07 & 0.03 & 0.06 & 0.96 \\
100 & 0.07 & 0.06 & 0.04  & 0.09 & 1 \\
200 & 0.15 & 0.08 & 0.08 & 0.09 & 1 \\
300 & 0.22 & 0.04 & 0.11 & 0.13 & 1 \\
\hline
\end{tabular}\\
\vspace{0.1in}
\small{(a)}
\end{minipage}
\begin{minipage}[c]{0.495\textwidth}
\centering
\includegraphics[width=2.2in]
    {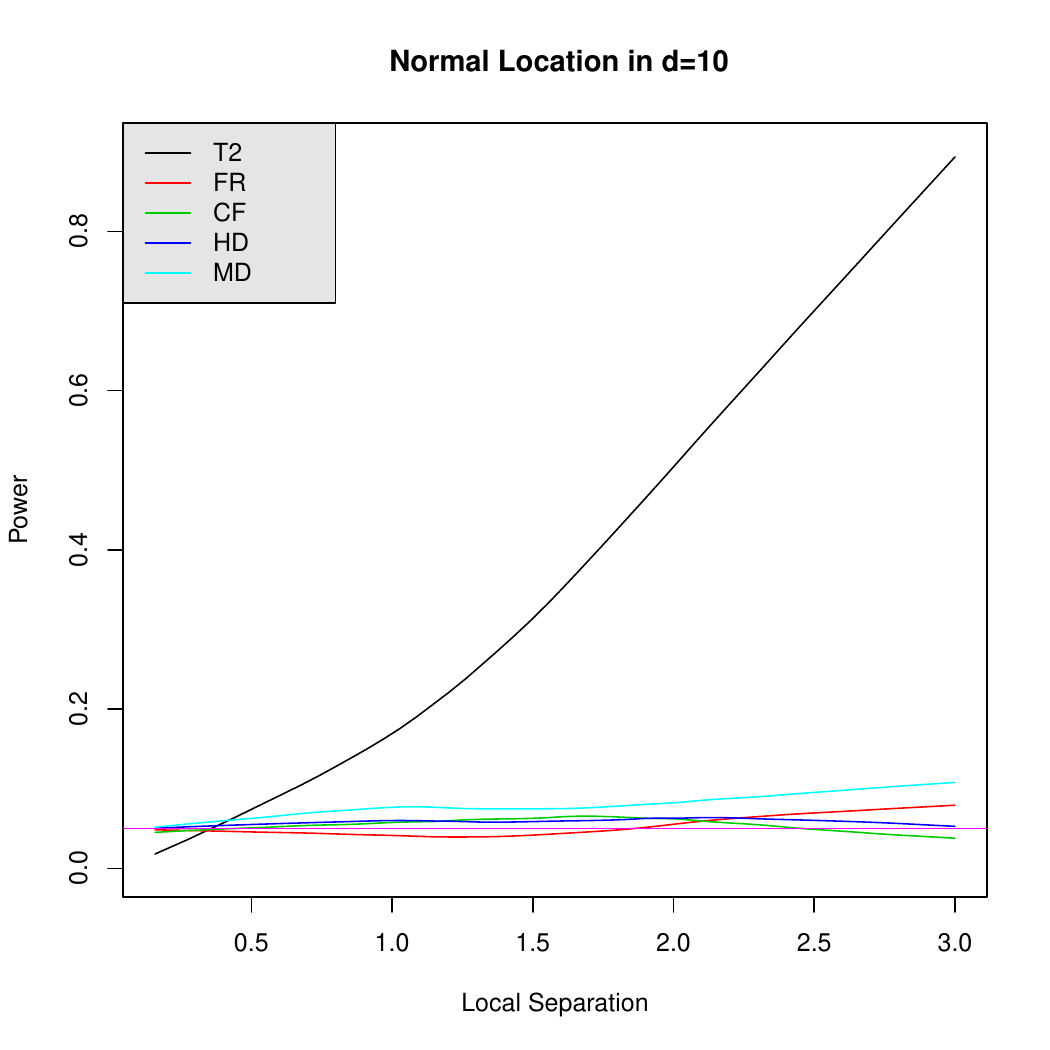}\\
\small{(b)}
\end{minipage}
\caption{\small{Power of the different tests for the normal location problem: (a) across increasing dimensions when the respective means differ by $2 \cdot \bm 1/\sqrt N$, and (b) in dimension $d=10$ when the means differ by $\delta \cdot \bm 1/\sqrt N$, as a function of $\delta$.}}
\label{nlocation}
\end{figure*}

\normalsize

\begin{example}[Normal Location]
\label{nlocationex}
Consider the parametric family $\P_{\theta}\sim N(\theta, \mathrm I)$, for $\theta \in \R^d$. The table in Figure~\ref{nlocation}(a) shows the empirical power (out of 100 repetitions) of the FR-test based on the MST, the test based on halfspace depth (HD), the test based on the Mahalanobis depth (MD), the CF test based on the MST, and the Hotelling's $T^2$ test, with $N_1=1000$ samples from $\P_{0}$ and  $N_2=500$  samples from $\P_{\frac{2 \cdot \pmb 1}{\sqrt N}}$, across increasing dimensions. (Here, $N=N_1+N_2=1500$.) The plot in Figure~\ref{nlocation}(b)  shows the empirical power (out of 100 repetitions) in dimension $d=10$ of these  tests, based on $N_1=1000$ samples from $\P_{0}$ and $N_2=500$ samples from $\P_{\frac{\delta \pmb 1}{\sqrt N}}$, over a grid of 20 values of $\delta$ in $[0, 3]$ (smoothed out using the \texttt{loess} function in \texttt{R}). The table and the plot show that the Hotelling's $T^2$-test, which is the most powerful test in this case, has the highest power. The power of the FR-test and CF-test improve slightly dimension, but is generally low, as predicted by the results above. In this case, the tests based on depth functions (the HD test and the MD test) also have low power (see Remark \ref{ex:depthlocation}). 
\end{example}

\begin{example}[Spherical Normal] \label{nscaleex}
Consider the parametric family $\P_{\sigma}\sim N(0, \sigma^2\mathrm I)$, for $\sigma > 0$. As before, the table in Figure~\ref{nscale}(a) shows the empirical power (out of 100 repetitions) of the different tests based on $N_1=1000$ samples from $\P_{0}$ and  $N_2=500$  samples from $\P_{1+ \frac{2}{\sqrt N}}$ across increasing dimensions, and  the plot in Figure~\ref{nlocation}(b)  shows the empirical power (out of 100 repetitions) in dimension $d=10$ of the different tests, based on $N_1=1000$ samples from $\P_{0}$ and $N_2=500$ samples from $\P_{1+\frac{\delta}{\sqrt N}}$, over a grid of 20 values of $\delta$ in $[0, 3]$. Here, the HD test performs very well across dimensions. The MD test also performs well for small to moderate dimensions, but starts to lose power for higher dimensions. On the other hand, the power of the FR and the CF tests are small in low dimension, however, quite interestingly, the power increase substantially with dimension, paralleling the HD test, with the CF test generally more powerful than the FR test.  This supports the findings in \cite{chenfriedman} where the FR and CF tests also exhibit high power as dimension increases, in finite-sample simulations. It is phenomenon like this that makes tests based on geometric graphs, such as the FR test and the CF test, particularly attractive for modern statistical applications. This remarkable {\it blessing of dimensionality}, can be mathematically explained as follows: Even though tests based on geometric graphs have no power in the $O(N^{-\frac{1}{2}})$ scale, the detection threshold of these tests for the spherical normal problem (and more general scale alternatives) is expected to be around $\Theta(N^{-\frac{1}{2}+\frac{1}{d}})$. (This has been proved recently by the author \cite{BBBG2} for the test based on the $K$-NN graph.) Note  that this threshold gets closer and closer to the parametric detection rate of $N^{-\frac{1}{2}}$ as $d$ increases, and, as a result, these tests attain high power as dimension increases for scale problems.


\begin{figure*}[h]
\centering
\begin{minipage}[l]{0.52\textwidth}
\centering
\small
\begin{tabular}{c||ccccccc}
\hline
Dimension & FR & HD & MD & $T^2$ & CovTest  & CF  \\
\hline
4 &  0.07 & 0.16  & 0.2  & 0.03 & 0.05 & 0.03  \\
10 & 0.07 & 0.30 & 0.45 & 0.12 & 0.11 & 0.12  \\
20 & 0.24 & 0.59 & 0.81 & 0.06 & 0.04 & 0.28  \\
30 & 0.34 & 0.76 & 0.92 & 0.08 & 0.18 & 0.36  \\
50 & 0.59 & 0.91 & 0.99 & 0.09 & 0.16 & 0.65  \\
100 & 0.72 & 1    & 1& 0.22 & 0.13 & 0.87  \\
200 & 0.95 & 1    & 1 & 0.51 & 0.05 & 0.98 \\
300 & 1 & 1    & 0.71 & 0.85 & 0.08 & 0.99 \\
\hline
\end{tabular}\\
\vspace{0.1in}
\small{(a)}
\end{minipage}
\begin{minipage}[c]{0.47\textwidth}
\centering
\includegraphics[width=2.2in]
    {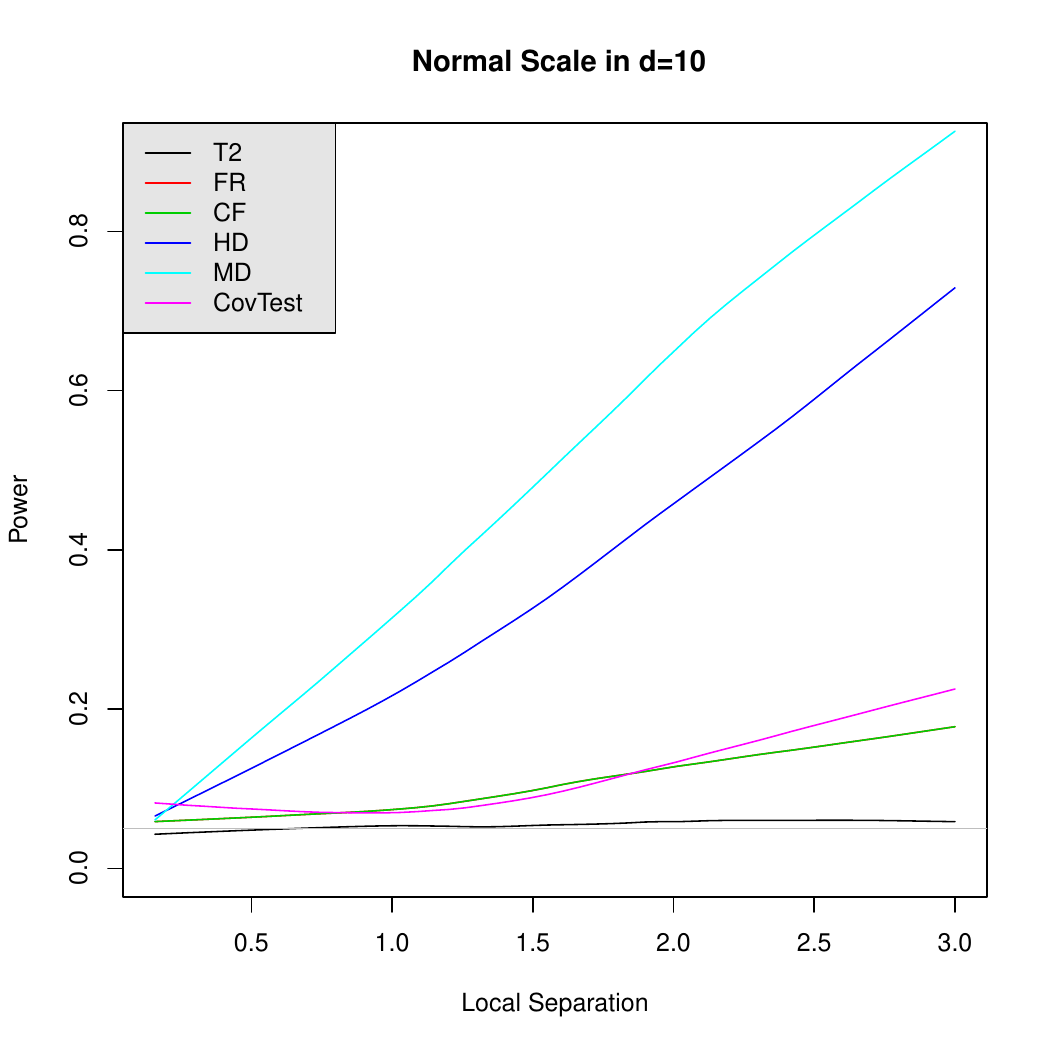}\\
\small{(b)}
\end{minipage}
\caption{\small{Power of the different tests for the normal scale problem: (a) across increasing dimensions when the standard deviations differ by $2/\sqrt N$, and (b) in dimension $d=10$ when the standard deviations differ by $\delta/\sqrt N$, as a function of $\delta$.}}
\label{nscale}
\end{figure*}

\normalsize

The table and the plot also show the power of the Hotelling $T^2$ test which, as expected, performs poorly for the scale problem, and the {\tt CovTest}, which is the parametric  likelihood ratio test for testing the equality of two normal covariance matrices. This rejects for large values  $N\log |\hat \Sigma_0|-{N_1}\log |\hat \Sigma_1|-N\log |\hat \Sigma_2|$, where $\hat \Sigma_0$, $\hat \Sigma_1$, and $\hat \Sigma_2$ are the maximum likelihood estimators of the covariance matrix of the whole data, the sample $\sX_{N_1}$, and the sample $\sY_{N_2}$, respectively. The  {\tt CovTest} performs quite poorly, as already observed in \cite{chenfriedman,fr}, which is expected because it has to estimate an increasing number of parameters as dimension increases and does not take into account the spherical structure of the covariance matrix. 
\end{example}

\begin{example}[Lognormal Location]
\label{lognormalex}
Consider the parametric family  $\P_{\theta}\sim \exp(N(\theta, \mathrm I))$ for $\theta \in \R^d$, where the exponent is taken coordinatewise. As before, the table in Figure~\ref{lognormal}(a) shows the empirical power (out of 100 repetitions) of the different tests based on $N_1=1000$ samples from $\P_{0}$ and  $N_2=500$  samples from $\P_{\frac{2 \cdot \bm 1}{\sqrt N}}$ across increasing dimensions, and  the plot in Figure~\ref{nlocation}(b)  shows the empirical power (out of 100 repetitions) in dimension $d=10$ of the different tests, based on $N_1=1000$ samples from $\P_{0}$ and $N_2=500$ samples from $\P_{\frac{\delta \pmb 1}{\sqrt N}}$, over a grid of 20 values of $\delta$ in $[0, 3]$. Changing the normal mean changes the lognormal distribution both in location and scale. In this case, the HD test is powerless (see Remark \ref{lognormal_halfspace}), but the test based the Mahalanobis depth (MD) performs very well, outperforming the Hotelling's $T^2$ when dimension increases. The FR and the CF tests have low power in small dimensions, but the power improves with dimension, for reasons similar to that in the spherical normal problem (recall Example \ref{nscaleex} above). 
\end{example}

\begin{figure*}[h]
\centering
\begin{minipage}[l]{0.495\textwidth}
\centering
\small
\begin{tabular}{c||ccccc}
\hline
Dimension & FR & HD & MD & CF & $T^2$   \\
\hline
4 & 0.05 & 0.02 & 0.11 & 0.06 & 0.17  \\
10 &  0.05 & 0.03 & 0.34 & 0.03 & 0.24  \\
20 & 0.08 & 0.09 & 0.45 & 0.08 & 0.48 \\
30 & 0.19 & 0.07 & 0.65 & 0.13 & 0.51  \\
50 & 0.37 & 0.06 & 0.83 & 0.47 & 0.81 \\
100 & 0.58 & 0.07 & 0.91 & 0.58 & 0.94 \\
200 & 0.74 & 0.07 & 1 & 0.61 & 1 \\ 
300 & 0.73 & 0.05 & 1 & 0.53 & 1 \\ 
\hline
\end{tabular}\\
\vspace{0.1in}
\small{(a)}
\end{minipage}
\begin{minipage}[c]{0.495\textwidth}
\centering
\includegraphics[width=2.2in]
    {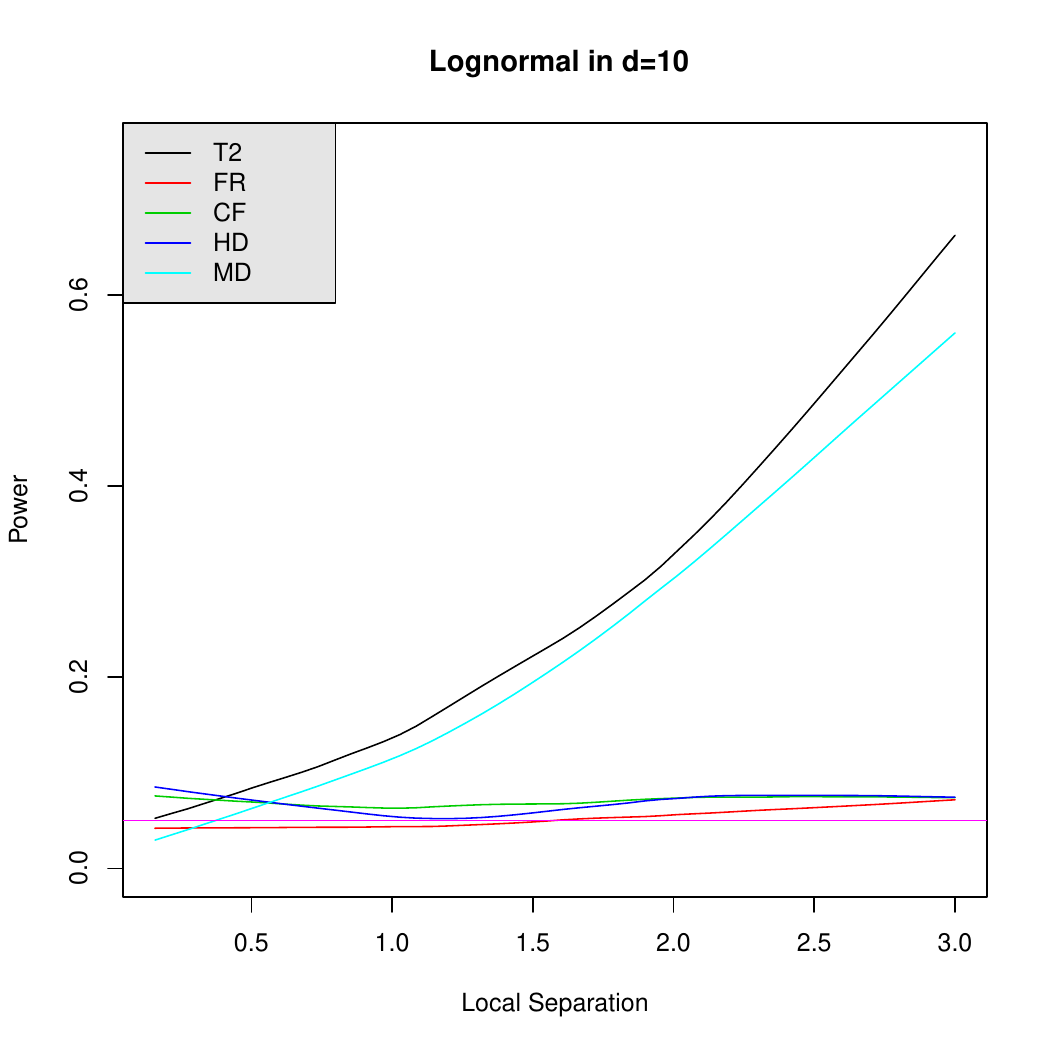}\\
\small{(b)}
\end{minipage}
\caption{\small{Power of the different tests for the log-normal location-scale problem: (a) across increasing dimensions when the means of the respective normals differ by $2 \cdot \bm 1/\sqrt N$, and (b) in dimension $d=10$ when the means of the respective normals differ by $\delta \cdot \bm 1/\sqrt N$, as a function of $\delta$.}} 
\label{lognormal}
\end{figure*}

The examples above show that when the dimension is large, the FR and the CF tests can effectively detect two distributions, unless the alternative is location-only. Moreover, both  these tests can be computed very efficiently, which makes them especially useful in applications. Moreover, Proposition \ref{ppn:knn} suggests that the test based on the $K$-NN graph can be powerful against $O(N^{-\frac{1}{2}})$ alternatives, when $K$ grows with $N$ sufficiently fast. We illustrate this result in the following example.

\begin{example}\label{knnexample} (Dependence on $K$ in the $K$-NN Test) To understand how the power of the $K$-NN depends on $K$, we consider the lognormal location family $\P_\theta\sim \exp(N(\theta, \Sigma))$, where $\theta \in \R^{10}$ and $\Sigma$ is known.  

\begin{itemize}

\item[(a)] Figure \ref{lognormal_nn}(a) shows the empirical power (out of 100 repetitions) in the independent case ($\Sigma=\mathrm I$) of the $K$-NN test for various values of $K$, and the power of the Hotelling's $T^2$ test, the HD test, and the MD test, based on $N_1=1000$ samples from $\P_0$, and $N_2=800$  samples from $\P_{\frac{\delta  \bm 1}{\sqrt N}}$, over a grid of 20 values of $\delta$ in $[0, 3]$.  For small values of $K$, the $K$-NN has low power. However, as $K$ increases, the power increases, and eventually it dominates all the other tests.

\item[(b)] Figure \ref{lognormal_nn}(b) shows the empirical power of the different tests when $\Sigma=\mathrm I+\bm 1 \bm 1'$, a rank 1-perturbation of the identity matrix. In this case, the coordinates of the lognormal are dependent. In this case, even though the overall power of all the tests is much lower, the $K$-NN dominates all the other tests, for $K$ large enough.

\end{itemize}

More simulations showing the power of the $K$-NN test are given in Appendix \ref{sec:knnexperiments}.  These experiments show that the $K$-NN test is powerful against local alternatives when $K$ grows with $N$ (especially when $K=\alpha N$, for some $\alpha \in (0, 1)$), which supports the result in Proposition \ref{ppn:knn} and illustrates the advantage of using dense geometric graphs. Note that although the computation cost for the $K$-NN test increases with $K$, it is always polynomial in $N$, $K$, and the dimension $d$, making it far more efficient than tests based on depth functions. This makes the $K$-NN test desirable, both theoretically (non-trivial
efficiency) and in applications (easy computation and good finite sample power). 

\end{example}

\begin{figure*}[h]
\centering
\begin{minipage}[c]{0.495\textwidth}
\centering
\includegraphics[width=2.2in]
    {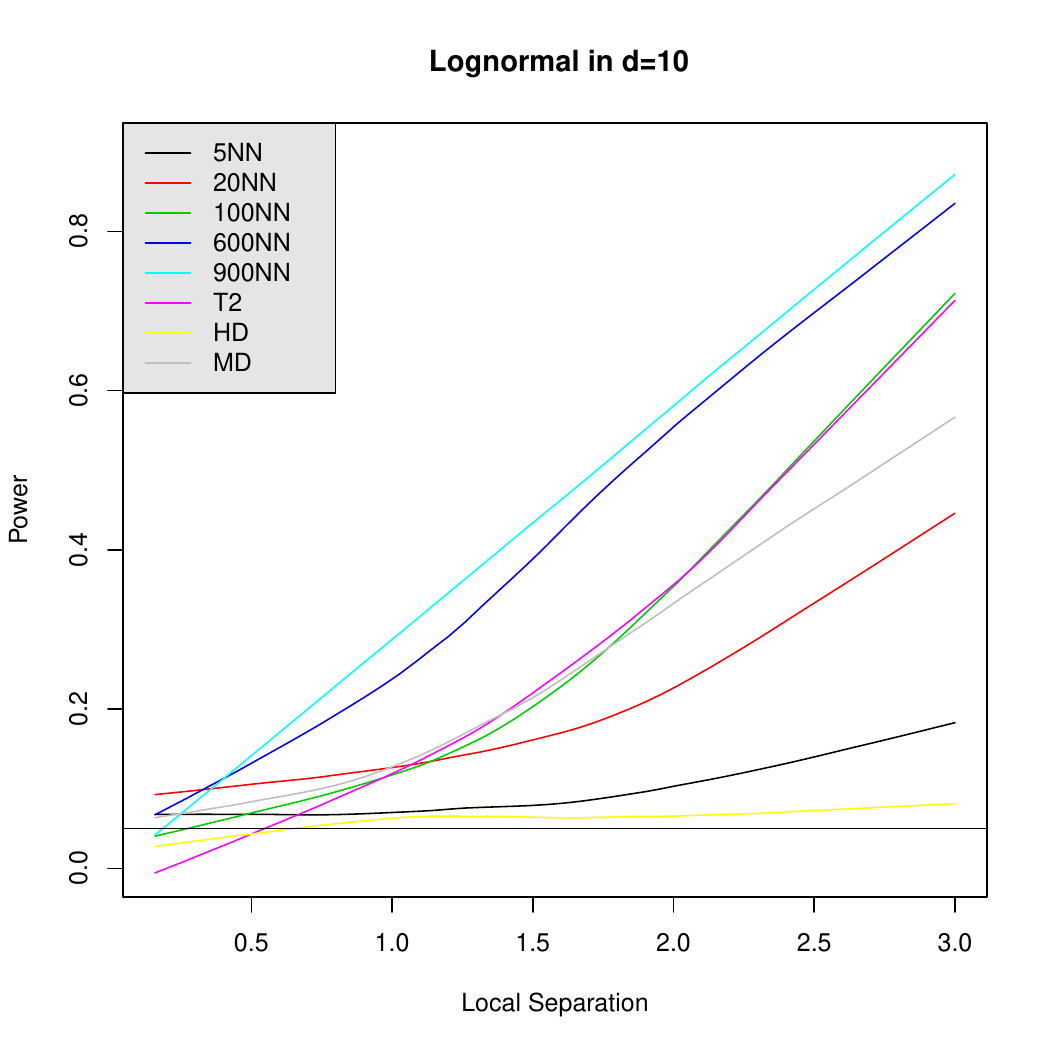}\\
\small{(a)}
\end{minipage}
\begin{minipage}[c]{0.495\textwidth}
\centering
\includegraphics[width=2.2in]
    {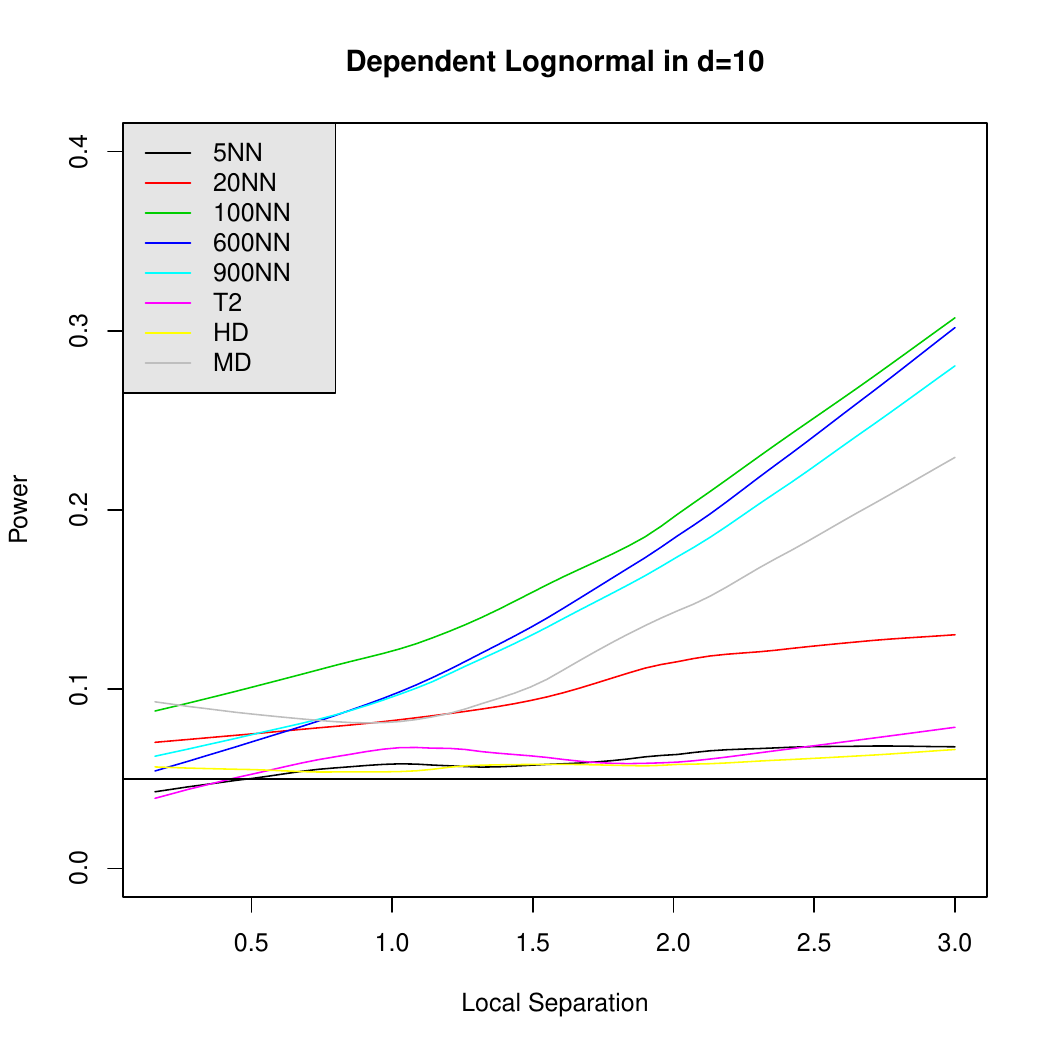}\\
\small{(b)}
\end{minipage}
\caption{\small{Power in the log-normal family in dimension 10 where the means of the respective normals differ by $\delta \cdot \bm 1/\sqrt N$, as a function of $\delta$. In (a) $\P_{\theta}\sim \exp(N(\theta, \mathrm I))$, the lognormal has independent coordinates, and in (b) $\P_{\theta}\sim \exp(N(\theta, \mathrm \Sigma))$, where $\Sigma=\mathrm I + \bm 1 \bm 1'$, the coordinates are dependent.}} 
\label{lognormal_nn}
\end{figure*}

\section{Application to Sensorless Drive Diagnosis Data Set}
\label{drivediag}

In this section we compare the performances of the tests based on the MST~\eqref{Tfr} and the halfspace depth (HD) \eqref{halfspace} on the sensorless drive diagnosis data set.\footnote{The data can be freely downloaded from the University of California, Irvine's \href{https://archive.ics.uci.edu/ml/datasets/Dataset+for+Sensorless+Drive+Diagnosis}{machine learning repository}.}  This dataset was used by Bayer et al. \cite{sensordriving} for sensorless diagnosis of an autonomous electric drive train, which is composed of a synchronous motor and several attached components, like bearings, axles, and a gear box. Damage to the drive causes severe disturbances and increases the risk of encountering breakdown costs. Monitoring the condition for such applications usually require additional sensors. Sensorless drive diagnosis instead directly uses the phase currents of the motor for determining the performance of the entire drive unit.

\begin{table}[h]

\begin{center}
\scriptsize
\begin{tabular}{c|ccccccc}
\hline
 \multirow{1}{*}{ Pairs }  & Test & PCA1 & PCA2  &  PCA3 & PCA48   \\
\hline
\multirow{2}{*}{(1, 2)}& HD & $1.498 \times 10^{-14}$ & $1.877\times 10^{-14}$ & $1.204 \times 10^{-12}$ & 0\\
 & MST   & 0.055 & 0.0256 &  $8.766 \times 10^{-5}$ & $1.615 \times 10^{-7}$ \\
\hline
\multirow{2}{*}{(1, 6)} & HD & $2.903 \times 10^{-11}$ & $9.47 \times 10^{-5}$ & 0.0011 &  $5.598\times 10^{-8}$\\
& MST & 0.397 & 0.618 & 0.124 & $8.85 \times 10^{-9}$\\
\hline
\multirow{2}{*}{(2, 6)} & HD & 0.038 & 0.266 & 0.166 &  0.002 \\
 & MST & 0.676 & 0.481 & 0.457 & 0.0005 \\
\hline
\multirow{2}{*}{(4, 9)} & HD & $1.366\times 10^{-7}$ & 0.0117 & 0.3743 &  $3.577\times 10^{-9
}$\\  
 & MST &   0.374 & 0.841 & 0.222  & $1.345\times 10^{-5}$ \\
\hline
\end{tabular}
\end{center}
\caption{\small{The asymptotic $p$-values of the two-sample tests based on the halfspace depth (HD) and the FR-test (MST) for the sensorless drive diagnosis data set.}}
\label{pdrive}
\end{table}

\normalsize

In the data collected, the drive train had intact and defective components, and current signals were measured  with a current probe and an oscilloscope on two phases under 11 different operating conditions, this means by different speeds, load moments and load forces. Thus, the dataset consists of 11 different classes, each class consists of 5319  data points of dimension 48. The 48 features were extracted using empirical mode decomposition (EMD) of the measured signals. The first three intrinsic mode functions (IMF) of the two phase currents and their residuals (RES) were used and broken down into sub-sequences. For each of this sub-sequences, the statistical features mean, standard deviation, skewness and kurtosis were calculated.

In order to detect the defects two-sample tests were performed on the $55={11\choose 2}$ pairs of data sets. The goal is to investigate which of the tests can successfully detect the defect, that is, reject the null hypothesis. Table~\ref{pdrive} shows the $p$-values of the tests based on the MST and the halfspace depth (computed using  the {\tt mdepth.TD} function in the {\tt R} package {\tt fda.usc}) for 4 such pairs.  For demonstrative purposes the data was projected onto the first principal component (PCA1), the first two principal components (PCA2), the first three principal components (PCA3), and PCA48 is the whole data set. The results show that in all the 4 pairs, the HD test has smaller $p$-values than the MST test for the first three principal components. In particular, the HD test performs significantly better in 1-dimension (PCA1). Even for higher principal components the HD test rejects at the 5\% level more often than the MST, illustrating that it is more  sensitive to detecting local changes compared to the MST, as shown in Sections~\ref{MSTAE} and~\ref{s2depth}. Both tests reject at the 5\% level for the whole data set, supporting the hypotheses that sensorless drive diagnosis is possible for the 4 pairs of defects considered above.

\section{Finite Sample Power in the High-Dimensional Regime}
\label{sec:hd}

We conclude with a simulation study of the finite sample power of the tests described above, when the sample size is comparable to the dimension. In this high dimensional regime the performance of the graph-based tests are quite different. We illustrate the performance of the various tests in  the 3 examples from Section \ref{examples}. As before, the level of significance is set at $\alpha=0.05$.

\begin{figure}
\begin{minipage}[l]{0.495\textwidth}
\centering
\scriptsize
\begin{tabular}{c||ccccc}
\hline
Dimension & FR & HD & MD & CF & $T^2$   \\
\hline
10 & 0.12 & 0.06 & 0.05 & 0.05 & 0.46  \\
30 &  0.28 & 0.09 & 0.05 & 0.19 & 0.7  \\
50 & 0.24 & 0.12 & -- & 0.19 & 0.78  \\
70 & 0.41 & 0.21 & -- & 0.31 & 0.72  \\
100 & 0.4 & 0.18 & -- & 0.39 & -- \\
\hline
\end{tabular}\\
\vspace{0.05in}
\small{(a)}
\end{minipage}
\begin{minipage}[c]{0.495\textwidth}
\centering
\scriptsize
\begin{tabular}{c||cccccc}
\hline
Dimension & FR & HD & MD & $T^2$ & CovTest & CF   \\
\hline
10 &  0.05 & 0.2 & 0.37 & 0.06 & 0.01 & 0.21  \\
30 & 0.17 & 0.67 & 0.28 & 0.09 & 0.06 & 0.36  \\
50 & 0.14 & 0.85 & -- & 0.05 & -- & 0.57 \\
70 & 0.17 & 0.96 & --  & 0.06 & -- & 0.61 \\
100 & 0.29 & 0.99 & -- & -- & -- & 0.84 \\
\hline
\end{tabular}\\
\vspace{0.05in}
\small{(b)}
\end{minipage}
\caption{\small{Power of the different tests across increasing dimensions with samples sizes $N_1=60$ and $N_2=40$ for (a) the normal location problem with mean difference $0.2 \cdot \bm 1$, and (b) the spherical normal problem with scale difference 0.2.}}
\label{fig:HD_I}
\end{figure}

\begin{enumerate}

\item[(1)] $\P_{\theta}\sim N(\theta, \mathrm I)$, for $\theta \in \R^d$: The table in Figure \ref{fig:HD_I}(a) shows the empirical power (out of 100 repetitions) of the FR-test based on the MST, the HD test, the MD test, the CF test based on the MST, and the Hotelling's $T^2$ test, with $N_1=60$ samples from $\P_{0}$ and  $N_2=40$ samples from $\P_{\frac{2 \cdot \pmb 1}{\sqrt N}}$ (here $N=N_1+N_2=100$), for dimensions 10, 30, 50, 70, and 100. Here, the parametric Hotelling's $T^2$ test has the highest power for dimensions up to 70, but is degenerate for dimension $d=100$. The HD test has low power across across dimensions. The MD test is powerless in low dimensions and degenerate for higher dimensions. Among the non-parametric tests, the FR test and the CF test are the most powerful as dimension increases. For instance, in dimension $d=100$, the FR and the CF tests dominate all the other tests.

\item[(2)] $\P_{\sigma}\sim N(0, \sigma^2\mathrm I)$, for $\sigma > 0$:   The table in Figure \ref{fig:HD_I}(b) shows the empirical power (out of 100 repetitions) of the various tests, based on $N_1=60$ samples from $\P_{0}$ and  $N_2=40$  samples from $\P_{1+ \frac{2}{\sqrt N}}$, across different dimensions. As expected, the Hotelling's $T^2$ test and the \texttt{CovTest} have no power in this case. The MD test has reasonable power for low dimensions, but is powerless for dimensions greater than 50. The FR test has reasonable power which increases with dimension. The HD test and the CF test are the two most powerful tests in this case, both of which have power improving with dimension. The HD test is slightly better than the CF test, however, the computation cost of the HD test is much higher.

\begin{figure*}[h]
\begin{minipage}[l]{1.0\textwidth}
\centering
\scriptsize
\begin{tabular}{c||ccccc}
\hline
Dimension & FR & HD & MD & CF & $T^2$   \\
\hline
10 & 0.1 & 0.06 & 0.3 & 0.13 & 0.26  \\
30 &  0.26 & 0.11 & 0.45 & 0.31 & 0.5  \\
50 & 0.4 & 0.12 & -- & 0.34 & 0.59  \\
70 & 0.41 & 0.13 & -- & 0.55 & 0.46  \\
100 & 0.43 & 0.17 & -- & 0.6 & -- \\
\hline
\end{tabular}
\end{minipage}
\caption{\small{Power of the different tests across increasing dimensions with samples sizes $N_1=60$ and $N_2=40$ for log-normal location-scale problem, where the corresponding normal means  differ by $0.2 \cdot \bm 1$.}}
\label{fig:HD_II}
\end{figure*} 

\item[(3)] $\P_{\theta}\sim \exp(N(\theta, \mathrm I))$ for $\theta \in \R^d$, where the exponent is taken co-ordinatewise. The table in Figure~\ref{fig:HD_II}(a) shows the empirical power (out of 100 repetitions) of the various tests, based on $N_1=60$ samples from $\P_{0}$ and  $N_2=40$  samples from $\P_{\frac{2 \cdot \bm 1}{\sqrt N}}$, across different dimensions. Here, the MD test and the HD test are powerless. The FR test has good power as dimension increases, however, the CF test dominates all the other tests for moderate to high dimensions. 
\end{enumerate}

The experiments above show that tests based on geometric graphs, such as the FR and the CF tests, shine in moderate to high dimensions, making these procedures very useful for modern statistical applications.

\section{Discussion}
\label{sec:extgen}

The asymptotic efficiency and finite sample results of the various graph-based two-sample tests obtained above, illustrate the strengths and weaknesses of existing methods, and show us how combinatorial properties of the underlying graph effect the performance of the associated two-sample test, which can help us decide which test to use in practice.

Theoretical results show tests based on sparse geometric graphs are powerless against $O(N^{-\frac{1}{2}})$ alternatives, when the dimension is fixed and the sample size is large. However, these tests exhibit good power in finite sample simulations which improves with increasing dimension, especially if the two distributions differ in scale. This is because the detection thresholds in scale problems for tests based on geometric graphs gets closer and closer to the parametric detection rate of $O(N^{-\frac{1}{2}})$ as dimension increases (recall discussion in Example \ref{nscaleex}). This blessing of dimensionality facilitates the application of these tests in modern statistical problems. Moreover, the asymptotic efficiency these tests can be improved by increasing the density of the underlying geometric graph. For instance, the test based on the $K$-NN graph, where $K$-grows polynomially with $N$, has non-trivial Pitman efficiency (Proposition \ref{ppn:knn}), and often dominates the other tests (both parametric and non-parametric) in finite sample settings. Even though the computational cost increases with $K$, it is still polynomial in the sample $N$, the number of nearest neighbors $K$, and the dimension $d$, which makes the $K$-NN test the frontrunner for practical applications in the fixed dimension, large sample size regime. Another test which perform reasonably well in this regime is the test based on the Mahalanobis depth (MD), however this become computationally unstable in large dimensions.

In case the dimension is comparable with the sample size, the situation is quite different. Here, simulation results in Section \ref{sec:hd} show that tests based on geometric graphs, such the FR and the CF tests, are the overall winner, reinforcing findings in \cite{chenfriedman}. The test based on the halfspace depth (HD) also performs particularly well in detecting pure scale changes, but does not have good power in the other cases. \\

\small{\noindent\textbf{Acknowledgements.} The author is indebted to his advisor Persi Diaconis for introducing him to graph-based tests and for his encouragement and support. The author thanks Ery Arias-Castro, Sourav Chatterjee, Probal Chaudhuri, Jerry Friedman, Shirshendu Ganguly, Anil Ghosh, Susan Holmes, and David Siegmund for helpful comments. The author thanks Hao Chen and Nelson Ray for their help with datasets. The author also thanks the Editor, the Associate Editor, and the anonymous referees for providing many thoughtful comments, which greatly improved the quality of the paper.}

\newpage

\normalsize 

\appendix

\section{Proof of Theorem~\ref{ASYME}}
\label{varcovlimits}

This section describes the proof of Theorem~\ref{ASYME}. To this end, let $\{\P_{\theta}\}_{\theta \in \Theta}$ be a parametric family of distributions satisfying Assumption~\ref{dist4} at $\theta=\theta_1\in \Theta$. Suppose $\sX_{N_1}$ and $\sY_{N_2}$ are i.i.d. samples from $\P_{\theta_1}$ and $\P_{\theta_2}$, respectively, and consider the testing problem~\eqref{localpower2}.  Let $\cZ_{N}=\sX_{N_1}\cup \sY_{N_2}$ be the pooled sample. The log-likelihood ratio for the testing problem (\ref{localpower2}) is:
\begin{equation}
L_{N}:=\sum_{j=1}^{N_2} \log\frac{f\left(Y_j\big|\theta_1+\frac{h}{\sqrt N}\right)}{f(Y_j|\theta_1)}=\sum_{j=1}^N\log \frac{f(Z_j|\theta_1+\frac{h}{\sqrt N})}{f(Z_j|\theta_1)} \pmb 1\{c_{j}=2\},
\label{likelihood}
\end{equation}
where  $c_j$ is the label of $Z_j$ as in~\eqref{graph2z}. Since $\P_{\theta_1}$ is QMD, by Lehmann and Romano \cite[Theorem 12.2.3]{lr}, in the usual asymptotic regime~\eqref{pq}, $L_N \dto N\left(-\frac{q\langle h, \mathrm I(\theta_1)h\rangle}{2}, q\langle h, \mathrm I(\theta_1)h\rangle\right)$, where $\mathcal I(\cdot)$ is the Fisher information matrix.  This implies the joint distributions of $\sX_{N_1}$ and $\sY_{N_2}$ under $H_0$ and $H_1$ (as in~\eqref{localpower2}) are mutually contiguous \cite[Corollary 12.3.1]{lr}. Then by Le Cam's Third Lemma \cite[Corollary 12.3.2]{lr}, if  under the null $H_0$, 
\begin{equation}
\left(
\begin{array}{c}
\cR(\sG(\cZ_N))  \\
L_N    
\end{array}
\right)\dto N \left(\left(
\begin{array}{c}
0\\
-\frac{q\langle h, \mathrm I(\theta_1)h\rangle}{2}
\end{array}
\right),
\left(
\begin{array}{cc}
\sigma_1^2  &  \sigma_{12}  \\
\sigma_{12}  &  q\langle h, \mathrm I(\theta_1)h\rangle
\end{array}
\right)\right),
\label{ll_TN_joint}
\end{equation}
then under the alternative $H_1$, $\cR(\sG(\cZ_N))\dto N(\sigma_{12}, \sigma_1^2)$. Then the limiting power of the two-sample test based on $\sG$ with rejection region \eqref{rejregion} is given by $\Phi(z_\alpha-\frac{\sigma_{12}}{\sigma_1})$, where $\Phi(\cdot)$ is the standard normal distribution function. Therefore, by \eqref{effdefn} the asymptotic efficiency of the test statistic $\cR(\sG(\cZ_N))$ is $\frac{\sigma_{12}}{\sigma_1}$.

The above discussion implies, in order to compute the efficiency of  $\cR(\sG(\cZ_N))$, it suffices to derive the joint distribution of $\cR(\sG(\cZ_N))$ and $L_N$ under the null $H_0$. To begin with, observe that the joint distribution of the pooled sample $\sX_{N_1}\cup \sY_{N_2}$ can be described as follows: Let $\cZ_N=\{Z_1, Z_2, \ldots, Z_N\}$ be i.i.d. from $\P_{\theta_1}$. Select a random subset of size $N_1$ from $[N]:=\{1, 2, \ldots, N\}$ and label its elements 1 and the remaining elements 2. Then the joint distribution of the elements labelled 1 and 2 is the same as the joint distribution of $\sX_{N_1}$ and $\sY_{N_2}$, under the null.  The labels of $\cZ_N$ under the null distribution are dependent, and this often makes computations difficult. A convenient way to circumvent this problem is the {\it bootstrap distribution} of the labelings: Let $\cZ_{N}=\{Z_1, Z_2, \ldots, Z_{N}\}$ be i.i.d. samples from $\P_{\theta_1}$. Assign label $c_i\in \{1, 2\}$ to every element in $\cZ_N$ independently with probability $\frac{N_1}{N}$ or $\frac{N_2}{N}$, respectively. The moments of the statistic~\eqref{2sampleG1} can be easily calculated under the  bootstrap distribution, because of the independence of the labelings. Moreover, the null distribution can be recovered from the bootstrap distribution as follows: Let $B_{N}$ be the {\it bootstrap count}, the number of elements in $\cZ_N$ assigned label 1. Under the bootstrap distribution, the joint distribution of the elements labelled 1 and 2 conditioned on the event $\{B_N=N_1\}$ is precisely the null distribution of $(\sX_{N_1}, \sY_{N_2})$.

The proofs of Theorem~\ref{ASYME} and Corollary~\ref{ASYMEundirected} have several steps, which are organized as follows: 

\begin{itemize}

\item Section~\ref{varlimits} computes the limiting conditional variance of the statistic~\eqref{2sampleG1} under the bootstrap distribution and the variance condition~(Assumption \ref{varcondition}). 

\item The limiting conditional covariance of the statistic~\eqref{2sampleG1} and the log-likelihood ratiounder the bootstrap distribution and the covariance condition~(Assumption \ref{covcondition}) is calculated in Section~\ref{covlimits}. 

\item Details on recovering the null distribution from the joint bootstrap distribution of the statistic, the log-likelihood ratio, and the bootstrap count are given in Section~\ref{bootstrapjoint}. The proofs of Theorem~\ref{ASYME} and Corollary~\ref{ASYMEundirected} are then completed assuming the joint normality (Assumption \ref{normalcondition})

\item The proof of the joint normality of the statistic and the log-likelihood, under the normality condition N1 and N2, are given later in Section~\ref{pfN1} and~\ref{pfN2}, respectively.

\end{itemize}

\subsection{Limiting Conditional Variance}
\label{varlimits}

Recall that $\{\P_{\theta}\}_{\theta \in \Theta}$ is a parametric family of distributions satisfying Assumption~\ref{dist4} at $\theta=\theta_1\in \Theta$. Let $Z_1, Z_2, \ldots$ be i.i.d. samples from $\P_{\theta_1}$, and $\cF:=\sigma(\{Z_i\}_{i \in \N})$ the associated sigma algebra. The following lemma gives the limiting variance of the statistic $\cR(\sG(\cZ_N))$ (as defined in \eqref{2sampleG1}), conditional on $\cF$, using the variance condition (Assumption~\ref{varcondition}).

\begin{lem}Let $\sG$ be a directed graph functional in $\R^d$ such that the pair $(\sG, \P_{\theta_1})$ satisfies the variance condition with parameters $(\beta_0, \beta_0^+, \beta_1^\uparrow, \beta_1^\downarrow, \beta^+_1)$. Then under the bootstrap distribution, 
$$\Var(\cR(\sG(\cZ_N))|\cF)\pto \sigma_{11}^2,$$
where 
\begin{align}
\label{condvartau}
\sigma_{11}^2:=r\left\{ \frac{\beta_0}{2} + q \beta^{\uparrow}_1+ p \beta^{\downarrow}_1-\frac{r}{2}\left(\frac{\beta_0}{2}+\beta^{+}_0+\beta^{\uparrow} +\beta^{\downarrow}_1+\beta^{+}_1\right)\right\}.
\end{align}
\label{varR}
\end{lem}

\begin{proof}From~\eqref{2sampleG1}
\begin{eqnarray}
\cR(\sG(\cZ_N))&=&\sqrt N\left(T(\sG(\cZ_N))-\frac{N_1N_2}{N^2}\right)+o(1),
\label{2sampleG2b}
\end{eqnarray}
where $T(\sG(\cZ_N))$ is defined in~\eqref{graph2z}. Under the bootstrap distribution, the labels of the vertices are independent, and so,
\begin{align}
\Var(\cR(\sG(\cZ_N))|\cF) 
=&\frac{N a}{|E(\sG(\cZ_N))|}-\frac{2cN|E^+(\sG(\cZ_N))|}{|E(\sG(\cZ_N))|^2} +b^{\uparrow}\frac{2NT_2^{\uparrow}(\sG(\cZ_N))}{|E(\sG(\cZ_N))|^2}\nonumber \\
&+b^{\downarrow}\frac{2NT_2^{\downarrow}(\sG(\cZ_N))}{|E(\sG(\cZ_N))|^2} -\frac{2c N|T_2^{+}(\sG(\cZ_N))|}{|E(\sG(\cZ_N))|^2},
\label{varRcalculation}
\end{align}
where 
$$a=\frac{N_1N_2}{N^2}-\left(\frac{N_1N_2}{N^2}\right)^2, \quad c=\frac{N_1^2N_2^2}{N^4},$$ 
and 
$$b^{\uparrow}=\frac{N_1N_2^2}{N^3}-\left(\frac{N_1N_2}{N^2}\right)^2, \quad  b^{\downarrow}=\frac{N_1^2N_2}{N^3}-\left(\frac{N_1N_2}{N^2}\right)^2.$$
The terms in~\eqref{varRcalculation} converges in probability since the variance condition~\eqref{varcondition} holds with parameters $(\beta_0, \beta_0^+, \beta_1^\uparrow, \beta_1^\downarrow, \beta^+_1)$. Substituting the limiting values and simplifying the expression, the result follows.
\end{proof}

\subsection{Limiting Conditional Covariance}
\label{covlimits}

Recall the definition of the statistic $\cR(\sG(\cZ_N))$ from~\eqref{2sampleG1} and the log-likelihood~\eqref{likelihood}. As before, let $Z_1, Z_2, \ldots$ be i.i.d. samples from $\P_{\theta_1}$, and $\cF:=\sigma(\{Z_i\}_{i \in \N})$ the associated sigma algebra. Since $\P_{\theta_1}$ is QMD (Assumption~\ref{dist4}), by \cite[Theorem 12.2.3]{lr}, the log-likelihood~\eqref{likelihood} has a second-order Taylor expansion $L_N=\dot L_N-\frac{q \langle h, \mathrm I(\theta_1) \rangle h}{2}+o_{P}(1)$, where 
\begin{align}\label{etaN}
\dot L_N:=\frac{1}{\sqrt N}\sum_{i=1}^{N_2}\left\langle h, \frac{\grad f(Y_i|\theta_1)}{f(Y_i|\theta_1)}\right\rangle=\frac{1}{\sqrt N}\sum_{i=1}^N\langle h, \eta(Z_i, \theta_1)\rangle\pmb 1 \{c_i=2\}.
\end{align}
Therefore, to show (\ref{ll_TN_joint}) it suffices to prove 
\begin{equation}
\left(
\begin{array}{c}
\cR(\sG(\cZ_N))  \\
\dot L_N    
\end{array}
\right)\dto N \left(\left(
\begin{array}{c}
0\\
0
\end{array}
\right),
\left(
\begin{array}{cc}
\sigma_1^2  &  \sigma_{12}  \\
\sigma_{12}  &  q\langle h, \mathrm I(\theta_1)h\rangle
\end{array}
\right)\right),
\label{etaN_TN_joint}
\end{equation}
for some $\sigma_1>0$ and $\sigma_{12}$.

The following lemma gives the value of $\sigma_{12}$, that is, the limiting covariance of the statistic $\cR(\sG(\cZ_N))$ and $\dot L_N$ (as defined in \eqref{etaN}), conditional on $\cF$, using the covariance condition (Assumption~\ref{covcondition}).

\begin{ppn}\label{ppn:covlimit} Let $\sG$ be a directed graph functional in $\R^d$ such that the pair $(\sG, \P_{\theta_1})$ satisfies the covariance condition~\eqref{covcondition}. Then under the bootstrap distribution 
$$\Cov(\cR(\sG(\cZ_N)), \dot L_N|\cF)\pto \sigma_{12}.$$
with 
\begin{align}
\sigma_{12}=\frac{r}{2}\left(p\int \langle h, \grad  f(z|\theta_1)\rangle \lambda^{\downarrow}(z)\mathrm d z-q \int \langle h, \grad f(z,\theta_1)\rangle \lambda^{\uparrow}(z)\mathrm d z\right).
\label{sigma12}
\end{align}
\end{ppn}

\begin{proof} Recall the definition of the function $\psi(\cdot, \cdot)$ from~\eqref{graph2z}. Let $\nu_{N_1, N_2}=\frac{N_1N_2}{N^2}$. For $j \in [N]$, define 
\begin{align}
T_{j}=&\frac{\sqrt N}{|E(\sG(\cZ_N))|}\sum_{i=1, i\ne j} ^{N}\left(\psi(c_i, c_j)-\nu_{N_1, N_2} \right)\pmb 1\{(Z_i, Z_j)\in E(\sG(\cZ_N))\}\nonumber\\
=&T_j^{(1)}-T_j^{(2)},
\label{Qj}
\end{align}
where $$T_j^{(1)}=\frac{\sqrt N}{|E(\sG(\cZ_N))|} \sum_{i=1, i\ne j} ^{N}\psi(c_i, c_j) \pmb 1\{(Z_i, Z_j)\in E(\sG(\cZ_N)) \},$$ and $$T_j^{(2)}:=\frac{\sqrt N}{|E(\sG(\cZ_N))|} \nu_{N_1, N_2} d^{\downarrow}(Z_j, \sG(\cZ_N)),$$ where $d^{\downarrow}$ is the in-degree function.  

Note that $\cR(\sG(\cZ_N))=\sum_{j=1}^N T_j$, and with $\dot L_N=\frac{1}{\sqrt N}\sum_{i=1}^N\langle h, \eta(Z_i, \theta_1)\rangle\pmb 1 \{c_i=2\}$ as defined in~\eqref{etaN}
\begin{eqnarray}
\cR(\sG(\cZ_N)) \dot L_{N}=\Gamma_1-\Gamma_2+\Gamma_3-\Gamma_4,
\label{eq:cov_terms}
\end{eqnarray}
where 
\begin{align}
\Gamma_1:=&\frac{1}{\sqrt N}\sum_{j=1}^N T_j^{(1)}\langle h, \eta(Z_j, \theta_1)\rangle  \pmb 1\{c_{j}=2\}, \nonumber\\
\Gamma_2:=&\frac{1}{\sqrt N}\sum_{j=1}^N T_j^{(2)}\langle h, \eta(Z_j, \theta_1)\rangle  \pmb 1\{c_{j}=2\},\nonumber\\
\Gamma_3:=&\frac{1}{\sqrt N}\sum_{1\leq j\ne k\leq N}T_j^{(1)} \langle h, \eta(Z_k, \theta_1)\rangle\pmb 1\{c_{k}=2\},\nonumber\\
\Gamma_4:=&\frac{1}{\sqrt N}\sum_{1\leq j\ne k\leq N}T_j^{(2)} \langle h, \eta(Z_k, \theta_1)\rangle\pmb 1\{c_{k}=2\}.\nonumber
\end{align}

To get the result, we need to compute the conditional expectation of the 4 terms above. We begin with $\Gamma_1$,
\begin{eqnarray}\label{eq:cov_11}
\E(\Gamma_1|\cF) &=& \frac{\nu_{N_1, N_2} }{|E(\sG(\cZ_N))|} \sum_{1\leq i\ne j\leq N} \langle h, \eta(Z_j, \theta_1)\rangle \pmb 1\{(Z_i, Z_j)\in E(\sG(\cZ_N))\}\nonumber\\
&=&\nu_{N_1, N_2}\cdot \frac{1}{N}\sum_{j=1}^{N} \langle h, \eta(Z_j, \theta_1)\rangle \lambda^{\downarrow}(Z_j, \sG(\cZ_N)),
\end{eqnarray}
where $\lambda^{\downarrow}$ is as defined in~\eqref{lambdadefn}. Similarly, 
\begin{eqnarray}
\E(\Gamma_2|\cF)=\nu_{N_1, N_2}\frac{N_2}{N}\cdot \frac{1}{N}\sum_{j=1}^{N} \langle h, \eta(Z_j, \theta_1)\rangle \lambda^{\downarrow}(Z_j, \sG(\cZ_N)).
\label{eq:cov_21}
\end{eqnarray}
Therefore, taking the difference of (\ref{eq:cov_11}) and (\ref{eq:cov_21}),
\begin{eqnarray}
\E(\Gamma_1-\Gamma_2|\cF)=\overline\nu_{N_1, N_2} \cdot \frac{1}{N}\sum_{j=1}^{N} \langle h, \eta(Z_j, \theta_1)\rangle \lambda^{\downarrow}(Z_j,\sG(\cZ_N)),
\label{eq:cov_1}
\end{eqnarray}
where $\overline \nu_{N_1, N_2}:=\nu_{N_1, N_2}\left(1-\frac{N_2}{N}\right)$. 

Now, consider the term $\Gamma_4$,
\begin{align}
& \E(\Gamma_4|\cF) \nonumber \\
=&\nu_{N_1, N_2}\frac{N_2}{N}\left\{\frac{1}{N}\sum_{1\leq j\ne k\leq N} \langle h, \eta(Z_k, \theta_1)\rangle \lambda^{\downarrow}(Z_j,\sG(\cZ_N)) \right\}\nonumber\\
=&\nu_{N_1, N_2}\frac{N_2}{N}\left\{\sum_{j=1}^N \langle h, \eta(Z_j, \theta_1)\rangle -\frac{1}{N}\sum_{j=1}^N \langle h, \eta(Z_j, \theta_1)\rangle \lambda^{\downarrow}(Z_j,\sG(\cZ_N)) \right\},
\label{eq:cov_22}
\end{align}
using $\sum_{i=1}^N \lambda^{\downarrow}(Z_i, \sG(\cZ_N))=N$. 

Finally, let $S:=\{i, j, k\in [N]: i\ne j, k\ne j\}$, and $S_1:=\{(i, j, k)\in S: i=k \ne j\}$. Then it is easy  to check that 
\begin{align}
\E(\Gamma_3|\cF)&=\nu_{N_1, N_2}\frac{N_2}{N}\cdot \frac{1}{|E(\sG(\cZ_N))|}\sum_{S\backslash S_1}\pmb 1\{(Z_i, Z_j)\in \sG(\cZ_N)\} \langle h, \eta(Z_k, \theta_1)\rangle \nonumber\\
&=\nu_{N_1, N_2}\frac{N_2}{N}\cdot \frac{1}{|E(\sG(\cZ_N))|}\sum_{1 \leq i \ne j \leq N}\sum_{k \ne i, j}\pmb 1\{(Z_i, Z_j)\in \sG(\cZ_N)\} \langle h, \eta(Z_k, \theta_1)\rangle \nonumber\\
&=\nu_{N_1, N_2}\frac{N_2}{N} \sum_{j=1}^N \langle h, \eta(Z_j, \theta_1)\rangle  -\nu_{N_1, N_2}\frac{N_2}{N} \left(\frac{1}{N}\sum_{i=1}^N\langle h, \eta(Z_i, \theta_1)\rangle\lambda^{\uparrow}(Z_i, \sG(\cZ_N))\right)\nonumber\\
&~~~~~~~~~-\nu_{N_1, N_2}\frac{N_2}{N} \left(\frac{1}{N}\sum_{i=1}^N\langle h, \eta(Z_i, \theta_1)\rangle\lambda^{\downarrow}(Z_i, \sG(\cZ_N))\right).
\label{eq:cov_24}
\end{align}
Subtracting (\ref{eq:cov_22}) from (\ref{eq:cov_24}) and adding (\ref{eq:cov_1}), and using the covariance condition~\ref{covcondition} gives the desired result.  
\end{proof}

\subsection{The Joint Null Distribution}
\label{bootstrapjoint}

The joint distribution of the test statistic and $\dot L_N$ as in~\eqref{etaN_TN_joint} can be derived from the joint bootstrap distribution of the statistic, $\dot L_N$ and the bootstrap count $B_N$. To this end, let
\begin{align}\label{wn}
W_N:=(\cR(\sG(\cZ_N)), \dot\ell_N, \overline B_N)^t,
\end{align}
where $\cR(\sG(\cZ_N))$ is the test statistic, $\dot\ell_N:=\dot L_N-\E(\dot L_N|\cF)$ the conditionally centered score function~\eqref{etaN}, and 
\begin{equation}\label{eq:BN}
\overline B_N:=\frac{B_N-N_1}{\sqrt N}=\frac{1}{\sqrt N}\sum_{i=1}^N\left\{\pmb 1\{c_i=1\}-\frac{N_1}{N}\right\},
\end{equation}
is the centered bootstrap count.  Let $\Sigma_N$ be the covariance matrix of $W_N$. 

Denote by $\Phi_s:\R^d\rightarrow \R$ the distribution function of a $s$-dimensional multivariate $N(0, \mathrm I_s)$. We begin by assuming that under the bootstrap distribution, 
\begin{equation}
|\P(\Sigma_N^{-\frac{1}{2}}W_N\leq x|\cF) -\Phi_3(x)|\pto 0,
\label{NORMALASSUMPTION}
\end{equation}
for all $x \in \R^d$. Then the following theorem allows us to move to the null distribution from the bootstrap distribution.

\begin{lem}\label{joint} Let $\sG$ be a directed graph functional such that the pair $(\sG, \P_{\theta_1})$ satisfies the variance condition with  parameters $(\beta_0, \beta_0^+, \beta_1^\uparrow, \beta_1^\downarrow, \beta^+_1)$, the covariance condition~\eqref{covcondition}, and the normality condition \eqref{normalcondition}. Then, if \eqref{NORMALASSUMPTION} holds, then under the null distribution 
\begin{equation}\label{jointbootstrap}
U_N:=\left(
\begin{array}{c}
\cR(\sG(\cZ_N)) \\
\dot L_N
\end{array}
\right)\dto N\left(\left(
\begin{array}{c}
0 \\
0 
\end{array}
\right),\left(
\begin{array}{cc}
\sigma_{1}^2 &  \sigma_{12} \\
\sigma_{12}  &   q  \langle h, \mathrm I(\theta_1) h\rangle \\
\end{array}
\right) \right),
\end{equation}
where $\sigma_1^2:=\sigma_{11}^2-\frac{1}{2}r(1-2r)$, and $\sigma_{11}^2$ and $\sigma_{12}$ are as \eqref{condvartau} and \eqref{sigma12}, respectively. 
\end{lem}

\begin{proof}  To begin with, recall \eqref{eq:BN} and observe that $\Var(\overline B_{N}|\cF)=\frac{N_1}{N}(1-\frac{N_1}{N})\rightarrow p(1-p):=\sigma_{33}^2$. 

Next, recall that 
\begin{equation}
\dot\ell_N:=\dot L_N-\E(\dot L_N|\cF)=\frac{1}{\sqrt N}\sum_{i=1}^N \langle h, \eta(Z_i, \theta_1) \rangle \left\{\pmb 1\{c_{i}=2\}-\frac{N_2}{N}\right\}.
\label{scorecentered}
\end{equation}
This implies, $\Var(\dot\ell_N|\cF)\pto  q(1-q)\langle h, \mathrm I(\theta)h\rangle:=\sigma_{22}^2$,  
since $\frac{1}{N}\sum_{i=1}^N \langle h, \eta(Z_i, \theta)\rangle ^2 \pto \langle h, \mathrm I(\theta)h\rangle$ by the weak law of large numbers. Also, if $\rho:=\Cov(\pmb 1\{c_1=1\}, \pmb 1\{c_1=2\})$, then 
$\Cov(\dot\ell_N, \overline B_{N}|\cF)=\rho\cdot \frac{1}{N}\sum_{i=1}^N\langle h, \eta(Z_i, \theta)\rangle\pto 0$. 

Finally, note that 
\begin{align}
\Cov(\mathcal  R(\sG(\cZ_N)),\overline B_{N}|\cF)=T_1+T_2-\frac{N_1^2N_2}{N^2},\label{sigma13}
\end{align}
where $T_1=\frac{1}{|E(\sG(Z_N))|} \E\left(\sum_{i\ne j}\pmb 1\{(Z_i, Z_j)\in E(\sG(\cZ_N))\} \psi(c_i, c_j)\Big| \cF\right)=\frac{N_1N_2}{N^2}$ and
\begin{align*}
T_2=&\frac{1}{|E(\sG(Z_N))|} \E\sum_{i_1\ne i_2,j\ne i_1} \pmb 1\{(Z_{i_1}, Z_j)\in E(\sG(\cZ_N))\} 1\{c_{i_1}=1, c_{i_2}=1, c_j=2\}\nonumber\\
=&(N-2)\frac{N_1^2N_2}{N^3}.
\end{align*}
Substituting the expressions for $T_1$ and $T_2$ in~\eqref{sigma13} gives $\Cov(\mathcal  R(\sG(\cZ_N)),\overline B_{N}|\cF)\pto  \frac{1}{2}r\left(1-2p\right):=\sigma_{13}$. 

Combining the above results, and using Lemma \ref{varR} and Proposition \ref{ppn:covlimit}, it follows that
\begin{equation}
\Sigma_N\pto \Sigma:=
\left(
\begin{array}{ccc}
\sigma_{11}^2  &  \sigma_{12} &  \sigma_{13} \\
\sigma_{12}  & \sigma_{22}^2  & 0  \\
\sigma_{13}  & 0  & \sigma_{33}^2  
\end{array}
\right)=
\left(
\begin{array}{cc}
\Sigma_{11}  &  \vec {\sigma}_{12} \\
\vec {\sigma}_{12}^t  &   \sigma_{33}^2 \\
\end{array}
\right),
\label{sigma}
\end{equation}
where $\vec{\sigma}_{12}=(\sigma_{12}, \sigma_{13})$ and $\Sigma_{11}$ is the leading principal $2\times 2$ sub-matrix of $\Sigma$. By (\ref{NORMALASSUMPTION}) and Slutsky's theorem, $W_N|\cF\dto N(0, \Sigma)$, under the bootstrap distribution.

If $W_N^1=(\cR(\sG(\cZ_N)), \dot\ell_N)^t$, then the distribution of $W_N^1$ conditional on $\{\overline B_N=0\}$ and $\cF$ converges to a $N(0, \Sigma_1)$, where 
\begin{equation}
\Sigma_1:=\Sigma_{11}- \frac{1}{\sigma_{33}^2}\vec {\sigma_{12}}\cdot \vec {\sigma_{12}}^t=\left(
\begin{array}{cc}
\sigma_{11}^2-\frac{\sigma^2_{13}}{\sigma_{33}^2}  &  \sigma_{12} \\
\sigma_{12}  &   \sigma_{22}^2 \\
\end{array}
\right)=\left(\begin{array}{cc}
\sigma_{1}^2  &  \sigma_{12}\\
\sigma_{12}  &   \sigma_{22}^2 \\
\end{array}
\right),
\label{sigma1}
\end{equation}
where the last step uses $\frac{\sigma^2_{13}}{\sigma_{33}^2}=\frac{1}{2}r\left(1-2p\right)^2=\frac{1}{2}r\left(1-2r\right)$. Therefore, under  the null distribution, 
\begin{equation}
|\P(\Sigma_1^{-\frac{1}{2}}W_N^1\leq x|\cF) -\Phi_2(x)|\pto 0.
\label{WN1}
\end{equation}

Finally, note that $\dot L_N=\dot\ell_N+\hat r_N$, where $\hat r_N:=\frac{N_2}{N}\cdot \frac{1}{\sqrt N}\sum_{i=1}^N \langle h, \eta(Z_i, \theta)\rangle$. By the central limit theorem,  $\hat r_N\dto V\sim N(0, q^2 \langle h, \mathrm I(\theta_1) h\rangle)$. Then by (\ref{WN1}), for $\vec t=(t_1, t_2)\in \R^2$, 
$$\E(e^{i \vec t U_N}|\cF)=e^{i t_2 \hat r_N}\E(e^{i  \vec  t W_N^1}|\cF)\dto e^{i t_2 V+\frac{1}{2}\vec t'\Sigma_1 \vec t}.$$ Therefore, by the Dominated Convergence Theorem 
\begin{align*}
\log \E(e^{i \vec t U_N})=\log \E\E(e^{i \vec t U_N}|\cF)\rightarrow & \frac{1}{2}\left(t_2^2 q^2 \langle h, \mathrm I(\theta_1) h\rangle +\frac{1}{2}\vec t\Sigma_1 \vec t\right) \nonumber\\
=&\frac{1}{2}\vec t \left(
\begin{array}{cc}
\sigma_{1}^2 &  \sigma_{12} \\
\sigma_{12}  &   q\langle h, \mathrm I(\theta_1)\rangle h\\
\end{array}
\right)  \vec t.
\end{align*}
and the proof is completed.
\end{proof}

\subsubsection{Completing the Proof of Theorem~\ref{ASYME} assuming~\eqref{NORMALASSUMPTION}}  

Let $(\sG, \P_{\theta_1})$ be as in Theorem~\ref{ASYME}. If~\eqref{NORMALASSUMPTION} holds, then, by Lemma~\ref{joint},~\eqref{jointbootstrap} holds. Therefore, by~\eqref{ll_TN_joint}, $\mathrm{AE}(\sG)=\frac{|\sigma_{12}|}{\sigma_1}$, where $\sigma_1^2:=\sigma_{11}^2-\frac{1}{2}r(1-2r)$, and $\sigma_{11}^2$ and $\sigma_{12}$ are as \eqref{condvartau} and \eqref{sigma12}, respectively.  Substituting these values and simplifying the expression, the result follows.

The joint normality in \eqref{NORMALASSUMPTION} is proved later in Section~\ref{pfN1} and Section~\ref{pfN2}, under the normality conditions N1 and N2 (Assumption~\ref{normalcondition}), respectively.

\subsubsection{Proof of Corollary~\ref{ASYMEundirected}}\label{pfundirected} Let $\sG$ be an undirected graph functional and $(\sG, \P_{\theta_1})$ be as in Corollary~\ref{ASYMEundirected}. Let $\sG_+$ be the directed version of $\sG$, obtained by replacing every undirected edge by two directed edges pointing in opposite directions. 

Note that $\frac{N}{|E(\sG_+(\cV_N))|}=\frac{N}{2|E(\sG(\cV_N))|}\pto \frac{\gamma_0}{2}$ and $\frac{N|E^+(\sG_+(\cV_N))}{|E(\sG_+(\cV_N))|^2}=\frac{N|E(\sG(\cV_N))|}{4|E(\sG(\cV_N))|}\pto \frac{\gamma_0}{4}$. Also, $\frac{NT^\downarrow(\sG(\cV_N))}{|E(\sG_+(\cV_N))|^2}=\frac{NT(\sG(\cV_N))}{4|E(\sG(\cV_N))|}\pto \frac{\gamma_1}{4}$, and similarly, $\frac{NT^\uparrow(\sG(\cV_N))}{|E(\sG_+(\cV_N))|^2}\pto \frac{\gamma_1}{4}$. Finally, $$\frac{N |T_2^+(\sG_+(\cV_N))|}{|E(\sG_+(\cV_N))|^2}=\frac{N\sum_{i=1}^N\{d^2(V_i, \sG(\cV_N))-d(V_i, \sG(\cV_N))\}}{4 |E(\sG(\cV_N))|^2}\pto \frac{\gamma_1}{2}.$$ 
Therefore, the pair $(\sG_+, \P_{\theta_1})$ satisfies the $(\frac{\gamma_0}{2}, \frac{\gamma_0}{4}, \frac{\gamma_1}{4}, \frac{\gamma_1}{4}, \frac{\gamma_1}{2})$ variance condition.  

Moreover, if the pair  $(\sG, \P_{\theta_1})$ satisfies the undirected covariance condition~\eqref{covundirected} with the function $\lambda(\cdot)$, then $\sG_+$ satisfies the directed covariance condition (Assumption~\ref{covcondition}) with functions $\lambda^\uparrow(\cdot)=\lambda^\downarrow(\cdot)=\frac{1}{2}\lambda(\cdot)$. 

Finally, since $\mathrm{AE}(\sG_+)=\mathrm{AE}(\sG)$, the result follows from Theorem~\ref{ASYME} after simplifications.

\section{Proof of~(\ref{NORMALASSUMPTION}) Under Condition N1}
\label{pfN1}

Recall the {\it normality condition} N1 from Assumption~\ref{normalcondition}. In this section we show the joint normality of $W_N:=(\cR(\sG(\cZ_N)), \dot\ell_N, \overline B_{N})^t$ under Condition N1.

\begin{ppn}\label{normalsparse} Let $\sG$ and $\P_{\theta_1}$ be as in Theorem~\eqref{ASYME}. Then \eqref{NORMALASSUMPTION} holds if the pair $(\sG, \P_{\theta_1})$ satisfies~Condition N1  and the variance condition~\eqref{varcondition}.
\end{ppn}

\subsection{Proof of Proposition \ref{normalsparse}} The proof uses the following version of Stein's method based on dependency graphs:

\begin{thm}[Chen and Shao \cite{stein_local}]
Let $\{V_i, i \in \cV\}$ be random variables indexed by the vertices of a dependency graph $H=(\cV, \cE)$ with maximum degree $D$. If  
$W=\sum_{i \in \cV}V_i$ with $\E(V_i)=0$, $\E V_i^2 = 1$, then 
\begin{equation}
\sup_{z\in \R} |\P(W \leq z)-\Phi(z)|\lesssim D^{10}\sum_{i\in \cV}\E|V_i|^3.
\label{eq:steinsmethod}
\end{equation}
\label{th:steinsmethod}
\end{thm}

To show~\eqref{NORMALASSUMPTION} it suffices to show that for every $\vec a=(a_1, a_2, a_3)^t\in \R^3$, the distribution of $\frac{\vec a^tW_N}{\sqrt{\vec a^t \Sigma_N \vec a}}\big|\cF \dto N(0, 1)$ under the bootstrap distribution. To this end, define 
$$T_1:=\frac{\sqrt N}{|E(\sG(\cZ_N))|} \sum_{j=1}^N\left( \psi(c_i, c_j)-\frac{N_1N_2}{N^2}\right) \pmb 1\{(Z_i, Z_j)\in E(\sG(\cZ_N))\},$$
$$T_2:=\frac{\langle h, \eta(Z_i, \theta_1)\rangle}{\sqrt N}\left(\pmb1\{c_i=1\}-\frac{N_1}{N}\right), \quad T_3:=\frac{1}{\sqrt N}\left(\pmb1\{c_i=2\}-\frac{N_2}{N}\right).$$ 

If $V_i:=\sum_{i=1}^3a_iT_i/\sqrt{\vec a^t \Sigma_N \vec a}$, then $\E(V_i)=0$, $\E(V_i^2)=1$, and $$\frac{\vec a^tW_N}{\sqrt{\vec a^t \Sigma_N \vec a}}=\sum_{i=1}^N V_i.$$ Construct a dependency graph $H=([N], E(H))$, with an $(i, j)\in E(H)$ whenever the graph distance $d(Z_i, Z_j)\leq 2$. Let $D$ be the maximum degree of $H$. By Condition N1, it follows that $D^{10}=O_P(1)$ and since the labels of $\cZ_N$ are independent under the bootstrap distribution, to apply Theorem \ref{th:steinsmethod} it suffices to bound
\begin{align}
\sum_{i=1}^N\E|V_i|^3\lesssim & |a_1|^3N \left(\frac{\sqrt N \Delta^\uparrow(\sG(\cZ_N))}{|E(\sG(\cZ_N))|}\right)^3\nonumber\\
+& |a_2|^3\frac{1}{N^{3/2}}\sum_{i=1}^N|\langle h, \eta(Z_i, \theta_1) \rangle|^3+\frac{|a_3|^3}{N^{1/2}}=o_P(1).
\label{momentbound}
\end{align}
The first term is $o_P(1)$ by Condition N1, and the second term is $O_P(N^{-\frac{1}{2}})$ by the finiteness of the third moment of the score function (Assumption~\ref{dist4}). Theorem \ref{th:steinsmethod} and (\ref{momentbound}) then gives
\begin{eqnarray}
\left|\P\left(\vec a^tW_N/\sqrt{\vec a^t \Sigma_N \vec a} \leq z \big| \cF\right)-\Phi(z)\right|\stackrel{P}\rightarrow 0,
\end{eqnarray} 
and \eqref{NORMALASSUMPTION} follows.

\section{Proof of~\eqref{NORMALASSUMPTION} Under Condition N2}
\label{pfN2}

Let $W=(W_1, W_2, \cdots, W_d)^t\in \R^d$ be a random vector. The following version of multivariate Stein's method will be used to prove~~\eqref{NORMALASSUMPTION} under condition N2.

\begin{thm}[Reinert and R\"ollin \cite{rr_stein}]
Assume that $(W, W')$ is an exchangeable pair of $\R^d$-valued random vectors such that
$\E W= 0$ and $\E W W^t= \mathrm \Sigma$, with $\Sigma \in \R^{d\times d}$ symmetric and positive definite. Suppose further that 
\begin{equation}
\E(W-W'|W')=\lambda  W.
\label{multistein2}
\end{equation}
Then, if $Z$ has $d$-dimensional standard normal distribution, and every three times differentiable function
$h$
\begin{equation}
|\E h(W)-\E h(\Sigma^{\frac{1}{2}}Z)|\leq \frac{||h''||_\infty}{4}A+\frac{||h'''||_\infty}{12}B,
\label{errorbound}
\end{equation}
where
\begin{equation}
A:=\frac{1}{\lambda}\sum_{i=1}^d\sum_{j=1}^d\sqrt{\Var(\E(W_i'-W_i)(W_j'-W_j)|W)}
\label{A}
\end{equation}
and
\begin{equation}
B:=\frac{1}{\lambda}\sum_{i=1}^d\sum_{j=1}^d\sum_{k=1}^d\E|(W_i'-W_i)(W_j'-W_j)(W_k'-W_k)|
\label{B}
\end{equation}
\label{multivariatestein}
\end{thm}

\begin{ppn}
\label{cltdense}
Let $\sG$ and $\P_{\theta_1}$ be as in Theorem~\eqref{ASYME}. Then \eqref{NORMALASSUMPTION} holds if the pair $(\sG, \P_{\theta_1})$ satisfies the normality Condition N2 and the variance condition~\eqref{varcondition}.
\end{ppn}

\subsection{Proof of Proposition \ref{cltdense}}

Recall the definition of $\cR(\sG(\cZ_N))$ from~\eqref{2sampleG1} and let $$W=(W_1, W_2, W_3)^t:=(\cR(\sG(\cZ_N)), \dot\ell_N, \overline B_{N})^t,$$ as in~\eqref{wn}. Suppose $(c_1', c_2', \cdots, c_N')$ is an independent copy of the labeling vector under the bootstrap labeling  distribution.  To define an exchangeable pair choose a random index $\mathrm I\sim \dU(\{1, 2, \ldots, N\})$ and replace the label of $c_{\mathrm I}$ by $c_{\mathrm I}'$. Denote the corresponding random variables by $$W'=(W_1', W_2', W_3')^t:=(\cR(\sG(\cZ_N))', \dot\ell_N', \overline B_{N}')^t.$$ Note that $(W, W')$ is an exchangeable pair. 

Recall $\psi(\cdot, \cdot)$ as defined in~\eqref{graph2z}. Define $$T_1(c_i, c_i', c_j):=\pmb 1\{(Z_i, Z_j)\in E(\sG(\cZ_N))\}(\psi(c_i, c_j)-\psi(c_i', c_j)),$$
and 
$$T_2(c_i, c_i', c_j):=\pmb 1\{(Z_j, Z_i)\in E(\sG(\cZ_N))\} (\psi(c_j, c_i)-\psi(c_j, c_i')),$$
and $T_0(c_{\mathrm I}, c_{\mathrm I}', c_j)=T_1(c_{\mathrm I}, c_{\mathrm I}', c_j)+T_2(c_{\mathrm I}, c_{\mathrm I}', c_j)$. Note that 
\begin{eqnarray}
W_1-W_1'&=&\frac{\sqrt N}{|E(\sG(\cZ_N))|}\left(\sum_{j=1}^N T_1(c_{\mathrm I}, c_{\mathrm I}', c_j)+\sum_{j=1}^N T_2(c_{\mathrm I}, c_{\mathrm I}', c_j) \right)\nonumber\\
&=&\frac{\sqrt N}{|E(\sG(\cZ_N))|}\sum_{j=1}^N T_0(c_{\mathrm I},c_{\mathrm I}', c_j).
\label{w1diff}
\end{eqnarray}
Also, let $\delta_1(c_i, c_i')=\pmb1\{c_{i}=1\}-\pmb1\{c'_{i}=1\}$ and $\delta_2(c_i, c_i')=\pmb1\{c_{i}=2\}-\pmb1\{c'_{i}=2\}$. Then 
\begin{equation}
W_2-W_2'=\frac{\langle h, \eta(Z_\mathrm I, \theta_1)\rangle}{\sqrt N}\delta_2(c_{\mathrm I}, c_{\mathrm I}'),  \quad \text{and} \quad W_3-W_3'=\frac{\delta_1(c_{\mathrm I}, c_{\mathrm I}')}{\sqrt N}.
\label{w2w3diff}
\end{equation}

From~\eqref{w1diff} and \eqref{w2w3diff} it follows that $\E(W-W'|W')=\frac{1}{N}W$, that is, $\lambda=1/N$ in (\ref{multistein2}). Therefore, by (\ref{B}) 
\begin{eqnarray}
B&=& N\sum_{i=1}^3\sum_{j=1}^3\sum_{k=1}^3\E(|W_i-W_i'||W_j-W_j'||W_k-W_k'|)\nonumber\\
&\lesssim & N (\E|W_1-W_1'|^3+\E|W_2-W_2'|^3+ \E|W_3-W_3'|^3)\nonumber\\
&\lesssim &  \frac{1}{\sqrt N}\left(\frac{N\Delta(\sG(\cZ_N))}{|E(\sG(\cZ_N))|}\right)^3+\frac{1}{N^{3/2}} \sum_{i=1}^N|\langle h, \eta(Z_i, \theta_1)\rangle|^3+O_P\left(\frac{1}{\sqrt N}\right)\nonumber\\
&=&O_P(1/\sqrt N).
\label{Bbnd}
\end{eqnarray}
The second last step uses the inequalities $\left|\sum_{j=1}^N T_0(c_{\mathrm I},c_{\mathrm I}', c_j)\right|\lesssim \Delta(\sG(\cZ_N))$, $|\delta_1(c_{\mathrm I}, c_{\mathrm I}')|\leq 1$, $|\delta_2(c_{\mathrm I}, c_{\mathrm I}')|\leq 1$, and the last step uses the normality condition (b).

To control (\ref{A}) it suffices to bound the variances of the following six quantities:
\begin{align*}
A_\alpha:=\frac{1}{\lambda}\E((W_\alpha-W_\alpha^t)^2|W,W'),~A_{\alpha, \beta}:=\frac{1}{\lambda}\E((W_\alpha-W_\alpha^t)(W_\beta-W_\beta^t)|W,W'),
\end{align*}
for $1\leq \alpha, \beta \leq 3$. The $A_\alpha$ terms  will be referred to as the {\it main terms}, and the $A_{\alpha,\beta}$ terms as the {\it cross terms}.

\subsubsection{Bounding Cross Terms} For the rest of the proof, let $e_N:=|E(\sG(\cZ_N))|$, number of edges in $\sG(\cZ_N)$, and $g(z):=\langle h, \eta(z, \theta_1)\rangle$, for $z\in \R^d$ and fixed $h \in \R^p$ and $\theta_1\in \Theta$. Then
\begin{align}
\Var(A_{23})&=\Var\left(\frac{1}{N}\sum_{i=1}^N g(Z_i) \delta_1(c_i, c_i') \delta_2(c_i, c_i')\right)\nonumber\\
&=\frac{1}{N^2}\sum_{i=1}^N |g(Z_i)|^2 \Var(\delta_1(c_i, c_i') \delta_2(c_i, c_i'))\nonumber\\
&=O_P\left(\frac{1}{N}\right).
\label{A23}
\end{align}
Next, we bound $\Var(A_{12})$. Note that
\begin{align}
\E A_{12}=&\frac{1}{e_N}\sum_{1\leq i\ne j\leq N} \langle h, \eta(Z_i, \theta_1)\rangle \E(\delta_2(c_i, c_i') T_0(c_i, c_i', c_j))\nonumber\\
=&\frac{1}{e_N}\sum_{1\leq i\ne j\leq N} \rho_{ij}g(Z_i),
\label{EA12}
\end{align}
where $\rho_{ij}:=\E(\delta_2(c_i, c_i') T_0(c_i, c_i', c_j))$. Also, $\E A^2_{12}=J_1+J_2+J_3$, where 
\begin{align}
J_1:=&\frac{1}{e_N^2}\sum_{1\leq i\ne j\leq N} |g(Z_i)|^2\E (\delta_2(c_i, c_i')^2 T_0^2(c_i, c_i', c_j)),
\label{J1}
\end{align}
\begin{align}
J_2:=&\frac{1}{e_N^2}\sum_{1\leq i\ne j_1\ne j_2\leq N} |g(Z_i)|^2\E \delta_2(c_i, c_i')^2 T_0(c_i, c_i', c_{j_1})T_0(c_i, c_i', c_{j_2}),
\label{J2}
\end{align}
and 
\begin{align}
J_3:=&\frac{1}{e_N^2}\sum_{i_1\ne j_1\ne i_2\ne j_2} g(Z_{i_1})g(Z_{i_2}) \rho_{i_1j_1} \rho_{i_2j_2}
\label{J3}
\end{align}
Note that $J_3\leq \E(A_{12})^2$, therefore to bound $\Var(A_{12})$ it suffices to bound $J_1$ and $J_2$. To this end, 
\begin{align}
J_1\lesssim &\frac{1}{e_N^2}\sum_{1\leq i\ne j\leq N} |g(Z_i)|^2 \E\{ \delta_2(c_i, c_i')^2 (T_1^2(c_i, c_i', c_j)+T_2^2(c_i, c_i', c_j))\}\nonumber\\
\lesssim &\frac{1}{e_N^2}\sum_{i=1}^N  |g(Z_i)|^2\left(|d^\downarrow(Z_i, \sG(\cZ_N))|+|d^\uparrow(Z_i, \sG(\cZ_N))|\right)\nonumber\\
\lesssim &\frac{N(\Delta(\sG(\cZ_N)))}{e_N^2}\cdot\frac{1}{N}\sum_{i=1}^N  |g(Z_i)|^2\nonumber\\
=&O_P(e_N^{-1})
\label{J1bound}
\end{align}
and 
\begin{eqnarray}
J_2 &\lesssim &\frac{N\Delta(\sG(\cZ_N))^2}{e_N^2}\cdot \frac{1}{N}\sum_{i=1}^N  |g(Z_i)|^2=O_P(N^{-1}).
\label{J2bound}
\end{eqnarray}
Combining (\ref{J1bound}) and (\ref{J2bound}) it follows that $\Var(A_{23})=o_P(1)$.

Finally, the cross term $A_{13}$ can be bounded similarly. All the steps go through verbatim with the function $g(\cdot)$ replaced by 1.

\subsubsection{Bounding Main Terms} To begin with, note that 
\begin{align}
\Var(A_2)=\Var\left(\frac{1}{N}\sum_{i=1}^N |g(Z_i)|^2\delta_2(c_i, c_i')^2\right)&=\frac{1}{N^2}\sum_{i=1}^N 
|g(Z_i)|^4\Var(\delta_2(c_i, c_i')^2)\nonumber\\
&=O_P\left(\frac{1}{N}\right).
\label{A2}
\end{align}
Similarly, $\Var(A_3)=\Var\left(\frac{1}{N}\sum_{i=1}^N \delta_1(c_i, c_i')^2\right)=O_P\left(\frac{1}{N}\right)$.

It remains to bound $\Var(A_1)$. This is a standard, but tiresome calculation. We sketch the main steps below. 
\begin{align}
\E A_1=&\frac{N}{e_N^2}\sum_{i=1}^N\E\left(\sum_{j=1}^NT_0(c_i, c_i', c_j)\right)^2\nonumber\\
=&\frac{N}{e_N^2}\sum_{i=1}^N\left(\sum_{j=1}^N\E T_0^2(c_i, c_i', c_j)+\sum_{j_1\ne j_2}\E\{T_0(c_i, c_i', c_{j_1})T_0(c_i, c_i', c_{j_2})\}\right)\nonumber\\
=&\frac{N}{e_N^2}\sum_{i=1}^N\left(\sum_{j=1}^N\kappa_{ij}+\sum_{j_1\ne j_2}\kappa_{ij_1j_2}\right),
\label{EA1}
\end{align}
where $\kappa_{ij}:=\E T_0^2(c_i, c_i', c_j)$ and $\kappa_{ij_1j_2}:=\E\{T_0(c_i, c_i', c_{j_1})T_0(c_i, c_i', c_{j_2})\}$.

Next, 
\begin{eqnarray}
A_1^2&=&\frac{N^2}{e_N^4}\left(\sum_{i=1}^N\sum_{j=1}^NT_0^2(c_i, c_i', c_j)+\sum_{i=1}^N\sum_{j_1\ne j_2}T_0(c_i, c_i', c_{j_1})T_0(c_i, c_i', c_{j_2})\right)^2\nonumber\\
&=& K_1+K_2+K_3,
\end{eqnarray}
where 
\begin{align}
K_1&=\frac{N^2}{e_N^4}\left(\sum_{i, j\in [N]} T_0^2(c_i, c_i', c_j)\right)^2,\nonumber\\
K_2&=\frac{N^2}{e_N^4} \left(\sum_{\substack{i\in [N]\\j_1\ne j_2}}T_0(c_i, c_i', c_{j_1})T_0(c_i, c_i', c_{j_2})\right)^2,\nonumber\\
K_3&=\frac{2N^2}{e_N^4}\sum_{i, j\in [N]} T_0^2(c_i, c_i', c_j)\cdot \sum_{\substack{i\in [N]\\j_1\ne j_2}} T_0(c_i, c_i', c_{j_1})T_0(c_i, c_i', c_{j_2}).
\label{K1K2K3}
\end{align}
Expanding the square, $\E K_1=\E K_{11}+\E K_{12}+\E K_{13}$, where 
\begin{align}
\E K_{11}=\frac{N^2}{e_N^4} \sum_{i, j \in [N]}\E T_0^4(c_i, c_i', c_j)& \lesssim \frac{N^2}{e_N^4} \sum_{i, j \in [N]}\{\E T_1^4(c_i, c_i', c_j)+\E T_2^4(c_i, c_i', c_j)\}\nonumber\\
&=O_P(N^3/e_N^4)=O_P(1/N),
\label{K11}
\end{align}
and
\begin{align}
\E K_{12}=\frac{N^2}{e_N^4} \sum_{\substack{i_1,\\ j_1\ne j_2 \in [N]}}\E \{T_0^2(c_{i_1}, c_{i_1}', c_{j_1})T_0^2(c_{i_1}, c_{i_1}', c_{j_2})\}&=\frac{N^3 \Delta(\sG(\cZ_N))}{e_N^4}\nonumber\\
&=O_P(1/N),
\label{K12}
\end{align}
and $K_{13}=K_1-(K_{11}+K_{12})$ is the sum over distinct indices $i_1, i_2, j_1, j_2$ over the terms $$\E \{T_0^2(c_{i_1}, c_{i_1}', c_{j_1})T_0^2(c_{i_2}, c_{i_2}', c_{j_2})\}.$$

Similarly, expanding the square $\E(K_2)=\E(K_{21})+\E(K_{22})+\E(K_{23})+\E(K_{24})+\E(K_{25})$, where 
\begin{align}
\E K_{21}=\frac{N^2}{e_N^4} \sum_{\substack{i\in [N]\\j_1\ne j_2}}\E\{T_0^2(c_i, c_i', c_{j_1})T_0^2(c_i, c_i', c_{j_2})\}\lesssim & \frac{N^3\Delta^2(\sG(\cZ_N))}{e_N^4}\nonumber\\
=&O_P(1/N),
\label{K21}
\end{align}
and
\begin{align}
|\E K_{22}|\leq &\frac{N^2}{e_N^4} \sum_{\substack{i\in [N]\\j_1\ne j_2\ne j_3}}\E|T_0(c_i, c_i', c_{j_1})T_0^2(c_i, c_i', c_{j_2})T_0(c_i, c_i', c_{j_3})|\nonumber\\
\lesssim & \frac{N^3\Delta^3(\sG(\cZ_N))}{e_N^4}=O_P(1/N).
\label{K22}
\end{align}
and
\begin{align}
|\E K_{23}|:=&\frac{N^2}{e_N^4} \sum_{\substack{i\in [N]\\j_1\ne j_2\ne j_3\ne j_4}}\E|T_0(c_i, c_i', c_{j_1})T_0(c_i, c_i', c_{j_2})T_0(c_i, c_i', c_{j_3})T_0(c_i, c_i', c_{j_4})|\nonumber\\
\lesssim & \frac{N^3\Delta^4(\sG(\cZ_N))}{e_N^4}=O_P(1/N),
\label{K23}
\end{align}
and 
\begin{align}
K_{24}:=&\frac{N^2}{e_N^4} \sum_{\substack{i_1\ne i_2\\j_1, j_2, j_3}}\E|T(c_{i_1}, c_{i_1}', c_{j_1})T(c_{i_1}, c_{i_1}', c_{j_2})T(c_{i_2}, c_{i_2}' c_{j_2})T(c_{i_2}, c_{i_2}', c_{j_3})|\nonumber\\
\lesssim &\frac{N^4\Delta(\sG(\cZ_N))^2}{e_N^4} =O_P(1/\Delta(\sG(\cZ_N)))=o_P(1).
\label{K24}
\end{align}
and $K_{25}=K_2-(K_{21}+K_{22}+K_{23}+K_{24})$ is the remaining term, which corresponds to summing over distinct indices $i_1, i_2, j_1, j_2, j_3, j_4$  of the quantity $$\E\{T_0(c_{i_1}, c_{i_1}', c_{j_1})T_0(c_{i_1}, c_{i_1}', c_{j_2})T_0(c_{i_2}, c_{i_2}', c_{j_3})T_0(c_{i_2}, c_{i_2}', c_{j_4})\}.$$

The term $K_3$ can be expanded similarly, and all the terms can be shown to be negligible expect the term $\E K_{34}$ which corresponds to summing over distinct indices $i_1, i_2, j_1, j_2, j_3$ of $$\E\{T_0^2(c_{i_1}, c_{i_1}', c_{j_1})T_0(c_{i_2}, c_{i_2}', c_{j_2})T_0(c_{i_2}, c_{i_2}', c_{j_3})\}.$$  Now, from~\eqref{EA1} it follows that $\E(K_{13}+K_{26}+K_{34})\leq \E(A_1)^2$. This together with~\eqref{K1K2K3}-\eqref{K24} show that $\Var(A_1)=o_P(1)$.

\section{Proof of Theorem \ref{MSTAE}}
\label{mstpf}

In this section the asymptotic efficiency of the FR-test based on the MST will be derived using the formula in Theorem~\ref{ASYME}, which entails verifying Assumption~\ref{varundirectedcond}.

To this end, let $\cV_N=\{V_1, V_2, \ldots, V_N\}$ be i.i.d. from $\P_{\theta_1}$ with density $f(\cdot |\theta_1)$, and $\cT$ the MST graph functional as in~Definition~\ref{mst}. In this case, $\frac{N}{|E(\cT(\cV_N))|}\to 1$, and by Henze and Penrose \cite[Theorem 1]{henzepenrose}
$$\frac{T_2(\cT(\cV_N))}{N}=\frac{1}{N}\sum_{i=1}^N {d(V_i, \cT(\cV_N))\choose 2}\pto \frac{1}{2}\Var(D_d)+1,$$
where $D_d:=d(0, \cT(\cP_1))$ and $\cT(\cP_1)$ the minimum spanning forest of $\cP_1$ as defined by Aldous and Steele \cite{aldous_steele}. Therefore, $(\cT, \P_{\theta_1})$ satisfies the $(1, \frac{1}{2}\Var(D_d)+1)$ variance condition. 

Moreover, the degree of a vertex in the MST of a set of points in $\R^d$ is bounded by a constant $B_d$, depending only on $d$ \cite{aldous_steele}. Therefore, $\Delta(\cT(\cV_N))=O_P(1)$ and the normality condition N1 in Assumption~\ref{normalcondition} is satisfied.

It remains to verify the covariance condition. By Henze and Penrose \cite[Proposition 1]{henzepenrose}, for all almost all $z \in \cK$,  
\begin{equation*}
\lim_{N\rightarrow \infty}\E(\lambda(z, \cT(\cV_N)))=2,
\label{degMST}
\end{equation*}
establishing~\eqref{lambda} with $\lambda(z)=2$. Now, fix $M>0$ and define 
\begin{equation*}
K_M(z):=|\langle h, \eta (z, \theta_1)\rangle|\pmb 1\{ |\langle h, \eta (z, \theta_1)\rangle| \leq M\},
\label{truncate}
\end{equation*}
and $\overline K_M(z):=|\langle h, \eta (z, \theta_1)\rangle|\pmb 1\{ |\langle h, \eta (z, \theta_1)\rangle| > M\}.$ Since $\lambda(V_i, \cT(\cV_N))\leq \Delta(\cT(\cV_N))$, $$\frac{1}{N}\sum_{i=1}^N \overline{K}_M(Z_i)\lambda(V_i, \cT(\cV_N))\leq B_d \frac{1}{N}\sum_{i=1}^N \overline{K}_M(V_i)\pto B_d\E(\overline{K}_M(V_1)),$$ as $N\rightarrow \infty$. Therefore, as $N\rightarrow \infty$ followed by $M\rightarrow \infty$
\begin{eqnarray}
\frac{1}{N}\sum_{i=1}^N \overline{K}_M(V_i)\lambda(V_i, \cT(\cV_N)\pto 0,
\label{Mg}
\end{eqnarray}
since $\E(|\langle h, \eta (V_1, \theta_1)\rangle|)<\infty$ by Assumption~\ref{dist4}.

Recall $\cV_{N-1}=\cV_N\backslash\{V_1\}$. Then $N\rightarrow \infty$,  
\begin{align}
\frac{1}{N}\sum_{i=1}^N\E\left(K_M(V_i)\lambda(V_i, \cT(\cV_N))\right)=&\E\left(K_M(V_1) \lambda(V_1, \cT(\cV_N))\right)\nonumber\\
=&\int K_M(z) \E\left(\lambda(z, \cT(\cV_{N-1}))\right)f(z|\theta_1)\mathrm d z\nonumber\\
\rightarrow &2\int K_M(z)f(z|\theta_1)\mathrm d z.
\label{Ml}
\end{align}
using the Dominated Convergence Theorem, since $|K_M(z) \E\left(\lambda(z, \cT(\cV_{N-1}))\right)|\leq MB_d$.

Let $V_i'$ be an independent copy of $V_i$ and $\cV_N^{(i)}=(V_1, V_2, \ldots, V_i', \ldots, V_N)$, for $i\in [N]$, and define $F_M(\cV_N):=\frac{1}{N}\sum_{i=1}^N K_M(V_i)\lambda(V_i, \cT(\cV_N))$. Replacing a point in $\cV_N$ with a new point only changes the MST in the neighborhood of the two points, and by \cite[Lemma 2.1]{steeleshepp} it follows that 
\begin{eqnarray}
|F_M(\cV_N)-F_M(\cV_N^{(i)})|\lesssim &\frac{M B_d}{N}.
\label{diffb}
\end{eqnarray} 
Then, by the Efron-Stein inequality \cite{esineq}, 
\begin{eqnarray}
\Var(F_M(\cV_N))&=&\frac{1}{2}\sum_{i=1}^N\E(|F_M(\cV_N)-F_M(\cV_N^{(i)})|^2)=O(1/N).
\label{varMl}
\end{eqnarray}
Combining (\ref{Ml}) and (\ref{varMl}), $$\frac{1}{N}\sum_{i=1}^N K_M(V_i)\lambda(V_i, \cT(\cV_N)) \pto 2\int K_M(z)f(z|\theta_1)\mathrm d z.$$ Now, as $M\rightarrow \infty$, $ 2\int K_M(z)f(z|\theta_1)\mathrm d z\rightarrow 2\E(\langle h, \eta(V_1, \theta_1)\rangle)=0$, by the Dominated Convergence Theorem. Therefore, when $N\rightarrow \infty$ followed by $M\rightarrow \infty$,
\begin{equation}
\frac{1}{N}\sum_{i=1}^N\E\left(K_M(V_i)\lambda(V_i, \cT(\cV_N))\right) \pto 0.
\label{Mlm}
\end{equation}
By (\ref{Mg}) and (\ref{Mlm}) it follows that $\cT$ satisfies the covariance condition~\eqref{covundirected} with the constant function $\lambda(z)=2$. Thus, using the formula~\eqref{effundirected} in Theorem \ref{ASYMEundirected}, $\mathrm{AE}(\cT)=0$.

\section{Proof of Theorem \ref{POWERSTABILIZE}}
\label{stabilizepf}

Given a graph functional $\sG$, $\varphi(z, \sG(\cZ))$ is a measurable $\R^+$ valued function defined for all locally finite set $\cZ\subset \R^d$ and $z\in \cZ$. If $z\notin \cZ$, then  $\varphi(z, \sG(\cZ)):=\varphi(z, \sG(\cZ\cup\{z\}))$. The function $\varphi$ is {\it translation invariant} if $\varphi(y+z, \sG(y+\cZ))=\varphi(z, \sG(\cZ))$, and {\it scale invariant} if $\varphi(a z, \sG(a\cZ))=\varphi(z, \sG(\cZ))$, for all $y \in \R^d$ and $a\in \R^+$.  Similar to stabilizing graph functionals, Penrose and Yukich \cite{py} defined stabilizing functions of  graph functionals as follows:

\begin{defn}(Penrose and Yukich \cite{py}) \label{defn:stabilize}
For any locally finite point set $\cZ\subset \R^d$ and any integer $m \in \N$
$$\overline \varphi(\sG(\cZ), M):=\sup_{N\in \N}\left(\esssup_{\substack{\cA\subset \R^d\setminus B(0, M)\\|\cA|=N}}\left\{\varphi(0, \sG(\cZ\cap B(0, M) \cup \cA))\right\}\right)$$
and
$$\underline \varphi(\sG(\cZ), M):=\inf_{N\in \N}\left(\essinf_{\substack{\cA\subset \R^d\setminus B(0, M)\\|\cA|=N}}\left\{\varphi(0, \sG(\cZ\cap B(0, M) \cup \cA))\right\}\right),$$
where the essential supremum/infimum is taken with respect to the Lebesgue measure on $\R^{dN}$. The functional $\varphi$ is said to {\it stabilize} $\sG(\cZ)$ if
\begin{equation}\label{eq:stabilize}
\liminf_{M\rightarrow \infty}\underline \varphi(\sG(\cZ), M)=\limsup_{M\rightarrow \infty}\overline\varphi(\sG(\cZ), M)=\varphi (0, \sG(\cZ)).
\end{equation}
\end{defn}

Recall that $\cP_\lambda$ is a Poisson process of rate $\lambda$ in $\R^d$. If $\sG$ is a scale invariant graph functional $\sG(\cP_\lambda)$ is isomorphic to $\sG(\cP_1)$, since $\cP_\lambda=\lambda^{-\frac{1}{d}} \cP_1$.  Therefore, $\sG$ stabilizes $\cP_\lambda$, if it stabilizes $\cP_1$. Moreover, from the definition of stabilization it is immediate that whenever a graph functional $\sG$ is stabilizing (as in Definition~\ref{stabilization}), the degree functional $d(z, \sG(\cZ))$, the degree of the vertex $z$ in  the graph $\sG(\cZ\cup \{z\})$, is also stabilizing (in the sense of Definition~\ref{defn:stabilize} above):

\begin{obs}\label{stable} If $\sG$ is a translation and scale invariant graph functional in $\R^d$ which stabilizes $\cP_1$, then the degree function $d(z, \sG(\cP_{\lambda}))$ stabilizes $\sG(\cP_\lambda)$, for any $\lambda>0$. \hfill $\Box$
\end{obs}

\subsection{Proof of Theorem~\ref{POWERSTABILIZE}}

Let $\sG$ be an undirected graph functional in $\R^d$. To apply Theorem~\ref{ASYME} the undirected variance and covariance conditions in Assumption~\ref{varundirectedcond} have to be verified:  (The proof for directed graph functionals is similar. For notational simplicity, the proof is given only for undirected graph functionals.)

Begin by showing that $(\sG, \mathbb P_{\theta_1})$ satisfies the undirected variance condition. To this end, let $Z_1, Z_2, \ldots, $ be i.i.d. samples from $\P_{\theta_1}$ with density $f(\cdot|\theta_1)$ in $\R^d$, and $\cZ_N=\{Z_1, Z_2, \ldots, Z_N\}$. 

\begin{obs}\label{varcondstabilize}Let $\sG$  and $\P_{\theta_1}$ be as in Theorem~\ref{POWERSTABILIZE}. Then $(\sG, \mathbb P_{\theta_1})$ satisfies the $(\gamma_0, \gamma_1)$-undirected variance condition with
\begin{equation}\label{gammastabilize}
\gamma_0=\frac{2}{\E(d(0, \sG(\cP_1)))} \quad \text{and} \quad \gamma_1= \gamma_0^2 \E{d(0, \sG(\cP_1))\choose 2}.
\end{equation}
\end{obs}

\begin{proof} Since the degree function stabilizes the Poisson process $\cP_\lambda$ (Observation~\ref{stable}), by \cite[Theorem 2.1]{py} and~\eqref{degcond} $$\frac{1}{N}\sum_{i=1}^N d(Z_i, \sG(\cZ_N))\pto \int \E(d(0, \sG(\cP_{f(z|\theta_1)})))f(z|\theta_1)\mathrm d z=\E(d(0, \sG(\cP_1))),$$
since $\sG$ is scale invariant. This implies $\gamma_0=\frac{2}{\E(d(0, \sG(\cP_1)))}$. 

Similarly, by \cite[Theorem 2.1]{py} $$\frac{1}{N}\sum_{i=1}^N {d(Z_i, \sG(\cZ_N))\choose 2}\pto \E{d(0, \sG(\cP_1))\choose 2},$$
and~\eqref{gammastabilize} follows.
\end{proof}

To complete the proof of the theorem it suffices to show $(\sG, \mathbb P_{\theta_1})$ satisfies the undirected covariance condition with the constant function $\lambda(z)=2$, for $z\in \R^d$.

\begin{ppn}\label{stabilizelimit} Let $\sG$  and $\P_{\theta_1}$ be as in Theorem~\ref{POWERSTABILIZE}. Then $(\sG, \mathbb P_{\theta_1})$ satisfies the covariance condition~\eqref{covundirected} with $\lambda(z)=2$.
\end{ppn}

\subsubsection{Proof of Proposition~\ref{stabilizelimit}}

Recall $\lambda(Z_j, \sG(\cZ_N))=\frac{N d(Z_j, \sG(\cZ_N))}{|E(\sG(\cZ_N))|}$. Since $\frac{N}{|E(\sG(\cZ_N))|}\pto \gamma_0$ as in~\eqref{gammastabilize}, it suffices to show~\eqref{covundirected} with $\lambda$ replaced by the degree functional $d(\cdot, \sG(\cdot))$.

Fix a point $z \in \cK$ (recall that $\cK$ is the support of $f(\cdot|\theta_1)$ which does not depend on $\theta_1$, by assumption). Let $\tilde \cP_1$ be a homogeneous Poisson process of rate 1 on $\R^d \times [0,\infty)$. As in the proof of \cite[Proposition 3.1]{py} define coupled point processes $\tilde \cP(N)$, $\cZ'_{N-1}$, and $\cH_N^z$ and a random variable $\zeta_N$ as follows:  Let $\tilde \cP(N)$ be the image of the restriction of $\tilde \cP_1$ to the set
$$\{(x, t) \in \R^d \times [0,\infty) : t \leq N f (x|\theta_1)\},$$
under the projection $(x, t) \rightarrow x$. Then $\tilde \cP(N)$ is a Poisson process in $\R^d$ with intensity function $Nf(\cdot|\theta_1)$, consisting of $C_N$ random points with common density $f(\cdot|\theta_1)$. Discard $(C_N- (N -1))_+$ of these points chosen at random and add $((N - 1)-C_N)_+$ extra independent points with common density $f(\cdot|\theta_1)$. The resulting set of points 
$\cZ'_{N-1}$, has the same distribution as  $\cZ_{N-1}=\cZ_N\backslash \{Z_1\}$. Let $\cH_N^z$ be the restriction of $\tilde \cP_1$ to the set $\{(x, t) : t \leq N f (z|\theta_1)\},$ under the mapping $(x, t)\rightarrow N^{1/d}(x - z)$. Then $\cH_N^z$ is a homogeneous Poisson process on $\R^d$ of intensity $f(z|\theta_1)$. 

Define $\zeta(z)=d (0, \sG(\cH_N^z))$ as the degree function for the point process $\cH_N^z$. Since $\sG$ is translation invariant, 
\begin{equation}
\zeta(z)\stackrel{D}=d(0, \sG(\cP_{f(z|\theta_1)}))\stackrel{D}=d(0, \sG(f(z|\theta_1)^{\frac{1}{d}}\cP_1))=d(0, \sG(\cP_1)), 
\label{zeta}
\end{equation}
which does not depend on $z$.

Since $\sG$ stabilizes $\cP_1$, by Observation \ref{stable}, $\sG$ stabilizes $\cP_{\lambda}$, for all $\lambda \in (0, \infty)$. Therefore, $ d (z, \sG(\cZ'_{N-1}))$ stabilizes $\cP_{\lambda}$, for all $\lambda \in (0, \infty)$ (Observation \ref{stable}). Note that $d(z, \sG(\cZ'_{N-1}))=d(0, \sG(N^{\frac{1}{d}}(\cZ'_{N-1}-z)))$, and from \cite[Lemma 3.2]{py} it follows that $d (z, \sG(\cZ_{N-1}')) \pto \zeta(z)$. By uniformly integrability (since~\eqref{degcond} holds), $\E d (z, \sG(\cZ_{N-1})) \rightarrow \E\zeta(z)$, since $\cZ'_{N-1}$ and $\cZ_{N-1}$ have the same distribution. Therefore, by the Dominated Convergence theorem, 
\begin{align}
\E\left(\frac{1}{N}\sum_{j=1}^{N}\langle h, \eta(Z_j, \theta_1)\rangle d(Z_j, \sG(\cZ_N))\right)=&\E\left(\langle h, \eta(Z_1, \theta_1)\rangle d(Z_1, \sG(\cZ_N))\right)\nonumber\\
=&\int \langle h, \eta(z, \theta_1)\rangle \E d(z, \sG(\cZ_{N-1}))  f(z|\theta_1)\mathrm dz\nonumber\\ 
\rightarrow &\int \langle h, \grad  f(z|\theta_1)\rangle \E \zeta(z) \mathrm d z\nonumber\\
=&\E \zeta(z) \int \langle h, \grad  f(z|\theta_1)\rangle  \mathrm d z\nonumber\\
=&0,
\label{expecstabilize}
\end{align}
since by (\ref{zeta}) the distribution of $\zeta(z)$ does not depend on $z$. 

It remains to prove that the convergence in~\eqref{expecstabilize} is in probability. This follows if 
\begin{align} 
\Cov\left(\langle h, \eta(Z_{1}, \theta_1)\rangle d(Z_{1}, \sG(\cZ_N)), \langle h, \eta(Z_{2}, \theta_1)\rangle d(Z_{2}, \sG(\cZ_N))\right) \rightarrow 0.
\label{covupstabilize}
\end{align}
By (\ref{expecstabilize}) it suffices to show 
\begin{equation}
\E \left\{ \langle h, \grad f(Z_1| \theta_1)\rangle  \langle h, \grad f(Z_2| \theta_1)\rangle d(Z_1, \sG(\cZ_{N}))d(Z_2, \sG(\cZ_{N}))\right\}  \rightarrow 0.
\label{expprod1}
\end{equation}
To show this define coupled point processes $\cZ_{N-2}'$ (with the same distribution as $\cZ_{N-2}=\cZ_N\backslash \{Z_1, Z_2\}$), and independent Poisson processes $\cH_N^x$ and $\cH_N^y$  with intensity $\cP_{f(x)}$ and $\cP_{f(y)}$ in $\R^d$, respectively (as in \cite[Proposition 3.1]{py}). As before, denote $\zeta(x)=d(0, \sG(\cH_N^x))$ and $\zeta(y)=d(0, \sG(\cH_N^y))$. Then from the proof of \cite[Proposition 3.1]{py}, it follows that 
\begin{align}\label{exppprodconv}
d(z_1, \sG(\cZ_{N-1})) d(z_2, \sG(\cZ_{N-1}))\pto & \zeta(z_1) \zeta(z_2).
\end{align}
Now, using (\ref{degcond}) and the Cauchy-Schwarz inequality 
$$\E(d(z_1, \sG(\cZ_{N-1})) d(z_2, \sG(\cZ_{N-1})))^{s/2}< \infty,$$ for some $s>2$. Therefore, the LHS of~\eqref{exppprodconv} is uniformly integrable and the expectation converges:
\begin{align*}
\E (d(z_1, \sG(\cZ_{N}))d(z_2, \sG(\cZ_{N}))) \rightarrow&\E \zeta(z_1)  \E \zeta(z_2),
\end{align*}   
where the last step uses the independence of $\zeta(z_1)$ and $\zeta(z_2)$.   Then by the Dominated Convergence Theorem, the LHS of~\eqref{expprod1} converges to $$\E \zeta(z_1)\E \zeta(z_2) \int \int \langle h, \grad  f(z_1|\theta_1)\rangle \langle h, \grad  f(z_2|\theta_1)\rangle  \mathrm d z_1\mathrm d z_2=0,$$
and the covariance condition is verified.

\section{Proof of Proposition \ref{ppn:knn}}
\label{sec:pfknn}

We will prove this result using Corollary \ref{ASYMEundirected}, which entails verifying Assumption \ref{varundirectedcond}.  Note that the normality condition in Assumption \ref{varundirectedcond} has been already verified in \eqref{eq:KNN_N2condition}. Therefore, it suffices to verify the undirected variance and covariance conditions. To this end, assume that $K=K_N \rightarrow \infty$ and $\cV_N=\{V_1, V_2, \ldots, V_N\}$ are i.i.d. observations from $f(\cdot|\theta_1)$. Then, 
\begin{align*}
\frac{\E |E(\cN_{K_N}(\cV_N))|}{{K_N}N} =\frac{ \E d(V_1, \cN_{K_N}(\cV_N))}{2 {K_N}} &= \frac{1}{2}\int \frac{\E d(z, \cN_{K_N}(\cV_{N-1}))}{{K_N}} f(z|\theta_1) \mathrm d z \nonumber \\
& \rightarrow \frac{1}{2} \int \eta_0(z) f(z|\theta_1) \mathrm dz,
\end{align*}
where the last step uses \eqref{eq:z01} and the Dominated Convergence Theorem (since for all $z\in \R^d$, $d(z, \cN_{K_N}(\cV_{N-1})) \leq C_d K_N$). Now, to show convergence in probability, we need to bound the variance of $\frac{|E(\cN_{K_N}(\cV_N))|}{{K_N}N}$. To this end, let $V_i'$ be an independent copy of $V_i$ and $\cV_N^{(i)}=(V_1, V_2, \ldots, V_i', \ldots, V_N)$, for $i\in [N]$, and define $F(\cV_N):=\frac{1}{2 {K_N}N}\sum_{i=1}^N d(V_i, \cN_{K_N}(\cV_N))$. Note that adding or deleting a point changes the degree of a vertex by $O({K_N})$ and the degree of $O({K_N})$ vertices by 1. Therefore, 
\begin{eqnarray*}
|F(\cV_N)-F(\cV_N^{(i)})|\lesssim &\frac{C_d}{N}.
\end{eqnarray*}
Then, by the Efron-Stein inequality \cite{esineq}, 
\begin{eqnarray*}
\Var(F(\cV_N))&=&\frac{1}{2}\sum_{i=1}^N\E(|F_M(\cV_N)-F_M(\cV_N^{(i)})|^2)=O(1/N).
\end{eqnarray*} 
This shows 
\begin{align}\label{eq:KNNvar1}
\frac{|E(\cN_{K_N}(\cV_N))|}{{K_N}N} \stackrel{L_2} \rightarrow \frac{1}{2} \int \eta_0(z) f(z|\theta_1) \mathrm dz, 
\end{align} 
which implies $\frac{N}{|E(\sG(\cV_N))|} \pto 0$. Similarly, it can be argued that $$\frac{|T_2(\sG(\cV_N))|}{K_N^2 N}=\frac{1}{N}\sum_{i=1}^N \frac{1}{K_N^2} {d(V_i, \cN_{K_N}(\cV_N)) \choose 2} \pto \int \eta_1(z) f(z|\theta_1) \mathrm dz,$$ hence,  
\begin{align*}
\frac{N|T_2(\sG(\cV_N))|}{|E(\sG(\cV_N))|^2} = \frac{K_N^2 N^2}{|E(\sG(\cV_N))|^2}\cdot \frac{|T_2(\sG(\cV_N))|}{K_N^2 N} \pto  \gamma_1:=\frac{4 \int \eta_1(z) f(z|\theta_1) \mathrm dz}{\left(\int \eta_0(z) f(z|\theta_1) \mathrm dz \right)^2}.
\end{align*}
This shows, the $K$-NN test satisfies the $(0, \gamma_1)$-undirected variance condition in Assumption \ref{varundirectedcond}.

To verify the undirected covariance condition, note that by assumption \eqref{eq:z01} and \eqref{eq:KNNvar1},  
\begin{align}\label{eq:KNNcov_I}
\E \lambda(z, \cN_{K_N}(\cV_N))=\E \frac{Nd(z, \cN_{K_N}(\cV_N^z))}{|E(\cN_{K_N}(\cV_N^z))|} \rightarrow \frac{\eta_0(z)}{\frac{1}{2} \int \eta_0(z) f(z|\theta_1) \mathrm dz}.
\end{align}
Recall, $\cV_{N-1}=\cV_N\backslash\{V_1\}$. Then as $N\rightarrow \infty$,  
\begin{align}\label{eq:KNNcov_II}
\frac{1}{N}\sum_{i=1}^N\E\left(\langle h, \grad \eta(V_i, \theta_1) \rangle \lambda(V_i, \cN_{K_N}(\cV_N))\right)=&\E\left( \langle h, \grad \eta(V_1, \theta_1) \rangle  \lambda(V_1, \cT(\cV_N))\right)\nonumber\\
=&\int \langle h, \grad \eta(z, \theta_1) \rangle \E\left(\lambda(z, \cN_{K_N}(\cV_{N-1}))\right)f(z|\theta_1)\mathrm d z\nonumber\\
\rightarrow & \frac{\int \eta_0(z) \grad  f(z|\theta_1) \mathrm dz }{\frac{1}{2} \int \eta_0(z) f(z|\theta_1) \mathrm dz},
\end{align}
where the last step uses Dominated Convergence Theorem, since, for all $z \in \cK$,  
$$| \langle h, \eta(z, \theta_1) \rangle  \E\left(\lambda(z, \cN_{K_N}(\cV_{N-1}))\right)|\leq C_d | \langle h,  \eta(z, \theta_1) \rangle |,$$ and $\E | \langle h,  \eta(V_1, \theta_1) \rangle| < \infty$ by Assumption \ref{dist4}. Finally, by another application of the Efron-Stein inequality, as in the proof of Theorem \ref{MSTAE}, it follows that the convergence in \eqref{eq:KNNcov_II} is in probability, which establishes the covariance condition in Assumption \ref{varundirectedcond} with $\lambda(\cdot)$ as in \eqref{eq:KNNcov_I}.

\section{Proof of Theorem~\ref{POWERDEPTH}}
\label{depthpf}

Let $D$ be a good depth function (recall Definition~\ref{depthconditions}) and $\sG_D$ be the associated graph. Begin by showing that $(\sG, \mathbb P_{\theta_1})$ satisfies the variance condition (Assumption~\ref{varcondition}).

\begin{obs}\label{depthpower}
Let $\P_{\theta_1}$ be as in~Assumption~\ref{dist4} with distribution function $F_{\theta_1}$. Then $(\sG_D, \P_{\theta_1})$ satisfies the $(0, 0, \frac{2}{3}, \frac{2}{3}, \frac{2}{3})$-variance condition.
\end{obs}

\begin{proof} Let $\cV_N=\{V_1, V_2, \ldots, V_N\}$ be i.i.d. from $\mathbb P_{\theta_1}$. Since $D(\cdot, F_{\theta_1})$ is a continuous depth function ((A1) in Definition~\ref{depthconditions}), $\P(D(V_1, F_{\theta_1})=D(V_2, F_{\theta_1}))=0$. 
This implies that $T_2^\uparrow(\sG_D)=T_2^\downarrow(\sG_D)\sim \sum_{i=1}^N {N-i\choose 2}\sim \frac{N^3}{6}$, 
and $T_2^+(\sG_D)\sim \sum_{i=1}^N (i-1)(N-i)\sim  \frac{N^3}{6}$. Therefore, $(\sG_D, \P_{\theta_1})$ satisfies the variance condition~\ref{varcondition} with $\beta_0=\beta_0^+=0$  and $\beta^\uparrow=\beta^\downarrow=\beta^+=\frac{2}{3}$. 
\end{proof}

Note that the total degree of every vertex in the graph $\sG_D$ is $N-1$. Therefore, $\sG_D$ satisfies the normality condition N2 in Assumption~\ref{normalcondition}. Thus, it remains to verify the covariance condition.  Recall the definition of relative outlyingness $R(\cdot, F_{\theta_1})$ from~\eqref{outlyingness}.

\begin{lem}\label{eq:depthlimit} Let $\cV_N=\{V_1, V_2, \ldots, V_N\}$ be i.i.d. from $\mathbb P_{\theta_1}$ and $z \in \cK$. Then for any good depth function $D$~(recall Definition~\ref{depthconditions}), 
\begin{align}\label{lambdaD}
\lim_{N\rightarrow \infty}\E \lambda^{\uparrow}(x, \sG_D(\cV_N)) &=  2 R(x, F_{\theta_1}),\nonumber\\
\lim_{N\rightarrow \infty} \E \lambda^{\downarrow}(x, \sG_D(\cV_N)) &=2 (1-R(x, F_{\theta_1})).
\end{align}
Moreover,
\begin{eqnarray}\label{upD}
\frac{1}{N}\sum_{j=1}^{N}\langle h, \eta(V_j, \theta_1)\rangle\lambda^{\uparrow}(V_j, \sG_D) \pto 2 \int \langle h, \grad f(x|\theta_1)\rangle R(x, F_{\theta_1}) \mathrm dx,
\end{eqnarray}
and 
\begin{eqnarray}\label{downD}
\frac{1}{N}\sum_{j=1}^{N}\langle h, \eta(V_j, \theta_1)\rangle\lambda^{\downarrow}(V_j, \sG_D) \pto 2 \int \langle h, \grad f(x|\theta_1)\rangle (1-R(x, F_{\theta_1})) \mathrm dx.
\end{eqnarray}
\end{lem}

\begin{proof} Note that $d^{\uparrow}(x, \sG_D(\cV_N))=\sum_{i=1}^N\pmb 1 \{(x, V_i)\in  E(\sG_D(\cV_N)) \}$. This implies 
\begin{align}\label{upx}
\E \lambda^{\uparrow}(x, \sG_D(\cV_N))=&\frac{N(N-1)}{|E(\sG_D(\cV_N))|}\E(\pmb 1\{D(x, F_{N_1})\leq D(V_1, F_{N_1})\})\nonumber\\
=&2\E \mathrm I(x, V_1, F_{N_1}),
\end{align}
where $\mathrm I(x, V_1, H):=\pmb 1\{D(x, H)\leq D(V_1, H)\}$, for any distribution function $H$ in $\R^d$, and $F_{N_1}=\frac{1}{N_1}\sum_{i=1}^{N_1}\delta_{V_i}$ is the empirical measure.  It can be verified that 
\begin{align}
& |\mathrm I(x, V_1, F_{N_1})-\mathrm I(x, V_1, F_{\theta_1})| \nonumber\\
& \leq \pmb 1\{|D(x, F_{\theta_1})-D(V_1, F_{\theta_1})|\leq 2\sup_{x\in \R^d} |D(x, F_{N_1})-D(x, F_{\theta_1})|\}. \nonumber
\end{align}
Therefore, by (A2) and (A3) in Definition~\ref{depthconditions}, 
\begin{align}\label{D}
\E|\mathrm I(x, V_1, F_{N_1})-\mathrm I(x, V_1, F_{\theta_1})|\lesssim \E\left(\sup_{x\in \R^d} |D(x, F_{N_1})-D(x, F_{\theta_1})|\right)\rightarrow 0.
\end{align}
Thus, by (\ref{upx}), $\lim_{N\rightarrow \infty}\E \lambda^{\uparrow}(x, \sG_D(\cV_N))=2 R(x, F_{\theta_1})$. The limit of scaled in-degree $\E \lambda^{\downarrow}(x, \sG_D(\cV_N))$ can be obtained similarly.

This implies 
\begin{align}\label{expupD}
&\frac{1}{N}\sum_{j=1}^{N}\E\langle h, \eta(V_j, \theta_1)\rangle\lambda^{\uparrow}(V_j, \sG_D(\cV_N))\nonumber\\
&=\E\langle h, \eta(V_1, \theta_1)\rangle\lambda^{\uparrow}(V_1, \sG_D(\cV_N))\nonumber\\
&=\int \langle h, \eta(x, \theta_1)\rangle\E \lambda^{\uparrow}(x, \sG_D(\cV_{N-1}))f(x|\theta_1)\mathrm dx\nonumber\\
&\rightarrow 2 \int \langle h, \grad f(x|\theta_1)\rangle R(x, F_{\theta_1}) \mathrm dx,
\end{align}
where the last step uses the Dominated Convergence Theorem, since  $\lambda^\uparrow(x, \sG_D(\cV_N))\leq 2$.

To show~\eqref{upD}, it suffices to show that the variance of the LHS of~\eqref{upD} goes to zero. This follows if 
\begin{align}\label{covup}
 & \frac{1}{N^2}\sum_{j_1\ne j_2}\Cov\left(\langle h, \eta(V_{j_1}, \theta_1)\rangle \lambda^{\uparrow}(V_{j_1}, \sG_D(\cV_N)), \langle h, \eta(V_{j_2}, \theta_1)\rangle \lambda^{\uparrow}(V_{j_2}, \sG_D(\cV_N))\right)\rightarrow 0. 
\end{align}

To this end, note that by (\ref{D}),
\begin{align}\label{expxyup}
\E & \pmb 1\{D(x, F_{N_1})\leq D(V_1, F_{N_1}), D(y, F_{N_1})\leq D(V_2, F_{N_1})\}\nonumber\\
=&\E \left( \mathrm I(x, V_1, F_{N_1}) \mathrm I(y, V_2, F_{N_1})\right)\nonumber\\
=&\E \left( \mathrm I(x, V_1, F_{\theta_1}) \mathrm I(y, V_2, F_{N_1})\right)+o(1)\nonumber\\
=&\E \left( \mathrm I(x, V_1, F_{\theta_1}) \mathrm I(y, V_2, F_{\theta_1})\right)+o(1)\nonumber\\
\rightarrow &\E \mathrm I(x, V_1, F_{\theta_1}) \E \mathrm I(y, V_2, F_{\theta_1}).
\end{align}
This implies that 
$$\Cov(\pmb 1\{D(x, F_{N_1})\leq D(V_1, F_{N_1})\}, \pmb 1\{ D(y, F_{N_1})\leq D(V_2, F_{N_1})\})\rightarrow  0,$$
and hence, $\Cov(d^{\uparrow}(x, \sG_D(\cV_N^{y})), d^{\uparrow}(y, \sG_D(\cV_N^{x})))
=O(N)$. Then (\ref{covup}) follows by the Dominated Convergence theorem, completing the proof of (\ref{upD}). The analogous result for the indgree (\ref{downD}) can be proved similarly.
\end{proof}

The proof of Theorem~\ref{POWERDEPTH} can be easily completed by substituting the values obtained above in  Theorem~\ref{ASYME}. From Observation~\ref{depthpower}, the denominator of the formula in Theorem~\ref{ASYME} is $\sqrt{r/6}$. From Lemma~\ref{eq:depthlimit}, the numerator is $-r \int \langle h, \grad  f(z|\theta_1)\rangle R(z, F_{\theta_1})\mathrm d z$. This implies that $\mathrm{AE}(\sG_D)=-\sqrt{6r} \int \langle h, 
\grad  f(z|\theta_1)\rangle R(z, F_{\theta_1})\mathrm d z$, and the result follows. \qed \\

We conclude by computing the efficiency of depth-based tests in some special cases.  

\begin{remark}\label{ex:depthlocation}(Location Family) As in Example~\ref{nlocationex}, consider parametric family $\P_{\theta}\sim N(\theta, \mathrm I)$, $\theta \in 
\R^d$. Then $R_{MD}(y, F_\theta)= R_{D}(y, F_\theta)$, where $F_{\theta}$ is the  distribution function of $\P_{\theta}$, for any depth function $D$ which is affine-invariant and the satisfies strict monotonicity property (refer to \cite[Theorem 5.2]{liu_singh} for details). Note that $R_{MD}(y, F_\theta)=\P_{F_{\theta}}(X: MD(X, F_{\theta}) \leq MD(y, F_{\theta}))=\P(\chi^2_d>  (y-\theta)^t(y-\theta))$. Then $$\int_{\R^d} \langle h, \grad f(x|\theta)\rangle R(x, F_\theta) \mathrm dx= -\frac{1}{(2\pi)^{d/2}}\int_{\R^d} \langle h, z e^{-\frac{z^t z}{2}}\rangle \P(\chi^2_d>  z^t z)\mathrm dz=0.$$ This implies that the asymptotic efficiency of any depth-based test is zero for the normal location problem, as seen in Example~\ref{nlocationex}. In fact, the same argument shows that depth-based tests are powerless against location alternatives for {\it elliptical} distributions \cite{symdistbook}. 
\end{remark}

On the other hand, the univariate depth function $D(x, F)=F(x)$ does not satisfy the strict monotonicity property. For this reason, the Mann-Whitney test for the normal location problem has non-zero efficiency, unlike the strictly monotonic depth functions.

\begin{remark}\label{mw} (Mann-Whitney $U$-Statistic) This is one of the most popular univariate two-sample tests, which corresponds to~\eqref{QFGsample} with $D(x, F)=F(x)$. This implies, $R(x, F)=1-F(x)$. Then for $\theta=\theta_1\in \Theta$, by Theorem~~\ref{POWERDEPTH}, the asymptotic efficiency of the Mann-Whitney test is $-\sqrt{6r} \cdot h \int \frac{\partial}{\partial \theta_1}f(x|\theta_1) R(x, F_{\theta_1}) \mathrm dx=-\sqrt{6r} \cdot h \int \frac{\partial}{\partial \theta_1}f(x|\theta_1) F_{\theta_1}(x) \mathrm dx= \sqrt{6r} \cdot h \int f^2(x|\theta_1) \mathrm dx$, where the last step uses integration by parts. 
\end{remark}

Another interesting exception is the zero efficiency of the test based on the halfspace depth in the lognormal location problem (see Figure \ref{lognormal}).

\begin{remark}\label{lognormal_halfspace} (Test Based on Halfspace Depth for Lognormal Location) It suffices to consider the 1-dimensional case (the result extends to the multivariate case, because $\exp(N(\theta, \mathrm I))$ is independent across coordinates). Let $\mu \in  \R$, and $F_\mu(t)=\Phi(\log t-\mu)$ be the distribution function and $f(t|\mu)=\frac{1}{t}\phi(\log t-\mu)$ the density of the lognormal $\exp(N(\mu, 1))$ (where $\Phi$ and $\phi$ denote the distribution function and density of the standard normal $N(0, 1)$, respectively). Next, note that the 1-dimensional halfspace-depth $HD(x, F_{\mu})=\min\{F_{\mu}(x), 1-F_{\mu}(x)\}$, which implies 
\begin{align}\label{eq:RHD}
R_{HD}(y, F_{\mu})&=2F_{\mu}(y) \bm 1\{F_{\mu}(y) \leq \frac{1}{2}\}+2(1-F_{\mu}(y)) \bm 1\{F_{\mu}(y) > \frac{1}{2}\} \nonumber \\
&=2F_{\mu}(y) \bm 1\{y \leq e^{\mu}\}+2(1-F_{\mu}(y)) \bm 1\{y >  e^{\mu}\}.
\end{align}
 Therefore, 
\begin{align}
\int_{0}^{e^\mu}\grad f(x|\mu)  F_{\mu}(x)  \mathrm dx& =-\int_{0}^{e^\mu} \frac{\log x-\mu}{x}\phi(\log x-\mu) \Phi(\log x-\mu)  \mathrm dx \nonumber \\
&=-\int_{-\infty}^0 z\phi(z) \Phi(z)  \mathrm dz \tag*{(substituting $z=\log x-\mu$)} \\
&=\int_0^\infty t\phi(t) (1-\Phi(t))  \mathrm dt \tag*{(substituting $t=-z$)} \\
&=-\int_{e^\mu}^\infty\grad f(x|\mu) (1- F_{\mu}(x))  \mathrm dx \nonumber.
\end{align}
Therefore, by \eqref{eq:RHD}, $\int_{\R}  \grad f(x|\mu)\rangle R_{HD}(x, F_\mu) \mathrm dx=0$,  which implies that the asymptotic efficiency of the test based on halfspace depth in the lognormal location problem is zero, as seen in the simulation in Figure \ref{lognormal}.
\end{remark}

\section{Proof of Theorem \ref{thm:frnew}}
\label{sec:pffrnew}

To begin with, define (similar to \eqref{wn})
\begin{align}\label{wnN}
W_N:=(\cR_1(\sG(\cZ_N)), \cR_2(\sG(\cZ_N)), \dot\ell_N, \overline B_N)^t,
\end{align}
where $\overline B_N$ is defined in \eqref{eq:BN} and $\dot\ell_N:=\dot L_N-\E(\dot L_N|\cF)$. Let $\Sigma_N$ be the variance-covariance matrix of $W_N|\mathcal F$ under the bootstrap distribution.

\begin{ppn}\label{jointNnew} Let $\sG$ be a undirected graph functional such that the pair $(\sG, \P_{\theta_1})$ satisfies the variance condition~\eqref{gamma0gamma1} with  parameters $(\gamma_0, \gamma_1)$ and the covariance condition~\eqref{covundirected}. Moreover, assume that under the bootstrap distribution 
\begin{equation}
|\P(\Sigma_N^{-\frac{1}{2}}W_N\leq x|\cF) -\Phi_4(x)|\pto 0,
\label{NORMALASSUMPTIONN}
\end{equation}
for all $x \in \R^d$. Then under the alternative $H_1$,  
\begin{equation}\label{eq:jointR12}
\Lambda_N^{-\frac{1}{2}}\begin{pmatrix}
\cR_1(\sG(\cZ_N)) \\ 
\cR_2(\sG(\cZ_N))
\end{pmatrix}\dto N\left( 
\Lambda^{-\frac{1}{2}}\begin{pmatrix}
\mu_1 \\
\mu_2
\end{pmatrix},  
\mathrm I \right),
\end{equation}
where $\mu=(\mu_1, \mu_2)^t$ and $\Lambda$ are as defined in Theorem \ref{thm:frnew}.  
\end{ppn}

Given the normality condition in Assumption \ref{varundirectedcond}, it is easy to verify \eqref{NORMALASSUMPTIONN} above, by arguments similar to those in Appendix \ref{pfN1} and Appendix \ref{pfN2}. Therefore, Theorem \ref{thm:frnew} is an immediate consequence of the above proposition, which is proved below.

\subsection{Proof of Proposition \ref{jointNnew}} As in the proof of Proposition \ref{jointNnew} we begin by computing the the limit of $\Sigma_N$ (the variance-covariance matrix of $W_N|\mathcal F$). 

\begin{lem}\label{lm:var_limit_N} Under the assumptions of Proposition \ref{jointNnew}, $\Sigma_N\pto  \Sigma_0:=((\sigma_{ij}))_{i, j \in [4]},$ 
where $\sigma_{ij}=\sigma_{ji}$ and 
\begin{itemize}
\item[(1)] $\sigma_{11}:=p^2((1-p^2) \gamma_0  + r \gamma_1 )$, $\sigma_{22}:=q^2 ((1-q^2) \gamma_0  + r \gamma_1 )$, $\sigma_{12}:=-p^2 q^2 ( \gamma_0  + 2 \gamma_1 )$;  

\item[(2)] $\sigma_{13}:=-p^2q \int \langle h, \grad  f(z|\theta_1)\rangle \lambda(z)\mathrm d z$,  $\sigma_{23}= pq^2 \int \langle h, \grad  f(z|\theta_1)\rangle \lambda(z)\mathrm d z$;

\item[(3)] $\sigma_{14}:= 2p^2q$, $\sigma_{24}:=-2pq^2$;

\item[(4)] $\sigma_{33}:= q(1-q)\langle h, \mathrm I(\theta)h\rangle$, $\sigma_{44}=p(1-p)$, and $\sigma_{34}=0$.
\end{itemize}
\end{lem}

\begin{proof} To begin with recall that (4) follows from the proof from Lemma \ref{joint}. It remains to prove (1)-(3). \\

\noindent{\it Proof of }(1): Recall~\eqref{R12}. Then, for $j \in \{1, 2\}$, 
\begin{eqnarray}
\cR_j(\sG(\cZ_N))&=&\sqrt N\left(T_j(\sG(\cZ_N))-\frac{N_j^2}{N^2}\right)+o(1),
\label{2sampleG2b}
\end{eqnarray}
where $T_j(\sG(\cZ_N))$ is defined in~\eqref{T12}. Under the bootstrap distribution, the labels of the vertices are independent, and so,
\begin{align}
\Var(\cR_j(\sG(\cZ_N))|\cF) 
=&\frac{N a_j}{|E(\sG(\cZ_N))|}+b_j\frac{2NT_2(\sG(\cZ_N))}{|E(\sG(\cZ_N))|^2},
\label{varRcalculation12}
\end{align}
where 
$$a_j=\frac{N_j^2}{N^2}-\left(\frac{N_j}{N}\right)^4, \quad b_j=\frac{N_j^3}{N^3}-\left(\frac{N_j}{N}\right)^4.$$
Now, since $(\sG, \P_{\theta_1})$ satisfies the variance condition with  parameters $(\gamma_0, \gamma_1)$, \eqref{varRcalculation12} implies $\Var(\cR_1(\sG(\cZ_N))|\cF) \pto p^2((1-p^2) \gamma_0  + r \gamma_1 )=\sigma_{11}$, and  $\Var(\cR_2(\sG(\cZ_N))|\cF) \pto q^2((1-q^2) \gamma_0  + r \gamma_1 )=\sigma_{22}$, as required. Similarly, 
\begin{align}
\Cov(\cR_1(\sG(\cZ_N)), \cR_2(\sG(\cZ_N))|\cF)=&\frac{N a}{|E(\sG(\cZ_N))|}+a\frac{2NT_2(\sG(\cZ_N))}{|E(\sG(\cZ_N))|^2},
\label{corRcalculation}
\end{align}
where $a=-\frac{N_1^2N_2^2}{N^4}$. Therefore, $\Cov(\cR_1(\sG(\cZ_N)), \cR_2(\sG(\cZ_N))|\cF) \pto -p^2 q^2 ( \gamma_0  + 2 \gamma_1 )$. \\

\noindent{\it Proof of }(2): We begin by computing $\Cov( \cR_1(\sG(\cZ_N)), \dot L_N|\mathcal F)$. Let $\nu_{N_1}=\frac{N_1^2}{N^2}$. For $j \in [N]$, define 
\begin{align}
T_{1j}=&\frac{\sqrt N}{|E(\sG(\cZ_N))|}\sum_{i=1, i\ne j} ^{N}\left(\psi_1(c_i, c_j)-\nu_{N_1} \right)\pmb 1\{(Z_i, Z_j)\in E(\sG(\cZ_N))\}\nonumber\\
=&T_{1j}^{(1)}-T_{1j}^{(2)},
\label{Qjn}
\end{align}
where $$T_{1j}^{(1)}=\frac{\sqrt N}{|E(\sG(\cZ_N))|} \sum_{i=1, i\ne j} ^{N}\psi_1(c_i, c_j) \pmb 1\{(Z_i, Z_j)\in E(\sG(\cZ_N)) \},$$ and $$T_{1j}^{(2)}:=\frac{\sqrt N}{|E(\sG(\cZ_N))|} \nu_{N_1} d(Z_j, \sG(\cZ_N)),$$ where $d(Z_j, \sG(\cZ_N))$ is the degree of the vertex $Z_j$ in the graph $\sG(\cZ_N)$. Note that $\cR_1(\sG(\cZ_N))=\frac{1}{2}\sum_{j=1}^N T_{1j}$, and with $\dot L_N=\frac{1}{\sqrt N}\sum_{i=1}^N\langle h, \eta(Z_i, \theta_1)\rangle\pmb 1 \{c_i=2\}$ as defined in~\eqref{etaN}
\begin{eqnarray}
\cR(\sG(\cZ_N)) \dot L_{N}=\frac{1}{2}\Big(\Gamma_1-\Gamma_2+\Gamma_3-\Gamma_4\Big),
\label{eq:cov_termsn}
\end{eqnarray}
where 
\begin{align}
\Gamma_1:=&\frac{1}{\sqrt N}\sum_{j=1}^N T_{1j}^{(1)}\langle h, \eta(Z_j, \theta_1)\rangle  \pmb 1\{c_{j}=2\}, \nonumber\\
\Gamma_2:=&\frac{1}{\sqrt N}\sum_{j=1}^N T_{1j}^{(2)}\langle h, \eta(Z_j, \theta_1)\rangle  \pmb 1\{c_{j}=2\},\nonumber\\
\Gamma_3:=&\frac{1}{\sqrt N}\sum_{1\leq j\ne k\leq N}T_{1j}^{(1)} \langle h, \eta(Z_k, \theta_1)\rangle\pmb 1\{c_{k}=2\},\nonumber\\
\Gamma_4:=&\frac{1}{\sqrt N}\sum_{1\leq j\ne k\leq N}T_{1j}^{(2)} \langle h, \eta(Z_k, \theta_1)\rangle\pmb 1\{c_{k}=2\}.\nonumber
\end{align}
Note that $\Gamma_1=0$. Now, as in \eqref{eq:cov_21}, 
\begin{eqnarray}
\E(\Gamma_2|\cF)=\nu_{N_1}\frac{N_2}{N}\cdot \frac{1}{N}\sum_{j=1}^{N} \langle h, \eta(Z_j, \theta_1)\rangle \lambda(Z_j, \sG(\cZ_N)).
\label{eq:cov_21n}
\end{eqnarray} 
Similarly, as in \eqref{eq:cov_22} and \eqref{eq:cov_24}, we get 
\begin{align}
\E(\Gamma_4|\cF)=\nu_{N_1}\frac{N_2}{N}\left\{\sum_{j=1}^N \langle h, \eta(Z_j, \theta_1)\rangle -\frac{1}{N}\sum_{j=1}^N \langle h, \eta(Z_j, \theta_1)\rangle \lambda(Z_j,\sG(\cZ_N)) \right\},
\label{eq:cov_22n}
\end{align}
and 
\begin{align}
\E(\Gamma_3|\cF)&=\nu_{N_1}\frac{N_2}{N} \left\{ \sum_{j=1}^N \langle h, \eta(Z_j, \theta_1)\rangle  -\frac{2}{N}\sum_{i=1}^N\langle h, \eta(Z_i, \theta_1)\rangle\lambda(Z_i, \sG(\cZ_N))\right\}. \label{eq:cov_24n}
\end{align}
Combining (\ref{eq:cov_22n}), and (\ref{eq:cov_24n}) using \eqref{eq:cov_termsn}, and using the covariance condition \eqref{covundirected} gives $\Cov( \cR_1(\sG(\cZ_N)), \dot L_N|\mathcal F) \pto -p^2q \int \langle h, \grad  f(z|\theta_1)\rangle \lambda(z)\mathrm d z=\sigma_{13}$, as required. The limit of  $\Cov( \cR_1(\sG(\cZ_N)), \dot L_N|\mathcal F)$ can be computed similarly. \\

\noindent{\it Proof of }(3):  Finally, note that 
\begin{align}
\Cov(\mathcal  R_1(\sG(\cZ_N)),\overline B_{N}|\cF)=S_{1}+S_{2}-\frac{N_1^3}{N^2},\label{sigma13n}
\end{align}
where $S_{1}=\frac{1}{2|E(\sG(Z_N))|} \E\left(\sum_{i\ne j}\pmb 1\{(Z_i, Z_j)\in E(\sG(\cZ_N))\} \psi_1(c_i, c_j)\Big| \cF\right)=\frac{N_1^2}{N^2}$ and
\begin{align*}
S_{2}=&\frac{1}{2|E(\sG(Z_N))|} \E\sum_{i_1\ne i_2,j\ne i_1} \pmb 1\{(Z_{i_1}, Z_j)\in E(\sG(\cZ_N))\} 1\{c_{i_1}=1, c_{i_2}=1, c_j=1\}\nonumber\\
=&(N-2)\frac{N_1^3}{N^3}+\frac{N_1^2}{N^2}.
\end{align*}
Substituting the expressions for $S_{1}$ and $S_{2}$ in~\eqref{sigma13n} gives $\Cov(\mathcal  R_1(\sG(\cZ_N)),\overline B_{N}|\cF)\pto 2p^2q=\sigma_{14}$. Similarly, it can be calculated that $\Cov(\mathcal  R_2(\sG(\cZ_N)),\overline B_{N}|\cF)\pto -2pq^2=\sigma_{24}$. 
\end{proof}

If $W_N^1=(\cR_1(\sG(\cZ_N)), \cR_2(\sG(\cZ_N)), \dot\ell_N)^t$, then the distribution of $W_N^1$ conditional on $\{\overline B_N=0\}$ and $\cF$ converges to a $N(0, \Sigma_1)$, where 
\begin{equation*}
\Sigma_1:=\left(\begin{array}{ccc}
\sigma_{11}-\frac{\sigma^2_{14}}{\sigma_{44}}    &  \sigma_{12}-\frac{\sigma_{14}\sigma_{24}}{\sigma_{44}}  &  \sigma_{13}\\
\sigma_{12}-\frac{\sigma_{14}\sigma_{24}}{\sigma_{44}}  &   \sigma_{22}^2 -\frac{\sigma^2_{24}}{\sigma_{44}} & \sigma_{23} \\
\sigma_{13} &  \sigma_{23} & \sigma_{33}
\end{array}
\right)=\left(\begin{array}{ccc}
\lambda_{11}  &  \lambda_{12}  &   \mu_1 \\
\lambda_{12}  &   \lambda_{22} & \mu_2 \\
\mu_1 & \mu_2 & q(1-q)\langle h, \mathrm I(\theta)h\rangle
\end{array}
\right),
\end{equation*}
where the last step uses Lemma \ref{lm:var_limit_N} and the definitions of $\mu$ and $\Lambda$ from Theorem \ref{thm:frnew}.

Finally, note that $\dot L_N=\dot\ell_N+\hat r_N$, where $\hat r_N:=\frac{N_2}{N}\cdot \frac{1}{\sqrt N}\sum_{i=1}^N \langle h, \eta(Z_i, \theta)\rangle$. By the central limit theorem,  $\hat r_N\dto V\sim N(0, q^2 \langle h, \mathrm I(\theta_1) h\rangle)$. Therefore, under the bootstrap distribution, $U_N:=(\cR_1(\sG(\cZ_N)), \cR_2(\sG(\cZ_N)), \dot L_N)^t \dto N(0, \Sigma)$, where 
\begin{equation}\label{eq:Sigma_N}
\Sigma:=\left(\begin{array}{ccc}
\lambda_{11}  &  \lambda_{12}  &   \mu_1 \\
\lambda_{12}  &   \lambda_{22} & \mu_2 \\
\mu_1 & \mu_2 & q\langle h, \mathrm I(\theta)h\rangle
\end{array}
\right).
\end{equation}

Now, let $\cR_N:=(\cR_1(\sG(\cZ_N)), \cR_2(\sG(\cZ_N)))^t$, $\bar \Lambda_N=\diag(\Lambda_N^{-\frac{1}{2}}, 1)$ be a $3 \times 3$ matrix (recall that $\Lambda_N$ is the variance covariance matrix of $\cR_N|\mathcal F$). It is easy to check, from the expression of $\Lambda_N$ in \cite[Lemma 2.1]{chenfriedman} and the proof of Proposition \ref{lm:var_limit_N}, that $\bar \Lambda_N  \pto \bar \Lambda:=\diag(\Lambda^{-\frac{1}{2}}, 1)$ and by the Slutsky's theorem, $\overline\Lambda_N U_N \dto N(0, \overline\Lambda\Sigma \overline\Lambda^t)$, where $\Sigma$ is defined in \eqref{eq:Sigma_N}. Therefore, 
$$\overline\Lambda_N U_N = \begin{pmatrix}
\Lambda_N^{-\frac{1}{2}} \cR_N \\
\dot L_N
\end{pmatrix} \dto N\left(0, \left(\begin{array}{ccc}
\mathrm I  & \Lambda^{-\frac{1}{2}}\mu \\
\mu^t(\Lambda^{-\frac{1}{2}})^t& q\langle h, \mathrm I(\theta)h\rangle
\end{array}
\right)\right),
$$
and by the LeCam's Third Lemma \cite[Corollary 12.3.2]{lr}, \eqref{eq:jointR12} follows. 

\section{Additional Simulations}
\label{sec:knnexperiments}

This section contains additional simulations, showing how the power of the $K$-NN test depends on $K$.

%
\begin{figure*}[h]
\centering
\begin{minipage}[c]{0.495\textwidth}
\centering
\includegraphics[width=2.2in]
    {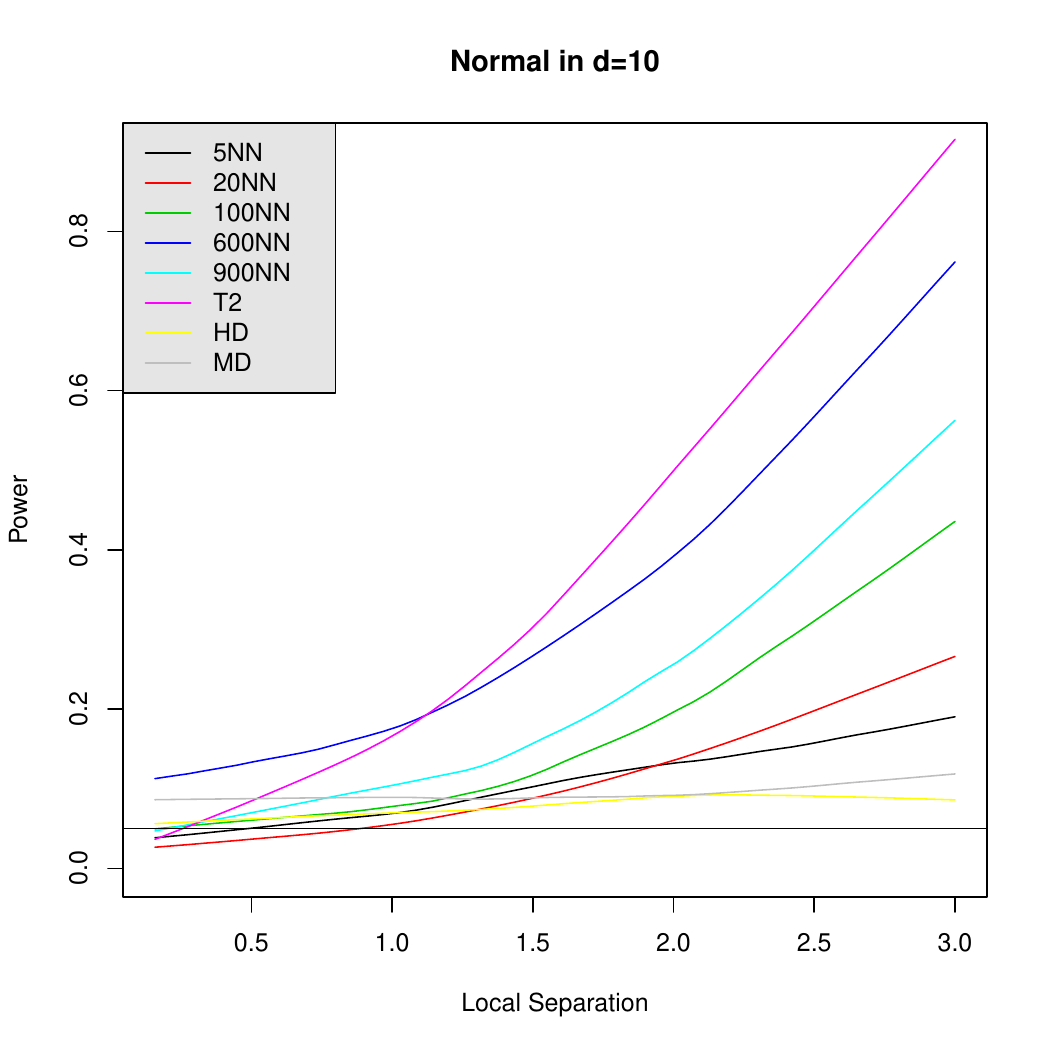}\\
\small{(a)}
\end{minipage}
\begin{minipage}[c]{0.495\textwidth}
\centering
\includegraphics[width=2.2in]
    {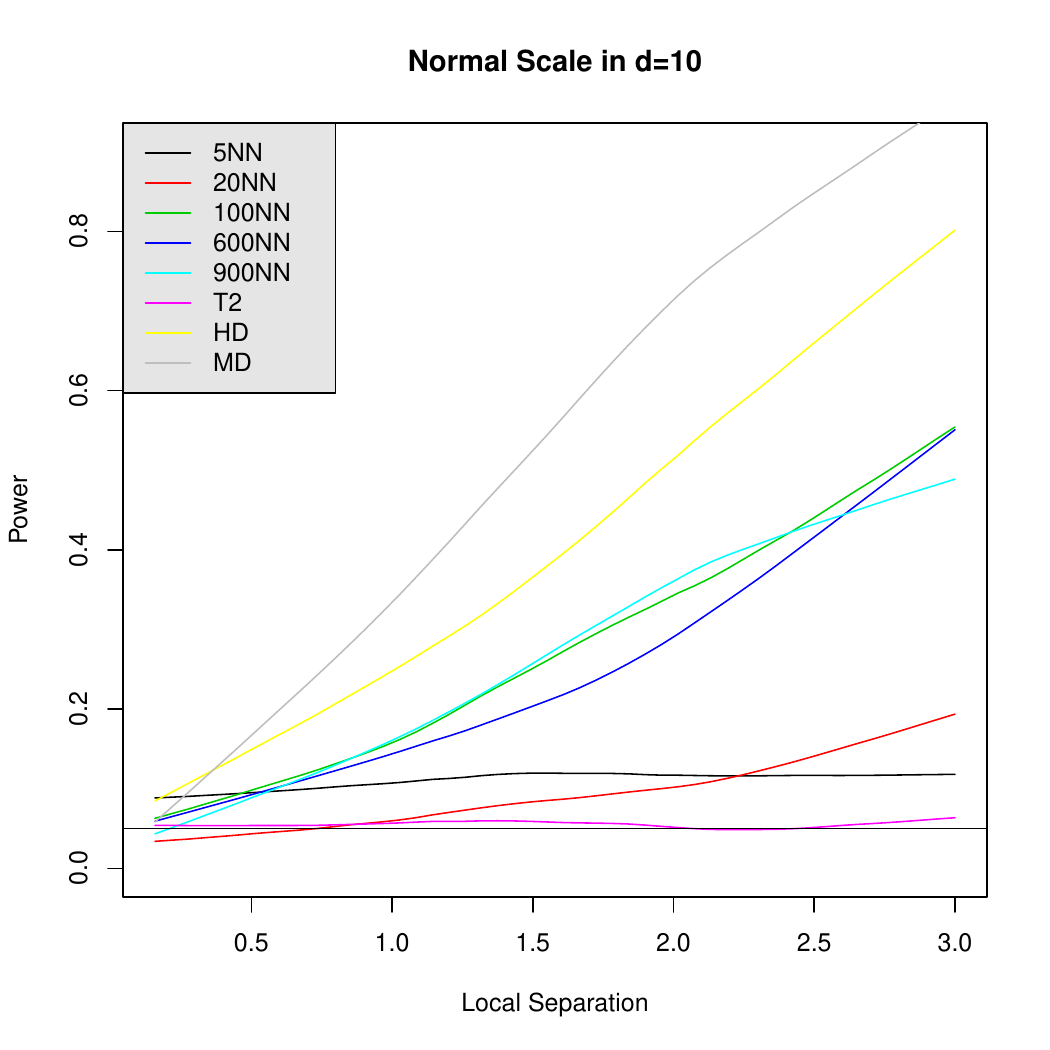}\\
\small{(b)}
\end{minipage}
\caption{\small{(a) Power in the normal location family in dimension 10 where the means differ by $\delta \cdot \bm 1/\sqrt N$, as a function of $\delta$. (b) Power in the spherical normal family in dimension 10 where the standard deviations differ by $\delta/\sqrt N$, as a function of $\delta$.}} 
\label{lognormal_nn}
\end{figure*}

\begin{itemize}

\item[(a)] $\P_{\theta}\sim N(\theta, \mathrm I)$, for $\theta \in \R^d$: Figure \ref{lognormal_nn}(a) shows the empirical power (out of 100 repetitions) of the $K$-NN test for various values of $K$. It also shows the power of the tests based on halfspace depth (HD) and the Mahalanobis depth (MD), and the Hotelling's $T^2$ test,  based on $N_1=1000$ samples from $\P_0$, and $N_2=800$  samples from $\P_{\frac{\delta  \bm 1}{\sqrt N}}$, over a grid of 20 values of $\delta$ in $[0, 3]$. (Note that $N=N_1+N_2=1800$.) As in the lognormal example (Example \ref{knnexample}), the power of the $K$-NN test generally increases with $K$, with the highest power attained when $K=\frac{N}{3}=600$, where it is comparable with the parametric Hotelling's $T^2$ test. Both the HD and the MD tests have low power in this case (as discussed in Remark \ref{ex:depthlocation}).

\item[(b)] $\P_{\sigma}\sim N(0, \sigma^2\mathrm I)$, for $\sigma > 0$:  Figure \ref{lognormal_nn}(b) shows the empirical power (out of 100 repetitions) of the different tests based on $N_1=1000$ samples from $\P_{0}$ and  $N_2=800$  samples from $\P_{1+ \frac{\delta}{\sqrt N}}$, over a grid of 20 values of $\delta$ in $[0, 3]$. Again, the power of the $K$-NN test increases with $K$, stabilizing around $K=100$. Both the HD and the MD tests have good power in this case. Also, as expected, the Hotelling's $T^2$ test has no power in this case. 

\end{itemize}

\end{document}